\documentclass[reqno,11pt]{amsart}

\PassOptionsToPackage{svgnames}{xcolor}
\usepackage[object=vectorian]{pgfornament}
\newcommand{\sectionlinetwo}[2]{
\nointerlineskip\vspace{.5\baselineskip}\hspace{\fill}
{\resizebox{0.5\linewidth}{1.2ex}
{\pgfornament[color = #1]{#2}}}
\hspace{\fill}
\par\nointerlineskip\vspace{.5\baselineskip}}

\usepackage{a4wide}
\usepackage{color}
\usepackage{mathrsfs}
\usepackage{mathtools}
\usepackage{amsmath}
\usepackage{amssymb}
\usepackage{bbm}
\usepackage{esint}
\numberwithin{equation}{section}
\usepackage[colorlinks,citecolor=green,linkcolor=red]{hyperref}

\usepackage[latin1]{inputenc}

\newcommand{\1}{\mathbbm 1}
\newcommand{\N}{\mathbb{N}}
\newcommand{\R}{\mathbb{R}}
\newcommand{\B}{\mathbb{B}}
\newcommand{\sfd}{{\sf d}}
\renewcommand{\d}{{\mathrm d}}
\newcommand{\X}{{\rm X}}
\newcommand{\Y}{{\rm Y}}

\newcommand{\fr}{\penalty-20\null\hfill\(\blacksquare\)}

\newtheorem{theorem}{Theorem}[section]
\newtheorem{corollary}[theorem]{Corollary}
\newtheorem{lemma}[theorem]{Lemma}
\newtheorem{proposition}[theorem]{Proposition}
\newtheorem{definition}[theorem]{Definition}

\newtheorem{remark}[theorem]{Remark}

\title[The metric-valued Lebesgue differentiation theorem in measure spaces]
{The metric-valued Lebesgue differentiation theorem in measure spaces and its applications}

\author{Danka Lu\v{c}i\'{c}}
\address[Danka Lu\v{c}i\'{c}]{Universit\`{a} di Pisa, Dipartimento di Matematica,
Largo Bruno Pontecorvo 5, 56127 Pisa, Italy}
\email{danka.lucic@dm.unipi.it}

\author{Enrico Pasqualetto}
\address[Enrico Pasqualetto]{Scuola Normale Superiore, Piazza dei Cavalieri 7,
56126 Pisa, Italy}
\email{enrico.pasqualetto@sns.it}

\begin{document}
\date{\today}
\keywords{Lebesgue differentiation theorem, von Neumann lifting, measurable Banach bundle, Radon--Nikod\'{y}m property, disintegration of a measure}
\subjclass[2020]{28A15, 28A51, 46G15, 18F15, 46G10, 46B22, 28A50}
\begin{abstract}
We prove a version of the Lebesgue Differentiation Theorem for mappings that are defined
on a measure space and take values into a metric space, with respect to the differentiation
basis induced by a von Neumann lifting. As a consequence, we obtain a lifting theorem for
the space of sections of a measurable Banach bundle and a disintegration theorem for
vector measures whose target is a Banach space with the Radon--Nikod\'{y}m property.
\end{abstract} 
\maketitle
\tableofcontents
\section{Introduction}
In recent years, there have been impressive developments in the study of differential calculus over metric measure spaces.
A breakthrough in the theory was Gigli's work \cite{Gigli18}, where he proposed the language of \emph{\(L^p\)-normed \(L^\infty\)-modules},
with the aim of building a tensor calculus on top of a robust functional-analytic framework. Similar notions -- whose interest goes beyond
the analysis on metric spaces -- were previously considered, \emph{e.g.}, in \cite{HLR91,Guo-2011}. One of the main concepts introduced in \cite{Gigli18}
is the \emph{cotangent module}, which can be regarded as the space of `abstract \(1\)-forms'. Starting from this object, it was possible to
develop an effective vector calculus even in very singular contexts, where a linear and/or smooth underlying structure is missing. However,
differently from the classical Riemannian setting, the elements of the cotangent module do not have (a priori) a `pointwise meaning'.
The problem of providing a fiberwise description of normed modules was addressed in \cite{LP18,DiMarinoLucicPasqualetto21}, where it is
shown that every separable normed module can be represented as the section space of a \emph{measurable Banach bundle}. Previously, similar
notions of a Banach bundle were investigated, \emph{e.g.}, in \cite{Gut93,GutKop00}. Besides its theoretical interest, the relevance of a fiberwise
characterisation of normed modules became evident with the recent paper \cite{GigliLucicPasqualetto22}, where it has been used to describe dual
modules and pullback modules. We point out that the dual of the pullback of the cotangent module comes into play naturally, for instance, when
studying the \emph{differential of a map of bounded deformation} \cite[Proposition 2.4.6]{Gigli18} or the \emph{velocity of a test plan}
\cite[Theorem 2.3.18]{Gigli18}, which are very important tools in analysis on nonsmooth spaces. As of now, no `direct' description
(\emph{i.e.}, without appealing to Banach bundles) of a dual normed module is available.
\medskip

Motivated by this line of research, in the present paper we pursue the following plan:
\begin{itemize}
\item[\(\rm i)\)] Prove a \emph{Lebesgue Differentiation Theorem} for the sections of a measurable Banach bundle.
In this way, we are entitled to consider the \emph{Lebesgue points} of such a section. At the level of generality
we deal with in this paper (which covers, \emph{e.g.}, metric measure spaces that are not doubling), the
differentiation bases we consider are those induced by a \emph{von Neumann lifting} of the given measure,
which exists under mild assumptions.
\item[\(\rm ii)\)] Construct a (weak form of) \emph{lifting} of the section space of a measurable Banach bundle.
To this aim, the Lebesgue Differentiation Theorem in i) has an essential role. This
allows to select `everywhere defined representatives' of sections in a consistent manner.
\item[\(\rm iii)\)] Obtain a \emph{Disintegration Theorem} for vector-valued measures of bounded variation whose
target is a Banach space having the Radon--Nikod\'{y}m property. The proof of the Disintegration Theorem builds
upon the lifting of sections mentioned in ii). Our interest in vector-valued measures mainly comes from the recent
paper \cite{BG22}, where the theory of \emph{local vector measures} has been introduced. The latter provides a unified
framework for many key objects in nonsmooth analysis, such as metric currents, differentials of Sobolev functions,
and distributional gradients of BV functions.
\end{itemize}
In addition to the new contributions mentioned above, we will provide a thorough (and partially revisited) account
of well-established notions, such as differentiation bases and liftings, with two objectives in mind: to adapt
previously known concepts to our setting, and to make the paper (almost) fully self-contained.
Next we discuss in detail the contents of the paper.
\subsection*{Lifting theory}
A fundamental result in Measure Theory is definitely the celebrated von Neumann--Maharam Theorem, which can be stated as follows:
given a complete, finite measure space \((\X,\Sigma,\mu)\), there exists a selection \(\ell\colon\Sigma\to\Sigma\) of \(\mu\)-a.e.\ representatives
which preserves the empty set, the whole space, finite intersections, and finite unions. An equivalent formulation
can be given at the level of functions: there exists a linear and multiplicative operator
\(\ell\colon L^\infty(\mu)\to\mathcal L^\infty(\Sigma)\) that selects
\(\mu\)-a.e.\ representatives; in other words, \(\ell(f)=f\) holds \(\mu\)-a.e.\ for
every \(f\in L^\infty(\mu)\), and \(\ell\) preserves constant functions, finite sums,
and finite products. Here, \(\mathcal L^\infty(\Sigma)\) stands for the space of all
real-valued, \(\Sigma\)-measurable functions on \(\X\), while \(L^\infty(\mu)\) denotes
its quotient up to \(\mu\)-a.e.\ equality. In either cases, the map \(\ell\) is called a
\textbf{von Neumann lifting}, or just a \textbf{lifting}, of the measure \(\mu\).
The question about the existence of liftings (in the setting of the Euclidean space
endowed with the Lebesgue measure) was originally raised by Haar in 1931
and soon after solved by von Neumann \cite{vonNeumann31}. Afterwards, this
result has been generalised by Maharam \cite{Maharam58} and Ionescu Tulcea--Ionescu
Tulcea \cite{IonescuTulceaIonescuTulcea61} to more general measure spaces. On the one hand,
the finiteness assumption on \(\mu\) can be relaxed: the case of \(\sigma\)-finite measures can be
easily reduced to the finite ones, but the result can be also proved in the larger class of strictly
localisable measures. On the other hand, the role of the completeness
assumption on \(\mu\) is more significant: it is still an open problem whether all
non-complete probability measures admit a lifting. As observed by Shelah \cite{Shelah83},
it is consistent with ZFC that no Borel lifting of the Lebesgue measure on \([0,1]\) exists.
However, under the continuum hypothesis, von Neumann and Stone proved in \cite{vonNeumannStone35}
that a Borel lifting of the Lebesgue measure on \([0,1]\) exists, a result later generalised by
Mokobodzki \cite{Mokobodzki75}. For a more detailed discussion about these topics, we refer to
the books \cite{StraussMacherasMusial02,IonescuTulceaIonescuTulcea69,Fremlin3,Bogachev07}.
\medskip

One of the possible strategies which have been used to prove existence of liftings
is the following: first, to provide a \textbf{lower density} \(\phi\colon\Sigma\to\Sigma\)
of the measure \(\mu\) under consideration, \emph{i.e.}, a map \(\phi\) verifying the
lifting axioms with the exception of the finite unions preservation; then, to `enlarge'
\(\phi\) to a lifting \(\ell\), \emph{i.e.}, to provide a lifting \(\ell\) of \(\mu\)
satisfying \(\phi(E)\subset\ell(E)\) for every set \(E\in\Sigma\). In many interesting
cases, an explicit lower density can be constructed. For instance, whenever a Lebesgue
Density Theorem is in force, one obtains a lower density by just associating to each
measurable set the family of its density points; this idea dates back to the original
paper \cite{vonNeumann31} by von Neumann. On an arbitrary measure space, to show existence
of a lower density is a more delicate issue, which requires a transfinite induction
argument to be employed (see \cite{GrafvonWezsacker76}). Once a lower density is at disposal,
a standard argument (using, \emph{e.g.}, Zorn Lemma \cite[341J]{Fremlin3} or
the Ultrafilter Lemma \cite{Traynor74}) yields the existence of a lifting.
The completeness of \(\mu\) plays a role exclusively in the passage from lower
densities to liftings. It is also worth pointing out that, excluding the very exceptional
framework of Stone spaces, no explicit lifting can be constructed, even when the
lower density is somewhat canonical (\emph{e.g.}, if the Lebesgue Density Theorem
holds with respect to some reasonable differentiation basis). A natural question, though,
is whether an additional structure on the measure space reflects on the possibility to
require further regularity of the lifting. In this regard, a significant example is that of
\textbf{strong liftings} \cite{IonescuTulceaIonescuTulcea67}: if \(\X\) is a complete, separable metric space and \(\Sigma\) is
the completion of the Borel \(\sigma\)-algebra of \(\X\) under \(\mu\), then one can find a
lifting \(\ell\) which is strong, meaning that \(\ell(f)\) and \(f\) agree everywhere on
the support of \(\mu\) whenever \(f\colon\X\to\R\) is bounded continuous.
The metrisability assumption cannot be dropped, as a counterexample built in \cite{Losert79} shows.
We point out that strong liftings that are not Borel liftings exist \cite{Johnson80}.
\medskip

Our presentation of the lifting theory mainly follows along the lines of \cite{Fremlin3} and \cite{Bogachev07}. In Section \ref{ss:lifting_von_Neumann},
we recall the definition of a lifting, the statement of von Neumann--Maharam Theorem, and how to pass from liftings of sets to liftings of functions
(and vice versa). In Section \ref{ss:strong_lift}, we introduce strong liftings. In Appendix \ref{app:Leb_to_lift}, we discuss lower densities, how to
pass from a lower density to a lifting, and how the validity of the Lebesgue Density Theorem induces a lower density. The latter fact provides a relatively
easy proof of the existence of von Neumann liftings on a vast class of measure spaces (see Theorem \ref{thm:exist_lift_separable}) and of strong liftings
on (complete and separable) metric measure spaces (see Corollary \ref{cor:strong_lifting_on_mms}).
\subsection*{Lebesgue Differentiation Theorems in measure spaces}
The classical Lebesgue Differentiation Theorem, which dates back to \cite{Lebesgue10},
states that for \(\mathcal L^n\)-a.e.\ point \(x\in\R^n\), the value of an integrable
function \(f\colon\R^n\to\R\) at \(x\) coincides with the limit of its infinitesimal
averages around \(x\). In fact, a stronger statement holds: setting
\(\fint_E f\,\d\mathcal L^n\coloneqq\mathcal L^n(E)^{-1}\int_E f\,\d\mathcal L^n\), one has
\[
\lim_{r\searrow 0}\fint_{B_r(x)}\big|f(y)-f(x)\big|\,\d\mathcal L^n(y)=0,
\quad\text{ for }\mathcal L^n\text{-a.e.\ }x\in\R^n.
\]
Later on, this fundamental result has been generalised in many different ways.
For instance, it is possible to consider integrable functions defined on domains different
from the Euclidean ones, and to take averages with respect to other classes of `shrinking'
sets than balls. The Lebesgue Differentiation Theorem holds, exactly in the same form as
above, on (asymptotically) doubling metric measure space \cite{Heinonen01}, but often it
can be formulated also in terms of other differentiation bases. By a \textbf{differentiation basis}
on an arbitrary measure space \((\X,\Sigma,\mu)\) we mean a
family \(\mathcal I\) of measurable sets, having positive and finite measure,
which are directed by downward inclusion. Namely, given a point \(x\in\X\) and two sets
\(I,I'\in\mathcal I\) containing \(x\), there exists \(J\in\mathcal I\) such that
\(J\subset I\cap I'\). The family of elements of \(\mathcal I\) containing \(x\) is
typically denoted by \(\mathcal I_x\). Differentiation bases come with a natural notion
of convergence: we say that a map \(\Phi\colon\mathcal I\to\Y\), the target \(\Y\) being
a topological space, has \textbf{\(\mathcal I\)-limit} \(y\in\Y\) at \(x\in\X\) provided
\[
\forall U\text{ neighbourhood of }y\text{ in }\Y,\quad\exists I\in\mathcal I_x:
\quad\forall J\in\mathcal I_x\text{ with }J\subset I,\quad\Phi(J)\in U.
\]
For brevity, we will just write \(\lim_{I\Rightarrow x}\Phi(I)=y\). With this terminology,
a generalised form of Lebesgue Differentiation Theorem can be formulated as follows:
for \(f\colon\X\to\R\) integrable,
\[
\lim_{I\Rightarrow x}\fint_I\big|f(y)-f(x)\big|\,\d\mu(y)=0,
\quad\text{ for }\mu\text{-a.e.\ }x\in\X.
\]
A well-studied problem in the literature is to understand for which spaces and differentiation bases
the above version of Lebesgue Differentiation Theorem holds; for a thorough account of this theory,
we refer to \cite{Bruckner71,BrucknerBrucknerThomson97,BrucknerBrucknerThomson08} and the references therein.
In this paper, we will focus our attention on a special class of differentiation bases:
those induced by a von Neumann lifting. Given a complete, finite measure space
\((\X,\Sigma,\mu)\) and a lifting \(\ell\) of \(\mu\), we denote by \(\mathcal I^\ell\) the
family of all sets of the form \(\ell(E)\), where \(E\in\Sigma\) satisfies \(\mu(E)>0\). 
It readily follows from the lifting axioms that \(\mathcal I^\ell\) defines a
differentiation basis on \((\X,\Sigma,\mu)\). Furthermore, it is known that
\(\mathcal I^\ell\) verifies the Lebesgue Differentiation Theorem as stated above
(see \cite{Kolzow68}). On the one hand, this type of differentiation basis has the advantage
of being available in the great generality of measure spaces, without the need for
an underlying topological or metric structure. On the other hand, the procedure of finding
liftings (and thus a fortiori the associated differentiation bases) is highly
non-constructive. In other words, albeit very interesting from a theoretical perspective,
differentiation bases induced by a lifting are maybe not very useful from a computational point
of view. Nonetheless, this issue is irrelevant for the applications that we will describe
in the present paper, since liftings will be used only as an intermediate technical tool during the
proofs.
\medskip

Our main result is Theorem \ref{thm:Lebesgue_diff}, which is a generalisation of the
above-discussed Lebesgue Differentiation Theorem in measure spaces to more general
targets than the real line, namely, to metric-valued maps. More specifically, we will
prove that any measurable map \(\varphi\colon\X\to\Y\), where \((\X,\Sigma,\mu)\) is a
complete measure space and \((\Y,\sfd)\) is a separable metric space, satisfies
\[
\lim_{I\Rightarrow x}\fint_I\sfd\big(\varphi(y),\varphi(x)\big)\,\d\mu(y)=0,
\quad\text{ for }\mu\text{-a.e.\ }x\in\X,
\]
where the limit is taken with respect to the differentiation basis \(\mathcal I^\ell\)
induced by an arbitrary lifting \(\ell\) of \(\mu\). As an immediate corollary, we deduce
that if \(\B\) is a Banach space and \(v\colon\X\to\B\) is a strongly measurable map,
then it holds that
\[
\lim_{I\Rightarrow x}\fint_I v(y)\,\d\mu(y)=v(x),\quad\text{ for }\mu\text{-a.e.\ }x\in\X,
\]
where the integral is intended in the sense of Bochner. The proof argument we provide
is rather elementary and much shorter than other ones we found in the literature
(for the real-valued Lebesgue Differentiation Theorem in measure spaces), such as the
one in \cite{BrucknerBrucknerThomson08}. Our strategy is to show that every measurable map
\(\varphi\colon\X\to\Y\) as above is `approximately continuous' at almost every point of \(\X\)
with respect to \(\mathcal I^\ell\), in the sense expressed by Lemma \ref{lem:pre_Leb}.
Once this property is established, Theorem \ref{thm:Lebesgue_diff} easily follows.
This line of thought, based on an approximate continuity argument, is similar in spirit to
the work done in \cite{BliedtnerLoeb00,LoebTalvila04}.
\medskip

An immediate consequence of Theorem \ref{thm:Lebesgue_diff} is Corollary
\ref{cor:diff_measures_RNP}, which states the following. Consider a \(\B\)-valued
measure \(\Omega\) on \((\X,\Sigma)\) with bounded variation, where \(\B\) is a
Banach space having the Radon--Nikod\'{y}m property (RNP for brevity),
and suppose that \(\Omega\ll\mu\). Then the Radon--Nikod\'{y}m derivative
\(\frac{\d\Omega}{\d\mu}\colon\X\to\B\) can be expressed as a `genuine' derivative as
\[
\frac{\d\Omega}{\d\mu}(x)=\lim_{I\Rightarrow x}\frac{\Omega(I)}{\mu(I)}\in\B,
\quad\text{ for }\mu\text{-a.e.\ }x\in\X,
\]
where the limit is taken with respect to the differentiation
basis \(\mathcal I^\ell\).
\medskip

In Sections \ref{ss:diff_basis}, \ref{ss:Bochner_int}, and \ref{ss:RNP} we remind
the basic notions concerning differentiation bases, Bochner integration, and the
Radon--Nikod\'{y}m property, respectively. In Section \ref{s:LDT} we prove our
version of Lebesgue Differentiation Theorem in measure spaces for metric-valued maps,
as well as its direct consequences.
In Appendix \ref{app:approx_cont} we study more in details the concept
of approximate continuity with respect to a differentiation basis,
which is used in Lemma \ref{lem:pre_Leb}.
\subsection*{Banach bundles and liftings of sections}
The main motivation behind our study of the metric-valued Lebesgue Differentiation
Theorem and its applications comes from the theory of measurable Banach bundles and
\(L^p\)-normed \(L^\infty\)-modules. Let us introduce the relevant terminology.
Fix a measure space \((\X,\Sigma,\mu)\) and an ambient Banach space \(\B\).
Following \cite{DiMarinoLucicPasqualetto21}, we say that a multivalued map
\({\mathbf E}\colon\X\twoheadrightarrow\B\) is a \textbf{measurable Banach \(\B\)-bundle}
over \((\X,\Sigma)\) provided it is weakly measurable (in an appropriate sense)
and each \({\mathbf E}(x)\) is a linear subspace of \(\B\). We denote by
\(\bar\Gamma_p(\mu;{\mathbf E})\) the space of all \(p\)-integrable sections of
\(\mathbf E\), while \(\Gamma_p(\mu;{\mathbf E})\) stands for its quotient up
to \(\mu\)-a.e.\ equality. Our interest towards measurable Banach \(\B\)-bundles
is, in turn, due to the fact that the space of sections \(\Gamma_p(\mu;{\mathbf E})\)
is an \textbf{\(L^p(\mu)\)-normed \(L^\infty(\mu)\)-module} in the sense of Gigli
\cite{Gigli18,Gigli17}. One of the primary outcomes of
\cite{DiMarinoLucicPasqualetto21} is that actually every separable \(L^p(\mu)\)-normed
\(L^\infty(\mu)\)-module can be written as a space of sections of the form
\(\Gamma_p(\mu;{\mathbf E})\) for some Banach \(\B\)-bundle \(\mathbf E\). This result is
relevant if read in conjuction with the fact that, under mild assumptions, the cotangent
module associated with a metric measure space is separable
(cf.\ with \cite[Appendix B]{DiMarinoLucicPasqualetto21}).
In reality, the language of normed modules was introduced precisely with the aim of
building a solid differential calculus for metric measure spaces, and cotangent modules
provide a convenient concept of a covector field in this (possibly very singular)
framework. Ultimately, the results we will discuss next say that it is possible to
talk about the Lebesgue points of a covector field on a metric measure space, whose
pointwise meaning is inherently encoded in the Banach bundle machinery.
We also mention that the analysis of measurable Banach bundles covers the well-studied
theory of Lebesgue--Bochner spaces, which correspond to considering those Banach bundles
whose fibers are constant.
\medskip

In Theorem \ref{thm:precise_repr} we introduce the precise representative map for
the \(L^p\)-sections of a given measurable Banach \(\B\)-bundle \(\mathbf E\) over
\((\X,\Sigma)\): as a corollary of Theorem \ref{thm:Lebesgue_diff}, one has that
any given element \(v\in\Gamma_p(\mu;{\mathbf E})\) admits a unique representative
\(\hat v\in\bar\Gamma_p(\mu;{\mathbf E})\) satisfying
\[
\hat v(x)=\lim_{I\Rightarrow x}\fint_I v(y)\,\d\mu(y)\in{\mathbf E}(x),
\]
for any \(x\in\X\) such that the limit in the right-hand side exists,
and \(\hat v(x)\coloneqq 0_{\mathbf E(x)}\) elsewhere. Any such
point (thus, \(\mu\)-a.e.\ point) is called a \textbf{Lebesgue point}
of \(v\). Observe that \(\hat v\) is independent of the chosen representative of \(v\).
Building upon Theorem \ref{thm:precise_repr}, we obtain (see Theorem \ref{thm:lift_sects})
a lifting map at the level of the sections of a Banach bundle, which generalises the
von Neumann liftings of functions in a natural way. More precisely, we show that
a von Neumann lifting \(\ell\) of \(\mu\) induces a linear choice of a representative
\(\ell\colon\Gamma_\infty(\mu;{\mathbf E})\to\bar\Gamma_\infty(\mu;\B'')\) such that
\[\begin{split}
\ell(f\cdot v)(x)=\ell(f)(x)\,\ell(v)(x),&\quad\text{ for every }v\in\Gamma_\infty
(\mu;{\mathbf E}),\,f\in L^\infty(\mu),\text{ and }x\in\X,\\
\big\|\ell(v)(x)\big\|_{\B''}\leq\ell\big(\|v(\cdot)\|_{{\mathbf E}(\cdot)}\big)(x),&
\quad\text{ for every }v\in\Gamma_\infty(\mu;{\mathbf E})\text{ and }x\in\X.
\end{split}\]
The first condition expresses a compatibility with the lifting \(\ell\) at the level
of functions, whence no ambiguity should arise. About the second condition, one would
probably expect it to be an identity rather than an inequality. However, with the
exception of some special situations, this is not only unlikely but in fact truly
impossible. Roughly speaking, the lifting of sections can verify
\(\big\|\ell(v)(\cdot)\big\|_{\B''}=\ell\big(\|v(\cdot)\|_{{\mathbf E}(\cdot)}\big)\)
only for those elements \(v\) that have compact (essential) image, as we will see
in Proposition \ref{prop:improved_lift_sects}. In the case of Lebesgue--Bochner spaces,
this kind of behaviour is consistent with the fact that a norm-preserving
lifting \(\ell\colon L^\infty(\mu;\B)\to\mathcal L^\infty(\mu;\B)\) can exist if and only if
\(\B\) is finite-dimensional (or \(\mu\) is purely atomic). This surprising phenomenon
was known, we will report it in Theorem \ref{thm:lift_Leb-Bochn}.
Due to this reason, in the literature we only found notions of liftings for maps
taking values into a locally compact space (see, \emph{e.g.}, \cite{IonescuTulceaIonescuTulcea69b,IonescuTulceaIonescuTulcea69}).
To the best of our knowledge, the weaker form of lifting we provide in Theorem \ref{thm:lift_sects}
for arbitrary target (which contracts, but does not preserve, the norm) was not
previously investigated. However, it has some interesting applications, such as a
disintegration theorem for RNP-valued measures (see Section \ref{ss:disint}).
Coming back to our lifting of sections of \(\mathbf E\), we point out that its target is the
larger space \(\bar\Gamma_\infty(\mu;\B'')\), where \(\B''\) stands for the bidual
of \(\B\). In some cases of interest (for instance, if \(\B\) has a predual,
cf.\ with Remark \ref{rmk:target_B}), the target of \(\ell\) can be actually required to be
\(\bar\Gamma_\infty(\mu;\B)\), but we do believe this is not possible in full generality.
\medskip

Let us spend a few words to describe the strategy we will adopt to achieve
Theorem \ref{thm:lift_sects}, which is partially inspired by von Neumann's original
proof of existence of liftings in the Euclidean space. Fix a complete, finite measure
space \((\X,\Sigma,\mu)\), a lifting \(\ell\) of \(\mu\), and a measurable Banach
\(\B\)-bundle \(\mathbf E\) on \((\X,\Sigma)\). Given any \(x\in\X\), the differentiation
basis \(\mathcal I^\ell\) induces a filter of `generalised neighbourhoods' of \(x\), which
we can extend to an ultrafilter \(\omega_x\) by applying the Ultrafilter Lemma. The key
feature of the ultrafilter \(\omega_x\) is that \(\lim_{I\Rightarrow x}\Phi(I)=y\) implies
that the ultralimit of \(\Phi\) with respect to \(\omega_x\) coincides with \(y\). The
construction of \(\omega_x\) will be committed to Lemma \ref{lem:construct_ultrafilter}.
At this point, the lifting map \(\ell\colon\Gamma_\infty(\mu;{\mathbf E})\to
\bar\Gamma_\infty(\mu;\B'')\) can be defined as follows: given any element
\(v\in\Gamma_\infty(\mu;{\mathbf E})\) and any point \(x\in\X\), we set
\[
\ell(v)(x)\coloneqq\omega_x\text{-}\lim\fint_I v(y)\,\d\mu(y),
\]
where we are tacitly identifying \(\B\) with its canonical image into \(\B''\).
The ultralimit is taken with respect to the weak\(^*\) topology of \(\B''\), and
its existence is ensured by Banach--Alaoglu Theorem; this explains why the lifting
of sections takes values into \(\B''\). Theorem \ref{thm:Lebesgue_diff} plays a
fundamental role here, as it ensures that \(\ell(v)\) is a representative of \(v\).
The linearity of \(\ell\) stems from general properties of ultralimits. The multiplicative
identity \(\ell(f\cdot v)=\ell(f)\,\ell(v)\) boils down to the almost tautological
(yet interesting) fact that for any \(f\in L^\infty(\mu)\) one has
\[
\ell(f)(x)=\omega_x\text{-}\lim\fint_I f(y)\,\d\mu(y),
\quad\text{ for \emph{every} point }x\in\X,
\]
which will be proved in Proposition \ref{prop:compat_lift}. Finally, the validity of
\(\big\|\ell(v)(\cdot)\big\|_{\B''}\leq\ell\big(\|v(\cdot)\|_{{\mathbf E}(\cdot)}\big)\)
follows from the weak\(^*\) lower semicontinuity of the norm \(\|\cdot\|_{\B''}\).
\medskip

In Sections \ref{ss:MBB} and \ref{ss:ultralimits} we recall the theory of measurable
Banach bundles and the most important concepts related to ultrafilters/ultralimits,
respectively. In Section \ref{s:lift_sections} we prove the existence theorem for
liftings of sections, as well as all the preparatory results that we need in order to
achieve it. Appendix \ref{app:lift_Leb-Boch} is dedicated to a slight revision of the
previously-known lifting theory of Lebesgue--Bochner spaces, which constitutes a
particular case of our study.
\subsection*{Disintegration of vector measures}
One of the most striking applications of the theory of von Neumann liftings (or, more
specifically, of strong liftings) is the Disintegration Theorem for measures. By suitably
adapting a known strategy and employing Theorem \ref{thm:lift_sects}, we obtain in Section
\ref{s:disint} a rather general disintegration result for RNP-valued measures on Polish spaces,
which states the following. Let \(\varphi\colon\X\to\Y\) be a Borel map between Polish spaces,
\(\B\) a Banach space having the Radon--Nikod\'{y}m property, and \(\Omega\) a
\(\B\)-valued measure on \(\X\) of bounded variation. Then there exists an (essentially
unique) disintegration \(\{\Omega_y\}_{y\in\Y}\) of \(\Omega\) along the map \(\varphi\).
Namely, \(\Omega_y\) is a \(\B\)-valued measure on \(\X\) of bounded variation with
\(\|\Omega_y\|_\B(\X)=1\) and concentrated on \(\varphi^{-1}(\{y\})\) for
\(\varphi_\#\|\Omega\|_\B\)-a.e.\ \(y\in\Y\); moreover, the map
\(\Y\ni y\mapsto\int f\,\d\Omega_y\in\B\) is strongly measurable and satisfies
\[
\int f\,\d\Omega=\int\!\!\!\int f\,\d\Omega_y\,\d(\varphi_\#\|\Omega\|_\B)(y),
\]
whenever \(f\in C_b(\X)\) is given. The existence of \(\{\Omega_y\}_{y\in\Y}\) will be proved in Theorem \ref{thm:disint_vect_meas}.
\medskip

We now briefly comment on the heuristic behind the proof of the Disintegration Theorem. For
the sake of exposition, suppose to know that a disintegration \(\{\Omega_y\}_{y\in\Y}\)
of \(\Omega\) along \(\varphi\) exists, and that \(\X\) is compact. Fix a strong lifting \(\ell\)
of \(\nu\coloneqq\varphi_\#\|\Omega\|_\B\). Then for every \(f\in C(\X)\)
we can formally compute
\[
\int\!\!\!\int f\,\d\Omega_y\,\d\nu(y)=\int f\,\d\Omega=\varphi_\#(f\Omega)(\Y)=
\int\ell\bigg(\frac{\d\varphi_\#(f\Omega)}{\d\nu}\bigg)(y)\,\d\nu(y),
\]
where in the last term the lifting
\(\ell\colon L^\infty(\nu;\B)\to\mathcal L^\infty(\nu;\B'')\) provided by Theorem \ref{thm:lift_sects} appears.
Technically speaking there is an inconsistency, as \(\ell\big(\frac{\d\varphi_\#(f\Omega)}{\d\nu}\big)(y)\) belongs to \(\B''\).
However, ignoring this issue for the moment, the previous formula is suggesting that
\[
\int f\,\d\Omega_y=\ell\bigg(\frac{\d\varphi_\#(f\Omega)}{\d\nu}\bigg)(y),\quad\text{ for every }f\in C(\X).
\]
We then take this as a definition of \(\Omega_y\). By using Theorem \ref{thm:lift_sects}, one can prove that \(\Omega_y\)
takes actually values in \(\B\) for \(\nu\)-a.e.\ \(y\in\Y\) and gives a disintegration of \(\Omega\) along \(\varphi\).
The fact of considering \(\ell\big(\frac{\d\varphi_\#(f\Omega)}{\d\nu}\big)(y)\) rather than \(\frac{\d\varphi_\#(f\Omega)}{\d\nu}(y)\)
is fundamental. The problem is already at the level of the definition: given any \(f\in C(\X)\), the vector
\(\frac{\d\varphi_\#(f\Omega)}{\d\nu}(y)\in\B\) is defined for \(\nu\)-a.e.\ \(y\in\Y\), but it is not clear
(unless we invoke the lifting theory) how to invert the quantifiers. In other words, how to show that for \(\nu\)-a.e.\ \(y\in\Y\)
the vector \(\frac{\d\varphi_\#(f\Omega)}{\d\nu}(y)\) is well-defined for every function \(f\in C(\X)\).
We also point out that the inequality \(\big\|\ell(v)(\cdot)\big\|_{\B''}\leq\ell\big(\|v(\cdot)\|_\B\big)\)
appearing in Theorem \ref{thm:lift_sects} is still sufficient to achieve the Disintegration Theorem, so that
no finite-dimensionality on the target space \(\B\) is required.
\subsection*{About the Axiom of Choice}
We conclude the Introduction by concisely commenting on the passages where the Axiom
of Choice, or some of its weaker forms, will be used throughout:
\begin{itemize}
\item The existence of von Neumann liftings strongly relies upon the Axiom of Choice.
\item In Lemma \ref{lem:construct_ultrafilter}, we employ the Ultrafilter Lemma.
\item During the proof of Theorem \ref{thm:lift_sects}, we use Banach--Alaoglu Theorem.
\end{itemize}
It is evident that, when working at this level of generality, the usage of the Axiom
of Choice is unavoidable. Nevertheless, we believe that in many cases of interest
(\emph{e.g.}, on metric measure spaces or when working with separable Banach bundles)
the use of liftings (and more generally of the Axiom of Choice) is not necessary,
as observed for example in \cite[Remark 4.14]{DiMarinoLucicPasqualetto21}.
\section{Preliminaries}
Let us begin by fixing some general terminology, which will be
used throughout the whole paper.
Given a topological space \((\Y,\tau)\) and a point \(y\in\Y\), we
denote by \(\mathcal N_\Y(y)\subset\tau\) the family of all open
neighbourhoods of \(y\) in \((\Y,\tau)\).
Given a measure space \((\X,\Sigma,\mu)\) and an exponent
\(p\in[1,\infty)\), we denote by \(\mathcal L^p(\mu)\)
the space of all \(p\)-integrable functions \(f\colon\X\to\R\),
\emph{i.e.}, of all measurable functions \(f\colon\X\to\R\) whose
associated quantity \(\|f\|_{\mathcal L^p(\mu)}\) is finite, where
\[
\|f\|_{\mathcal L^p(\mu)}\coloneqq
\bigg(\int|f|^p\,\d\mu\bigg)^{1/p}.
\]
We denote by \(\mathcal L^\infty(\Sigma)\) the space of all
bounded, measurable functions \(f\colon\X\to\R\) and we define
\[
\|f\|_{\mathcal L^\infty(\Sigma)}\coloneqq\sup_{x\in\X}|f(x)|.
\]
It holds that \(\mathcal L^p(\mu)\) is a complete seminormed space
for every \(p\in[1,\infty)\), while \(\mathcal L^\infty(\Sigma)\) is
a Banach space. For any \(p\in[1,\infty)\), we introduce an
equivalence relation \(\sim_\mu\) on \(\mathcal L^p(\mu)\): given any
\(f,g\in\mathcal L^p(\mu)\), we declare that \(f\sim_\mu g\) if and
only if \(f=g\) holds \(\mu\)-a.e.\ on \(\X\). We denote by
\(L^p(\mu)\) the quotient space \(\mathcal L^p(\mu)/\sim_\mu\),
which is called the \textbf{\(p\)-Lebesgue space}, and by
\(\pi_\mu\colon\mathcal L^p(\mu)\to L^p(\mu)\) the canonical
projection map. Observe that \(L^p(\mu)\) is a Banach space.
Similarly, we introduce an equivalence relation \(\sim_\mu\) on
\(\mathcal L^\infty(\Sigma)\) as above, we denote by \(L^\infty(\mu)\)
the quotient space \(\mathcal L^\infty(\Sigma)/\sim_\mu\), and we call
\(\pi_\mu\colon\mathcal L^\infty(\Sigma)\to L^\infty(\mu)\) the canonical
projection map. It holds that \(L^\infty(\mu)\) is a Banach space if endowed
with the quotient norm, \emph{i.e.}, with
\[
\|f\|_{L^\infty(\mu)}\coloneqq\inf\big\{\|\bar f\|_{\mathcal L^\infty(\Sigma)}
\;\big|\;\bar f\in\mathcal L^\infty(\Sigma),\,\pi_\mu(\bar f)=f\big\},\quad
\text{ for every }f\in L^\infty(\mu).
\]

Given a Banach space \(\B\), we denote by \(\B'\) its continuous dual space and by \(J_\B\colon\B\hookrightarrow\B''\)
the \textbf{James' embedding} operator into the bidual space, which is defined as
\[
J_\B(v)(\omega)\coloneqq\omega(v),\quad\text{ for every }v\in\B\text{ and }\omega\in\B'.
\]
Recall that \(J_\B\) is a linear isometric map and that it is surjective if and only if \(\B\) is reflexive.
\subsection{Liftings in the sense of von Neumann}\label{ss:lifting_von_Neumann}
Let us recall the basic notions in the lifting theory of von Neumann. Our presentation follows along the one in Fremlin's book \cite{Fremlin3}.
\medskip

A \textbf{Boolean ring} is a ring \((\mathscr R,+,\cdot)\)
such that \(a^2=a\) for every \(a\in\mathscr R\). In particular, it
holds that \(a=-a\) and \(ab=ba\) for every \(a,b\in\mathscr R\).
A \textbf{Boolean algebra} is a Boolean ring with a multiplicative
identity \(1_{\mathscr R}\). Given any two Boolean algebras \(\mathscr R\)
and \(\mathscr R'\), we say that a map \(\phi\colon\mathscr R\to\mathscr R'\)
is a \textbf{Boolean homomorphism} provided it is a ring homomorphism
such that \(\phi(1_{\mathscr R})=1_{\mathscr R'}\). Given a set
\(\X\neq\varnothing\) and an algebra of sets \(\Sigma\subset 2^\X\),
the triplet \((\Sigma,\Delta,\cap)\) is a Boolean algebra with zero
\(\varnothing\) and identity \(\X\). We recall the notion of a von Neumann lifting:
\begin{definition}[von Neumann lifting]\label{def:lifting}
Let \((\X,\Sigma,\mu)\) be a measure space with \(\mu\neq 0\). Then a map \(\ell\colon\Sigma\to\Sigma\)
is called a \textbf{von Neumann lifting} of \(\mu\) if it satisfies the following properties:
\begin{itemize}
\item[\(\rm i)\)] \(\ell\) is a Boolean homomorphism, namely, it holds that
\[
\ell(\varnothing)=\varnothing,\quad\ell(\X)=\X,\quad\ell(E\cap F)=\ell(E)\cap\ell(F),\quad\ell(E\cup F)=\ell(E)\cup\ell(F),
\]
for every \(E,F\in\Sigma\).
\item[\(\rm ii)\)] \(\ell(E)=\ell(F)\) for every \(E,F\in\Sigma\) such that \(\mu(E\Delta F)=0\).
\item[\(\rm iii)\)] \(\mu\big(E\Delta\ell(E)\big)=0\) for every \(E\in\Sigma\).
\end{itemize}
\end{definition}

It easily follows from the lifting axioms that each von Neumann lifting \(\ell\) satisfies
\[
\ell\big(\ell(E)\big)=\ell(E),\quad\text{ for every }E\in\Sigma.
\]
Liftings exist in high generality. For our purposes, the following existence result is sufficient:
\begin{theorem}[von Neumann--Maharam]\label{thm:von_Neumann_Maharam}
Let \((\X,\Sigma,\mu)\) be a complete, \(\sigma\)-finite measure space with \(\mu\neq 0\). Then a von Neumann lifting
\(\ell\) of the measure \(\mu\) exists.
\end{theorem}

We refer, \emph{e.g.}, to \cite[Section 341]{Fremlin3} for a proof of Theorem \ref{thm:von_Neumann_Maharam}.
However, under some additional (quite mild) assumptions, its proof can be simplified (see Theorem \ref{thm:exist_lift_separable}).
\begin{remark}\label{rmk:wlog_mu_finite}{\rm
Given any \(\sigma\)-finite measure space \((\X,\Sigma,\mu)\), it is clearly possible to construct a finite measure \(\tilde\mu\)
on \((\X,\Sigma)\) such that \(\mu\ll\tilde\mu\ll\mu\), and thus accordingly \(L^\infty(\tilde\mu)=L^\infty(\mu)\).
Observe also that a given map \(\ell\colon\Sigma\to\Sigma\) is a von Neumann lifting of \(\mu\) if and only if it is a von Neumann
lifting of \(\tilde\mu\). Therefore, in several occasions (for instance, during the proof of Theorem \ref{thm:lift_sects})
it will not be restrictive to assume that \(\mu\) is finite.
\fr}\end{remark}

Given a measurable space \((\X,\Sigma)\) and a set \(E\in\Sigma\), we denote by \(\1_E\in\mathcal L^\infty(\Sigma)\) the
\textbf{characteristic function} of \(E\), namely, we define \(\1_E(x)\coloneqq 1\) if \(x\in E\) and \(\1_E(x)\coloneqq 0\)
if \(x\in\X\setminus E\). The space of \textbf{simple functions} on \((\X,\Sigma)\) is then given by
\begin{equation}\label{eq:def_simple_fct}
\mathcal S(\Sigma)\coloneqq\bigg\{\sum_{i=1}^n\lambda_i\1_{E_i}\;\bigg|\;n\in\N,\,(\lambda_i)_{i=1}^n\subset\R,\,
(E_i)_{i=1}^n\subset\Sigma\text{ partition of }\X\bigg\}.
\end{equation}
Now fix a measure \(\mu\) on \((\X,\Sigma)\). For the sake of brevity, we denote
\[
\1_E^\mu\coloneqq\pi_\mu(\1_E)\in L^\infty(\mu),\quad\text{ for every }E\in\Sigma.
\]
Moreover, we denote by \(S(\mu)\) the image of \(\mathcal S(\Sigma)\) under \(\pi_\mu\), namely,
\[
S(\mu)\coloneqq\bigg\{\sum_{i=1}^n\lambda_i\1_{E_i}^\mu\;\bigg|\;n\in\N,\,(\lambda_i)_{i=1}^n\subset\R,
\,(E_i)_{i=1}^n\subset\Sigma\text{ partition of }\X\bigg\}\subset L^\infty(\mu).
\]
It holds that the spaces \(\mathcal S(\Sigma)\) and \(S(\mu)\) are dense in \(\mathcal L^\infty(\Sigma)\) and \(L^\infty(\mu)\), respectively.
\begin{theorem}[Lifting of functions]\label{thm:lift_fcts}
Let \((\X,\Sigma,\mu)\) be a measure space and \(\ell\) a von
Neumann lifting of \(\mu\). Then there exists a unique linear,
continuous map \(\ell\colon L^\infty(\mu)\to\mathcal L^\infty(\Sigma)\)
such that
\begin{equation}\label{eq:def_lift_fcts}
\ell(\1_E^\mu)=\1_{\ell(E)},\quad\text{ for every }E\in\Sigma.
\end{equation}
Moreover, the operator
\(\ell\colon L^\infty(\mu)\to\mathcal L^\infty(\Sigma)\) satisfies
\begin{subequations}\begin{align}
\label{eq:lift_fcts_p1}\pi_\mu\big(\ell(f)\big)=f,&\quad
\text{ for every }f\in L^\infty(\mu),\\
\label{eq:lift_fcts_p2}\ell(fg)=\ell(f)\,\ell(g),&\quad
\text{ for every }f,g\in L^\infty(\mu),\\
\label{eq:lift_fcts_p3}\big\|\ell(f)\big\|_{\mathcal L^\infty(\Sigma)}
=\|f\|_{L^\infty(\mu)},&\quad\text{ for every }f\in L^\infty(\mu).
\end{align}\end{subequations}
\end{theorem}
\begin{proof}[Sketch of proof]
The required linearity and \eqref{eq:def_lift_fcts} force to define \(\ell\colon S(\mu)\to\mathcal S(\Sigma)\) as
\begin{equation}\label{eq:def_lift_simple_fcts}
\ell\bigg(\sum_{i=1}^n\lambda_i\1_{E_i}^\mu\bigg)\coloneqq\sum_{i=1}^n\lambda_i\1_{\ell(E_i)},\quad
\text{ for every }\sum_{i=1}^n\lambda_i\1_{E_i}^\mu\in S(\mu).
\end{equation}
Then the required continuity and the density of \(S(\mu)\) in \(L^\infty(\mu)\) imply that the map defined in
\eqref{eq:def_lift_simple_fcts} can be uniquely extended to a linear, continuous map \(\ell\colon L^\infty(\mu)\to\mathcal L^\infty(\Sigma)\).
Finally, the identities stated in \eqref{eq:lift_fcts_p1}, \eqref{eq:lift_fcts_p2}, and \eqref{eq:lift_fcts_p3} can
be first checked on \(S(\mu)\) directly from the definition \eqref{eq:def_lift_simple_fcts}, and then deduced on the whole
\(L^\infty(\mu)\) by an approximation argument.
\end{proof}

Under the identification \(\Sigma\ni E\mapsto\1_E\in\mathcal L^\infty(\Sigma)\), we know from \eqref{eq:def_lift_fcts} that the
von Neumann lifting \(\ell\colon\Sigma\to\Sigma\) agrees with \(\ell\circ\pi_\mu\colon\mathcal L^\infty(\Sigma)\to\mathcal L^\infty(\Sigma)\).
Therefore, to adopt the same symbol \(\ell\) both for the von Neumann lifting \(\ell\colon\Sigma\to\Sigma\) and for the induced lifting
\(\ell\colon L^\infty(\mu)\to\mathcal L^\infty(\Sigma)\) at the level of functions should cause no ambiguity.
We also point out that in some references (see, \emph{e.g.}, the monograph \cite{Bogachev07}) von Neumann liftings are defined
directly on the space of functions and a posteriori their action on sets is deduced via \eqref{eq:def_lift_fcts}.
\begin{remark}{\rm
In general, in Theorem \ref{thm:lift_fcts} the space \(L^\infty(\mu)\) cannot be replaced
by other reasonable function spaces. For instance, it is shown in \cite{vonNeumann31} that liftings of
\(L^0(\mu)\) might not exist, where \(L^0(\mu)\) stands for the space of measurable
functions on \(\X\) considered up to \(\mu\)-a.e.\ equality. Concerning \(L^p(\mu)\)
spaces with \(p\in[1,\infty)\), a first issue is given by the fact that they are typically
not closed under multiplication. However, it is proved in \cite{StraussMacherasMusial02}
that not even \emph{linear} liftings defined on \(L^p(\mu)\) exist for a generic measure \(\mu\). 
\fr}\end{remark}
\subsection{Differentiation bases}\label{ss:diff_basis}
Below we recall the key concepts concerning the language of differentiation bases, which we
learnt from \cite{BrucknerBrucknerThomson08}. As a matter of convenience, we slightly modify some definitions,
in order to adapt them to our purposes. In this paper, we will be mostly concerned with
differentiation bases induced by a lifting (cf.\ with Lemma \ref{lem:diff_basis_by_lift}).
\begin{definition}[Differentiation basis]\label{def:diff_basis}
Let \((\X,\Sigma,\mu)\) be a measure space. Then a family
\(\mathcal I\subset\Sigma\) is said to be a \textbf{differentiation basis}
on \((\X,\Sigma,\mu)\) provided the following properties hold:
\begin{itemize}
\item[\(\rm i)\)] \(0<\mu(I)<+\infty\) for every \(I\in\mathcal I\).
\item[\(\rm ii)\)] Calling
\(\mathcal I_x\coloneqq\{I\in\mathcal I\,:\,x\in I\}\)
for any \(x\in\X\), we have that
\(\hat\X\coloneqq\{x\in\X\,:\,\mathcal I_x\neq\varnothing\}\in\Sigma\) and \(\mu(\X\setminus\hat\X)=0\). We say that \(\hat\X\) is the
\textbf{regular set} associated with \(\mathcal I\).
\item[\(\rm iii)\)] \(\mathcal I\) is \textbf{directed by downward inclusion},
meaning that if \(x\in\hat\X\) and \(I,I'\in\mathcal I_x\),
then there exists \(J\in\mathcal I_x\) such that \(J\subset I\cap I'\).
\end{itemize}
Moreover, given any \(x\in\hat\X\) and \(I\in\mathcal I_x\), we define
the family \(\mathcal U^{\mathcal I}_x(I)\subset\mathcal I_x\) as
\[
\mathcal U^{\mathcal I}_x(I)\coloneqq\{J\in\mathcal I_x\;|\;J\subset I\}.
\]
\end{definition}

Each differentiation basis \(\mathcal I\) is naturally associated with the
following concept of \(\mathcal I\)-limit.
\begin{definition}[\(\mathcal I\)-limit]
Let \((\X,\Sigma,\mu)\) be a measure space and let \(\mathcal I\) be
a differentiation basis on \((\X,\Sigma,\mu)\). Let \((\Y,\tau)\)
be a Hausdorff topological space and let \(\Phi\colon\mathcal I\to\Y\)
be a given map. Fix \(x\in\hat\X\). Then we declare that the
\textbf{\(\mathcal I\)-limit} of \(\Phi\) at \(x\) exists and
is equal to \(y\in\Y\), briefly
\[
\lim_{I\Rightarrow x}\Phi(I)=y,
\]
provided for any \(U\in\mathcal N_\Y(y)\) there exists \(I\in\mathcal I_x\)
such that \(\Phi(J)\in U\) holds for every \(J\in\mathcal U^{\mathcal I}_x(I)\).
\end{definition}
Observe that if two maps \(\Phi,\Psi\colon\mathcal I\to\Y\) agree on \(\mathcal U^{\mathcal I}_x(I)\) for some \(x\in\hat\X\) and
\(I\in\mathcal I_x\), then it holds \(\lim_{I\Rightarrow x}\Phi(I)=\lim_{I\Rightarrow x}\Psi(I)\), in the sense that one such
\(\mathcal I\)-limit exists if and only if the other one exists, and in that case they coincide. Consequently, one can define
\(\lim_{I\Rightarrow x}\Phi(I)\) as soon as \(\Phi\) is defined just on \(\mathcal U^{\mathcal I}_x(I)\) for some 
\(I\in\mathcal I_x\), and not necessarily on the whole \(\mathcal I\).
\medskip

Hereafter, we will mostly focus our attention on those differentiation bases that are induced by a von Neumann lifting,
in the sense that is explained by the following result.
\begin{lemma}[Differentiation basis induced by a lifting]\label{lem:diff_basis_by_lift}
Let \((\X,\Sigma,\mu)\) be a complete, \(\sigma\)-finite measure space. Let \(\ell\) be a von Neumann lifting of \(\mu\). Then the family
\[
\mathcal I^\ell\coloneqq\big\{\ell(E)\;\big|\;E\in\Sigma,\,0<\mu(E)<+\infty\big\}\subset\Sigma
\]
is a differentiation basis on \((\X,\Sigma,\mu)\).
\end{lemma}
\begin{proof}
Items i) and iii) of Definition \ref{def:diff_basis} easily follow from items iii) and i) of Definition \ref{def:lifting}, respectively.
Concerning item ii) of Definition \ref{def:diff_basis}, observe that the \(\sigma\)-finiteness assumption on \(\mu\) yields existence of
a partition \((E_n)_{n\in\N}\subset\Sigma\) of \(\X\) with \(0<\mu(E_n)<+\infty\) for all \(n\in\N\). Hence, \(\ell(E_n)\in
\mathcal I^\ell_x\) for every \(n\in\N\) and \(x\in\ell(E_n)\), thus accordingly \(\bigcup_n\ell(E_n)\subset\hat\X\). Since \(\mu\) is
complete and \(\mu\big(\X\setminus\bigcup_n\ell(E_n)\big)=0\), we conclude that \(\hat\X\in\Sigma\) and \(\mu(\X\setminus\hat\X)=0\).
\end{proof}
\begin{remark}{\rm
It can be readily checked that if \(\hat\X\) is the regular set associated with \(\mathcal I^\ell\), then
\[
\mu(\X)<+\infty\quad\Longrightarrow\quad\hat\X=\X.
\]
Indeed, in this case it holds that \(x\in\X=\ell(\X)\), and thus \(\X\in\mathcal I^\ell_x\), for every \(x\in\X\).
\fr}\end{remark}
\subsection{Bochner integration}\label{ss:Bochner_int}
Here, we recall the basics of the integration theory in the sense of Bochner
for maps that take values into a Banach space. The notion of Bochner integral,
which was introduced in \cite{Bochner33} and \cite{Dunford35}, is a vector-valued generalisation
of the Lebesgue integral, so in particular it is defined by means of an approximation
argument using simple maps. Some classical references about the Bochner integral are
\cite{DunfordSchwartz58} and \cite{DiestelUhl77}.
\medskip

Let \((\X,\Sigma,\mu)\) be a measure space and \(\B\) a Banach space. Then we denote by
\(\mathcal S(\mu;\B)\) the family of all \textbf{\(\B\)-simple maps} on \((\X,\Sigma,\mu)\). Namely, we define
\begin{equation}\label{eq:def_simple_map}
\mathcal S(\mu;\B)\coloneqq\bigg\{\sum_{i=1}^n\1_{E_i} v_i\;\bigg|\;
\begin{array}{ll}
n\in\N,\,(E_i)_{i=1}^n\subset\Sigma\text{ pairwise disjoint sets},\\
\mu(E_i)<\infty\text{ for all }i=1,\ldots,n,\,(v_i)_{i=1}^n\subset\B
\end{array}\bigg\}.
\end{equation}
Observe that \(\mathcal S(\mu;\B)\) is a linear space if endowed with the natural pointwise operations. Given a set \(E\in\Sigma\),
we define the map \(\mathcal S(\mu;\B)\ni v\mapsto\int_E v\,\d\mu\in\B\) as
\begin{equation}\label{eq:int_simple}
\int_E v\,\d\mu\coloneqq\sum_{i=1}^n\mu(E\cap E_i) v_i\in\B,\quad\text{ for every }v=\sum_{i=1}^n\1_{E_i} v_i\in\mathcal S(\mu;\B).
\end{equation}
The resulting operator \(\int_E\#\,\d\mu\) is linear, and satisfies \(\int_E\|v(\cdot)\|_\B\,\d\mu=\sum_{i=1}^n\mu(E\cap E_i)\|v_i\|_\B\) and
\(\big\|\int_E v\,\d\mu\big\|_\B\leq\int_E\|v(\cdot)\|_\B\,\d\mu\). In the case \(E=\X\), we just write \(\int v\,\d\mu\) in place of \(\int_\X v\,\d\mu\).
\begin{remark}{\rm
We point out that \(\mathcal S(\Sigma)\) and \(\mathcal S(\mu;\R)\) do not coincide when \(\mu\) is an infinite measure. However, we have that
\[
\mathcal S(\mu;\R)=\big\{f\in\mathcal S(\Sigma)\;\big|\;\mu(\{f\neq 0\})<+\infty\big\},
\quad\text{ where }\{f\neq 0\}\coloneqq\big\{x\in\X\;\big|\;f(x)\neq 0\big\}.
\]
On the one hand, for the elements of \(\mathcal S(\Sigma)\) a finiteness assumption on \(\mu(\{f\neq 0\})\) would
destroy the density of \(\mathcal S(\Sigma)\) in \(\mathcal L^\infty(\Sigma)\). On the other hand, such finiteness
assumption is needed for the elements of \(\mathcal S(\mu;\B)\) in order to give a meaning to the integral in \eqref{eq:int_simple}.
\fr}\end{remark}
\begin{definition}[Strong measurability]
Let \((\X,\Sigma,\mu)\) be a \(\sigma\)-finite measure space, \(\B\) a Banach space. Then a map \(v\colon\X\to\B\) is said to
be \textbf{strongly measurable} if it is \(\mu\)-measurable and there exists a sequence \((v_n)_{n\in\N}\subset\mathcal S(\mu;\B)\)
of \(\B\)-simple maps on \((\X,\Sigma,\mu)\) such that
\[
\lim_{n\to\infty}\big\|v_n(x)-v(x)\big\|_\B=0,
\quad\text{ for }\mu\text{-a.e.\ }x\in\X.
\]
\end{definition}

As an immediate consequence of the definition, all \(\B\)-simple maps are strongly measurable. Hereafter, we will often consider
strongly measurable maps up to \(\mu\)-a.e.\ equality. Namely, we define an equivalence relation \(\sim_\mu\) on the set of all
strongly measurable maps from \(\X\) to \(\B\) as follows: given any \(v,w\colon\X\to\B\) that are strongly measurable, we declare that
\begin{equation}\label{eq:def_equiv_rel_mu_ae}
v\sim_\mu w\quad\Longleftrightarrow\quad v(x)=w(x),\,\text{ for }\mu\text{-a.e.\ }x\in\X.
\end{equation}
We denote by \(S(\mu;\B)\) the quotient of \(\mathcal S(\mu;\B)\) with respect to such equivalence relation \(\sim_\mu\).
\begin{definition}[Lebesgue--Bochner space]
Let \((\X,\Sigma,\mu)\) be a \(\sigma\)-finite measure space, \(\B\) a Banach space.
Let \(p\in[1,\infty]\) be given. Then we say that a strongly measurable map \(v\colon\X\to\B\) belongs to \(\mathcal L^p(\mu;\B)\)
if it holds that \(\|v(\cdot)\|_\B\in\mathcal L^p(\mu)\). The \textbf{Lebesgue--Bochner space} \(L^p(\mu;\B)\) is defined as
the quotient of \(\mathcal L^p(\mu;\B)\) with respect to the equivalence relation \(\sim_\mu\) as in \eqref{eq:def_equiv_rel_mu_ae}.
\end{definition}

Each Lebesgue--Bochner space \(L^p(\mu;\B)\) is a Banach space with respect to the norm
\[
\|v\|_{L^p(\mu;\B)}\coloneqq\big\|\|v(\cdot)\|_\B\big\|_{L^p(\mu)},\quad\text{ for every }v\in L^p(\mu;\B).
\]
It is well-known (see, \emph{e.g.}, \cite[Proposition 1.2.2]{HNVW16}) that for any \(v\in\mathcal L^1(\mu;\B)\)
there exists a sequence \((v_n)_{n\in\N}\subset\mathcal S(\mu;\B)\) such that \(\lim_n\big\|v_n(x)-v(x)\big\|_\B=0\) holds
for \(\mu\)-a.e.\ point \(x\in\X\) and \(\lim_n\int\|v_n(\cdot)-v(\cdot)\|_\B\,\d\mu=0\). Consequently, the following definition is meaningful.
\begin{definition}[Bochner integral]
Let \((\X,\Sigma,\mu)\) be a \(\sigma\)-finite measure space, \(\B\) a Banach space.
Let \(v\colon\X\to\B\) be strongly measurable. Then we say that
\(v\) is \textbf{Bochner integrable} over a set \(E\in\Sigma\) provided
\(\1_E\cdot v\in\mathcal L^1(\mu;\B)\). The \textbf{Bochner integral}
of \(v\) over \(E\) is defined as follows: given any
\((v_n)_{n\in\N}\subset\mathcal S(\mu;\B)\) such that
\(\lim_n\big\|v_n(x)-v(x)\big\|_\B=0\) for \(\mu\)-a.e.\ \(x\in E\)
and \(\lim_n\int_E\|v_n(\cdot)-v(\cdot)\|_\B\,\d\mu=0\), we set
\begin{equation}\label{eq:def_Bochner_int}
\int_E v\,\d\mu\coloneqq\lim_{n\to\infty}\int_E v_n\,\d\mu\in\B,
\end{equation}
where the limit is intended with respect to the strong topology of \(\B\).
In the case where \(E=\X\), we just write \(\int v\,\d\mu\) in place of
\(\int_\X v\,\d\mu\).
\end{definition}
It is easy to show that \eqref{eq:def_Bochner_int} is well-posed, in the sense that
the limit appearing in \eqref{eq:def_Bochner_int} exists and it does not depend on the
specific choice of the approximating sequence \((v_n)_{n\in\N}\).
\medskip

It is worth pointing out that if \(v\colon\X\to\B\)
is Bochner integrable on \(E\in\Sigma\), then it holds
\[
\bigg\|\int_E v\,\d\mu\bigg\|_\B\leq\int_E\|v(\cdot)\|_\B\,\d\mu.
\]
This follows from the corresponding statement for \(\B\)-simple
maps and the density of \(S(\mu;\B)\) in \(L^1(\mu;\B)\), thus
the integral operator \(\int_E\#\,\d\mu\colon L^1(\mu;\B)\to\B\)
is linear and continuous.
\begin{remark}[Local Bochner integrability]\label{rmk:local_Boch}{\rm
Let \((\X,\Sigma,\mu)\) be a \(\sigma\)-finite measure space,
\(\mathcal I\) a differentiation basis on \((\X,\Sigma,\mu)\),
and \(\B\) a Banach space. Then a given strongly measurable map
\(v\colon\X\to\B\) is said to be \textbf{locally Bochner integrable
with respect to \(\mathcal I\)} provided for \(\mu\)-a.e.\ point
\(x\in\hat\X\) there exists a set \(I\in\mathcal I_x\) such that
\(v\) is Bochner integrable over \(I\).

In the case where \(\mu(\X)<+\infty\) and \(\mathcal I=\mathcal I^\ell\) for some von Neumann lifting \(\ell\) of \(\mu\), \emph{every}
\(\mu\)-measurable map \(v\colon\X\to\B\) is locally Bochner integrable with respect to \(\mathcal I^\ell\). Indeed, calling
\[
E_k\coloneqq\big\{x\in\X\;\big|\;k-1\leq\|v(x)\|_\B<k\big\}\in\Sigma,\quad\text{ for every }k\in\N,
\]
we have that \((E_k)_{k\in\N}\) is a partition of \(\X\) and that \((\1_{E_k}\cdot v)_{k\in\N}\subset\mathcal L^\infty(\mu;\B)\).
In particular, the map \(v\) is Bochner integrable over each set \(\ell(E_k)\in\mathcal I^\ell\). Since \(\bigcup_k\ell(E_k)\)
has full \(\mu\)-measure in \(\X\), we deduce that \(v\) is locally Bochner integrable with respect to \(\mathcal I^\ell\).
This observation explains why, differently from the classical Lebesgue differentiation theorem, in the statement
of our main Theorem \ref{thm:Lebesgue_diff} no local integrability assumption will be explicitly required.
\fr}\end{remark}
\subsection{Measurable Banach bundles}\label{ss:MBB}
In this section, we recall the notion of a measurable Banach bundle over a
measurable space, which has been introduced in \cite{DiMarinoLucicPasqualetto21}.
Before passing to the actual definition,
we first need to fix some terminology. A \textbf{multivalued map}
\(\boldsymbol\varphi\colon S\twoheadrightarrow T\) between two
sets \(S\) and \(T\) is a mapping from \(S\) to the power set
\(2^T\) of \(T\). Given a measurable space \((\X,\Sigma)\)
and a topological space \((\Y,\tau)\), we say that a
multivalued map \(\boldsymbol\varphi\colon\X\twoheadrightarrow\Y\)
is \textbf{weakly measurable} provided it holds that
\[
\big\{x\in\X\;\big|\;\boldsymbol\varphi(x)\cap U\neq\varnothing\big\}
\in\Sigma,\quad\text{ for every }U\subset\Y\text{ open.}
\]

Let us now introduce the notion of a measurable Banach bundle
over a measurable space:
\begin{definition}[Banach bundle]\label{def:MBB}
Let \((\X,\Sigma)\) be a measurable space and \(\B\)
a Banach space. Then a given multivalued map \({\mathbf E}\colon\X
\twoheadrightarrow\B\) is said to be a \textbf{measurable Banach
\(\B\)-bundle} (or just a \textbf{Banach \(\B\)-bundle}) over
\((\X,\Sigma)\) provided it is weakly measurable and
\({\mathbf E}(x)\) is a closed, linear subspace of \(\B\)
for every \(x\in\X\).
\end{definition}
In \cite[Definition 4.1]{DiMarinoLucicPasqualetto21} the definition of Banach
\(\B\)-bundle was formulated only for \(\B\) separable. Here, we allow for an
arbitrary Banach space \(\B\), so that also the case of Lebesgue--Bochner spaces
is covered by this theory. In this regard, observe that the multivalued map
associating to every \(x\in\X\) the whole space \(\B\) is, trivially, a Banach
\(\B\)-bundle over \((\X,\Sigma)\); with a slight abuse of notation, we denote
it by \(\B\). A consistency check will be expressed by \eqref{eq:LB_vs_bundle} below.
\medskip

Let us remind the concept of a measurable section of a given measurable Banach bundle:
\begin{definition}[Section of a Banach bundle]
Let \((\X,\Sigma)\) be a measurable space. Let \(\B\) be a Banach
space and \(\mathbf E\) a Banach \(\B\)-bundle over \((\X,\Sigma)\).
Then a map \(v\colon\X\to\B\) is said to be a
\textbf{measurable section} of \(\mathbf E\) provided
it is strongly measurable and \(v(x)\in{\mathbf E}(x)\) for all \(x\in\X\).
\end{definition}

Given a \(\sigma\)-finite measure space \((\X,\Sigma,\mu)\), a Banach space \(\B\),
a Banach \(\B\)-bundle \(\mathbf E\) over \((\X,\Sigma)\), and an
exponent \(p\in[1,\infty]\), we define the space of
\textbf{\(p\)-integrable sections} of \(\mathbf E\) as
\[
\bar\Gamma_p(\mu;{\mathbf E})\coloneqq
\big\{v\in\mathcal L^p(\mu;\B)\;\big|\;
v\text{ is a measurable section of }{\mathbf E}\big\}.
\]
Calling \(\pi_\mu\colon\mathcal L^p(\mu;\B)\to L^p(\mu;\B)\)
the quotient map, the space of \textbf{\(L^p\)-sections}
of \(\mathbf E\) is given by
\[
\Gamma_p(\mu;{\mathbf E})\coloneqq
\pi_\mu\big(\bar\Gamma_p(\mu;{\mathbf E})\big).
\]
It holds that \(\Gamma_p(\mu;{\mathbf E})\) is a closed,
linear subspace of \(L^p(\mu;\B)\). Also, the very definitions yield
\begin{equation}\label{eq:LB_vs_bundle}
\bar\Gamma_p(\mu;\B)=\mathcal L^p(\mu;\B),
\qquad\Gamma_p(\mu;\B)=L^p(\mu;\B).
\end{equation}
\begin{remark}\label{rmk:sections_NMod}{\rm
It is easy to check that the space \(\Gamma_p(\mu;{\mathbf E})\) is actually
an \textbf{\(L^p(\mu)\)-normed \(L^\infty(\mu)\)-module}, in the
sense of Gigli (see \cite[Definition 1.2.10]{Gigli18} for the
relevant definition).

As a matter of fact, the theory of normed modules (or, more specifically,
of separable normed modules) was one of the main motivations behind the
introduction of measurable Banach bundles in the sense of Definition \ref{def:MBB}.
Separable normed modules play a key role in the setting of
differential calculus over metric measure spaces, since under mild assumptions
the \emph{cotangent module} is separable (see \cite[Proposition 2.2.5]{Gigli18} and
\cite[Theorem B.1]{DiMarinoLucicPasqualetto21}).

When \(\B\) is separable, Banach \(\B\)-bundles are completely characterised by their
sections. One of the main achievements of \cite{DiMarinoLucicPasqualetto21} is
that \emph{every} separable \(L^p(\mu)\)-normed \(L^\infty(\mu)\)-module (over a
\(\sigma\)-finite measure space) can be written as the space of \(L^p\)-sections of some
Banach \(\B\)-bundle, where \(\B\) is a fixed \emph{universal} separable Banach space,
\emph{i.e.}, a separable Banach space where every separable Banach space can be embedded
linearly and isometrically. We recall that the existence of universal separable Banach
spaces follows from the \emph{Banach--Mazur Theorem} (cf.\ with \cite[Proposition 7.5 of Chapter II]{BessagaPelczynski75}), which states that \(C([0,1])\) enjoys this property.

Conversely, in the case where \(\B\) is non-separable, we do not know if Banach
\(\B\)-bundles are completely determined by their sections, the reason being that we
do not have a measurable selection theorem for an arbitrary target \(\B\).
In other words, there might possibly exist two different Banach \(\B\)-bundles
\(\mathbf E\) and \(\mathbf F\) such that \(\Gamma_p(\mu;{\mathbf E})\)
and \(\Gamma_p(\mu;{\mathbf F})\) are isomorphic.
\fr}\end{remark}
\subsection{Radon--Nikod\'{y}m property}\label{ss:RNP}
Aim of this section is to recall the basics of the theory of vector-valued measures,
as well as the strictly related concept of a Banach space having the Radon--Nikod\'{y}m
property, often abbreviated as RNP in the literature. In short, a Banach space \(\B\) has
the RNP provided every \(\B\)-valued measure of bounded variation verifies the classical
Radon--Nikod\'{y}m Theorem. For a thorough account of the Radon--Nikod\'{y}m property,
we refer to \cite{DiestelUhl77,BenyaminiLindenstrauss00} and the references
therein. Our brief presentation, which will cover only those few notions that are strictly
needed for our purposes, is mostly taken from \cite{HNVW16}.
\medskip

Let \((\X,\Sigma)\) be a measurable space and \(\B\) a Banach
space. A mapping \(\Omega\colon\Sigma\to\B\) is said to be a
\textbf{\(\B\)-valued measure} on \((\X,\Sigma)\) provided
the following property is satisfied:
\[
\lim_{N\to\infty}\bigg\|\Omega(E)-\sum_{n=1}^N\Omega(E_n)\bigg\|_\B=0,
\quad\text{ whenever }(E_n)_{n\in\N}\subset\Sigma
\text{ is a partition of }E\in\Sigma.
\]
We say that \(\Omega\) is \textbf{bounded} (or \textbf{of bounded semivariation}) provided its range is a bounded subset
of \(\B\) (cf.\ with \cite[Proposition I.1.11]{DiestelUhl77}). Following \cite[Definition I.1.12]{DiestelUhl77}, it is
possible to integrate bounded scalar-valued functions with respect to a bounded vector-valued measure:
\begin{definition}[Elementary Bartle integral]\label{def:Bartle_int}
Let \((\X,\Sigma)\) be a measurable space and \(\B\) a Banach space. Let \(\Omega\) be a bounded
\(\B\)-valued measure on \((\X,\Sigma)\). Then we define the map
\[
\mathcal L^\infty(\Sigma)\ni f\mapsto\int f\,\d\Omega\in\B,
\]
called the \textbf{elementary Bartle integral}, as the unique linear continuous operator such that
\begin{equation}\label{eq:def_Bartle_int}
\int f\,\d\Omega=\sum_{i=1}^n\lambda_i\,\Omega(E_i),\quad\text{ for every }f=\sum_{i=1}^n\lambda_i\1_{E_i}\in\mathcal S(\Sigma).
\end{equation}
Moreover, given any \(f\in\mathcal L^\infty(\Sigma)\) and \(E\in\Sigma\), we define \(\int_E f\,\d\Omega\coloneqq\int \1_E f\,\d\Omega\in\B\).
\end{definition}

Observe that, indeed, the operator \(\mathcal S(\Sigma)\ni f\mapsto\int f\,\d\Omega\in\B\) defined in \eqref{eq:def_Bartle_int} verifies
\[\begin{split}
\bigg|\omega\bigg(\int f\,\d\Omega\bigg)\bigg|&=\bigg|\sum_{i=1}^n\lambda_i\,\omega\big(\Omega(E_i)\big)\bigg|
\leq\sum_{i=1}^n|\lambda_i|\big|\omega\big(\Omega(E_i)\big)\big|\\
&\leq\|f\|_{\mathcal L^\infty(\Sigma)}\bigg(\sum_{i\in P}\omega\big(\Omega(E_i)\big)-\sum_{i\in N}\omega\big(\Omega(E_i)\big)\bigg)\\
&=\|f\|_{\mathcal L^\infty(\Sigma)}\Big(\omega\big(\Omega(E^+)\big)-\omega\big(\Omega(E^-)\big)\Big)\\
&\leq\|f\|_{\mathcal L^\infty(\Sigma)}\Big(\|\Omega(E^+)\|_\B+\|\Omega(E^-)\|_\B\Big)\\
&\leq 2\,\|f\|_{\mathcal L^\infty(\Sigma)}\sup_{E\in\Sigma}\|\Omega(E)\|_\B,
\end{split}\]
for every \(\omega\in\B'\) with \(\|\omega\|_{\B'}\leq 1\) and \(f=\sum_{i=1}^n\lambda_i\1_{E_i}\in\mathcal S(\Sigma)\), where we set
\[
P\coloneqq\Big\{i=1,\ldots,n\;\Big|\;\omega\big(\Omega(E_i)\big)\geq 0\Big\},\quad N\coloneqq\{1,\ldots,n\}\setminus P,
\quad E^+\coloneqq\bigcup_{i\in P}E_i,\quad E^-\coloneqq\bigcup_{i\in N}E_i.
\]
By passing to the supremum over all \(\omega\in\B'\) with \(\|\omega\|_{\B'}\leq 1\), we thus obtain that
\begin{equation}\label{eq:ineq_Bartle}
\bigg\|\int f\,\d\Omega\bigg\|_\B\leq 2\,\|f\|_{\mathcal L^\infty(\Sigma)}\sup_{E\in\Sigma}\|\Omega(E)\|_\B, \quad\text{ for every }f\in\mathcal S(\Sigma).
\end{equation}
Moreover, standard verifications show that the operator \(\mathcal S(\Sigma)\ni f\mapsto\int f\,\d\Omega\in\B\) is linear.
Therefore, Definition \ref{def:Bartle_int} is well-posed and the inequality \eqref{eq:ineq_Bartle} holds for all \(f\in\mathcal L^\infty(\Sigma)\).
\medskip

The \textbf{variation} \(\|\Omega\|_\B\colon\Sigma\to[0,+\infty]\) of a \(\B\)-valued measure \(\Omega\) is defined at \(E\in\Sigma\) as
\begin{equation}\label{eq:def_tv_B-val}
\|\Omega\|_\B(E)\coloneqq
\sup\bigg\{\sum_{i=1}^n\|\Omega(E_i)\|_\B\;\bigg|\;n\in\N,\,(E_i)_{i=1}^n\subset\Sigma\text{ partition of }E\bigg\}.
\end{equation}
We say that \(\Omega\) has \textbf{bounded variation} provided \(\|\Omega\|_\B(\X)<+\infty\). In this case, it holds that
the variation \(\|\Omega\|_\B\) is a finite measure on \((\X,\Sigma)\) and that \(\Omega\) is bounded. Observe that any
(non-negative) finite measure \(\mu\) on \((\X,\Sigma)\) has bounded variation and satisfies \(\|\mu\|_\R=\mu\).
\begin{remark}\label{rmk:char_variation}{\rm
If \(\Omega\colon\Sigma\to\B\) has bounded variation, its variation \(\|\Omega\|_\B\) can be characterised as the least
measure \(\mu\colon\Sigma\to[0,+\infty)\) such that \(\|\Omega(E)\|_\B\leq\mu(E)\) for every \(E\in\Sigma\).
\fr}\end{remark}
%
We say that \(\Omega\colon\Sigma\to\B\) is \textbf{absolutely continuous} with respect to \(\mu\), briefly \(\Omega\ll\mu\), if
\[
\Omega(N)=0_\B,\quad\text{ for every }N\in\Sigma\text{ with }\mu(N)=0.
\]
It holds that \(\Omega\ll\mu\) if and only if \(\|\Omega\|_\B\ll\mu\).
\medskip

The set of \(\B\)-valued measures having bounded variation is closed under various operations:
\begin{itemize}
\item[\(\rm i)\)] \textsc{Restriction.} Let \((\X,\Sigma)\) be a measurable space, \(\B\) a Banach space, and \(\Omega\) a bounded
\(\B\)-valued measure on \((\X,\Sigma)\). Then for any given \(f\in\mathcal L^\infty(\Sigma)\) we define \(f\Omega\colon\Sigma\to\X\) as
\[
(f\Omega)(E)\coloneqq\int_E f\,\d\Omega,\quad\text{ for every }E\in\Sigma.
\]
It holds that \(f\Omega\) is a bounded \(\B\)-valued measure on \((\X,\Sigma)\). Moreover, if \(\Omega\) has bounded variation,
then \(f\Omega\) has bounded variation as well and it holds that
\begin{equation}\label{eq:tv_restr}
\|f\Omega\|_\B=|f|\|\Omega\|_\B,\quad\text{ for every }f\in\mathcal L^\infty(\Sigma).
\end{equation}
In particular, for any set \(F\in\Sigma\) we can consider the \textbf{restriction} of \(\Omega\) to \(F\), which is given by
\(\Omega|_F\coloneqq\1_F\Omega\) and satisfies \(\|\Omega|_F\|_\B=\|\Omega\|_\B|_F\) when \(\Omega\) has bounded variation.
\item[\(\rm ii)\)] \textsc{Pushforward.} Let \((\X,\Sigma_\X)\), \((\Y,\Sigma_\Y)\) be measurable spaces, \(\B\) a Banach space,
\(\Omega\) a \(\B\)-valued measure on \((\X,\Sigma_\X)\), and \(\varphi\colon\X\to\Y\) a measurable map. Then we define the
\textbf{pushforward} of \(\Omega\) under \(\varphi\) as
\[
\varphi_\#\Omega(E)\coloneqq\Omega\big(\varphi^{-1}(E)\big)\in\B,
\quad\text{ for every }E\in\Sigma_\Y.
\]
It holds that \(\varphi_\#\Omega\colon\Sigma_\Y\to\B\) is a
\(\B\)-valued measure on \((\Y,\Sigma_\Y)\). Moreover, if
\(\Omega\) has bounded variation, then \(\varphi_\#\Omega\)
has bounded variation as well and it holds that
\begin{equation}\label{eq:tv_pushfrwd}
\|\varphi_\#\Omega\|_\B\leq\varphi_\#\|\Omega\|_\B.
\end{equation}
\end{itemize}

In Section \ref{s:disint} we will focus our attention on the following special class of Banach spaces:
\begin{definition}[Radon--Nikod\'{y}m property]
Let \(\B\) be a Banach space. Let \((\X,\Sigma,\mu)\) be a
\(\sigma\)-finite measure space. Then we say that \(\B\) has the
\textbf{Radon--Nikod\'{y}m property with respect to \((\X,\Sigma,\mu)\)}
if the following property holds: given any \(\B\)-valued measure
\(\Omega\) on \((\X,\Sigma)\) having bounded variation and satisfying
\(\Omega\ll\mu\), there exists \(\frac{\d\Omega}{\d\mu}
\in L^1(\mu;\B)\) such that
\begin{equation}\label{eq:def_RN_deriv}
\Omega(E)=\int_E\frac{\d\Omega}{\d\mu}\,\d\mu,
\quad\text{ for every }E\in\Sigma.
\end{equation}
The element \(\frac{\d\Omega}{\d\mu}\), which is uniquely determined
by \eqref{eq:def_RN_deriv}, is called the
\textbf{Radon--Nikod\'{y}m derivative} of \(\Omega\)
with respect to \(\mu\). Moreover, we say that \(\B\)
has the \textbf{Radon--Nikod\'{y}m property} if it has the
Radon--Nikod\'{y}m property with respect to every \(\sigma\)-finite
measure space.
\end{definition}

A Banach space has the Radon--Nikod\'{y}m property if and only if it has the Radon--Nikod\'{y}m property with respect to
\(\big([0,1],\mathscr B([0,1]),\mathcal L^1|_{[0,1]}\big)\), where \(\mathscr B([0,1])\) is the Borel \(\sigma\)-algebra of
the interval \([0,1]\) and \(\mathcal L^1\) is the Lebesgue measure on \(\R\); cf.\ with \cite[Theorem 1.3.26]{HNVW16}.
\begin{remark}\label{rmk:RNP_no_copy_c_0}{\rm
The Radon--Nikod\'{y}m property is \textbf{separably determined}, in the sense that a Banach space \(\B\) has the RNP if and only if every
closed, separable linear subspace of \(\B\) has the RNP. We refer, \emph{e.g.}, to \cite[Theorem 1.3.18]{HNVW16} for a proof of this fact.
We also recall that the Banach space \(c_0\) of all sequences \((\alpha_n)_{n\in\N}\subset\R\) converging to \(0\), endowed with the supremum
norm \(\big\|(\alpha_n)_{n\in\N}\big\|_{c_0}\coloneqq\sup_{n\in\N}|\alpha_n|\), does not have the RNP (according to \cite[Example III.1.1]{DiestelUhl77}).
Therefore, no linear subspace of a Banach space \(\B\) having the RNP is isometrically isomorphic to \(c_0\). In particular, given any
compact, Hausdorff topological space \((\X,\tau)\), we know from \cite[Theorem VI.2.15]{DiestelUhl77} that every linear continuous operator
\(T\colon C(\X)\to\B\) is \textbf{weakly compact}, meaning that the weak closure of the image under \(T\) of the unit sphere of \(C(\X)\)
is a weakly compact set. This fact will play a role during the proof of Theorem \ref{thm:disint_vect_meas}.
\fr}\end{remark}
\subsection{Ultrafilters and ultralimits}\label{ss:ultralimits}
In the sequel, it will be often convenient to work with filter convergence and
ultralimits, which we are going to remind for the reader's usefulness.
The material we will present is very standard, see for instance the book \cite{Bourbaki65}.
\begin{definition}[Filter and ultrafilter]\label{def:filter}
Let \(\Lambda\) be a non-empty set. Then a family \(\mathcal F\subset 2^\Lambda\)
of subsets of \(\Lambda\) is said to be a \textbf{filter} on \(\Lambda\)
provided the following properties are satisfied:
\begin{itemize}
\item[\(\rm i)\)] \(\varnothing\notin\mathcal F\).
\item[\(\rm ii)\)] If \(S,T\in\mathcal F\), then \(S\cap T\in\mathcal F\).
\item[\(\rm iii)\)] If \(S\in\mathcal F\) and \(T\in 2^\Lambda\) satisfy
\(S\subset T\), then \(T\in\mathcal F\).
\end{itemize}
A filter \(\omega\) on \(\Lambda\) which is maximal with respect to inclusion
is said to be an \textbf{ultrafilter} on \(\Lambda\).
\end{definition}

It holds that a given filter \(\mathcal F\) on \(\Lambda\) is an ultrafilter if
and only if for any \(S\in 2^\Lambda\) one has that either \(S\in\mathcal F\)
or \(\Lambda\setminus S\in\mathcal F\). An ultrafilter \(\omega\) on \(\Lambda\)
is said to be \textbf{principal} if it admits a least element,
\textbf{non-principal} if not. In the case where \(\omega\) is principal,
its least element is uniquely determined and is a singleton. When the set
\(\Lambda\) is finite, every ultrafilter on \(\Lambda\) is principal.
\medskip

An important example of filter is the following one: given a
topological space \((\Y,\tau)\) and a point \(y\in\Y\), the
family of all (not necessarily open) neighbourhoods of \(y\) is a filter on \(\Y\).
Conversely, \(\mathcal N_\Y(y)\) is typically not a filter,
as there might well be neighbourhoods of \(y\) which are not open.
In this paper, we mostly work with the filter induced by a
differentiation basis:
\begin{definition}[Filter induced by a differentiation basis]
\label{def:filter_ind_by_diff_basis}
Let \((\X,\Sigma,\mu)\) be a measure space and let \(\mathcal I\)
be a differentiation basis on \((\X,\Sigma,\mu)\). Let us define
\(\tilde{\mathcal I}\coloneqq\{\X\setminus I\,:\,I\in\mathcal I\}\) and
\[\begin{split}
\tilde{\mathcal I}_x\coloneqq\big\{\X\setminus I\;\big|\;
I\in\mathcal I_x\big\}&,\quad\text{ for every }x\in\hat\X,\\
\tilde{\mathcal U}_x^{\mathcal I}(I)\coloneqq\big\{\X\setminus J
\;\big|\;J\in\mathcal U_x^{\mathcal I}(I)\big\}&,\quad
\text{ for every }x\in\hat\X\text{ and }I\in\mathcal I_x.
\end{split}\]
Then for any regular point \(x\in\hat\X\) we define the filter
\(\mathcal F_x^{\mathcal I}\) on \(\tilde{\mathcal I}_x\) as
\[
\mathcal F_x^{\mathcal I}\coloneqq\big\{S\subset\tilde{\mathcal I}_x
\;\big|\;\tilde{\mathcal U}_x^{\mathcal I}(I)\subset S,
\text{ for some }I\in\mathcal I_x\big\}.
\]
We say that \(\mathcal F_x^{\mathcal I}\) is the
\textbf{filter at \(x\) induced by \(\mathcal I\)}.
In the case where \(\mathcal I=\mathcal I^\ell\) for
some von Neumann lifting \(\ell\) of \(\mu\), we write
\(\mathcal U_x^\ell(I)\) and \(\tilde{\mathcal U}_x^\ell(I)\)
in place of \(\mathcal U_x^{\mathcal I^\ell}(I)\) and
\(\tilde{\mathcal U}_x^{\mathcal I^\ell}(I)\), respectively.
\end{definition}
\medskip

Filters on a topological space may be regarded as `locating schemes',
and as such they naturally come with the following notion of convergence:
\begin{definition}[Convergence of filters]
Let \((\Y,\tau)\) be a Hausdorff topological space
and \(\mathcal F\) a filter on \(\Y\).
Let \(y\in\Y\) be given. Then we say that \textbf{\(\mathcal F\)
converges to \(y\)}, briefly \(\mathcal F\to y\), provided it
holds that \(\mathcal N_\Y(y)\subset\mathcal F\).
\end{definition}

Given two non-empty sets \(\Lambda\) and \(\Xi\), a map
\(\phi\colon\Lambda\to\Xi\), and a filter \(\mathcal F\)
on \(\Lambda\), we define
\[
\phi_*\mathcal F\coloneqq\big\{T\in 2^\Xi\;\big|\;
\phi(S)\subset T,\text{ for some }S\in\mathcal F\big\}.
\]
The family \(\phi_*\mathcal F\), which is a filter on \(\Xi\), is
called the \textbf{pushforward filter of \(\mathcal F\) under \(\phi\)}.
\medskip

The following result shows that \(\mathcal I\)-limits can be expressed
in terms of filter convergence.
\begin{lemma}\label{lem:equiv_I-limit}
Let \((\X,\Sigma,\mu)\) be a measure space and \(\mathcal I\) a
differentiation basis on \((\X,\Sigma,\mu)\). Let \((\Y,\tau)\) be
a Hausdorff topological space and \(\Phi\colon\mathcal I\to\Y\) a map.
Fix \(x\in\hat\X\) and \(y\in\Y\). Then
\[
\lim_{I\Rightarrow x}\Phi(I)=y\quad\Longleftrightarrow\quad
\tilde\Phi_*\mathcal F_x^{\mathcal I}\to y,
\]
where the map \(\tilde\Phi\colon\tilde{\mathcal I}\to\Y\)
is defined as \(\tilde\Phi(\X\setminus I)\coloneqq\Phi(I)\)
for every \(I\in\mathcal I\).
\end{lemma}
\begin{proof}
The statement follows from the observation that, given any
\(U\in\mathcal N_\Y(y)\) and \(I\in\mathcal I_x\), it holds that
\(\Phi(J)\in U\) for every \(J\in\mathcal U_x^{\mathcal I}(I)\) if and
only if \(\tilde\Phi\big(\tilde{\mathcal U}_x^{\mathcal I}(I)\big)
=\Phi\big(\mathcal U_x^{\mathcal I}(I)\big)\subset U\).
\end{proof}
\begin{definition}[Ultralimit]
Let \(\Lambda\neq\varnothing\) be a set and \(\omega\) an ultrafilter on
\(\Lambda\). Let \((\Y,\tau)\) be a Hausdorff topological space and
\(\Phi\colon\Lambda\to\Y\) a given map. Then we declare that 
the \textbf{ultralimit} of the map \(\Phi\) with respect to \(\omega\)
exists and is equal to \(y\in\Y\), briefly
\[
\omega\text{-}\lim_i\Phi(i)=y,
\]
provided for any \(U\in\mathcal N_\Y(y)\) it holds that
\(\big\{i\in\Lambda\,:\,\Phi(i)\in U\big\}\in\omega\).
\end{definition}
\begin{remark}\label{rmk:equiv_ultralim}{\rm
It is worth pointing out that the notion of ultralimit can
be reformulated in terms of filter convergence.
Namely, it holds \(\omega\text{-}\lim_i\Phi(i)=y\) if and only if \(\Phi_*\omega\to y\).
\fr}\end{remark}

Let us recall a few basic calculus rules concerning ultralimits:
\begin{itemize}
\item[\(\rm i)\)] When the ultralimit \(\omega\text{-}\lim_i\Phi(i)\) exists,
it is uniquely determined (by the Hausdorff assumption on \(\Y\)).
If \(\omega\) is principal and its least element is given by \(\{i_0\}\), then
\[
\exists\,\omega\text{-}\lim_i\Phi(i)=\Phi(i_0).
\]
Conversely, if \(\omega\) is non-principal, then the ultralimit
\(\omega\text{-}\lim_i\Phi(i)\) might not exist. A sufficient
condition for its existence is given by the compactness of \((\Y,\tau)\).
\item[\(\rm ii)\)] Let \(f\colon\Y\to\R\) be a lower semicontinuous
function. Suppose that both \(\omega\text{-}\lim_i\Phi(i)\) and
\(\omega\text{-}\lim_i f\big(\Phi(i)\big)\) exist. Then it holds that
\begin{equation}\label{eq:ultra_lsc}
f\Big(\omega\text{-}\lim_i\Phi(i)\Big)\leq
\omega\text{-}\lim_i f\big(\Phi(i)\big).
\end{equation}
\item[\(\rm iii)\)] Suppose \((\Y,\sfd)\) is a separated topological
vector space. Let \(\Phi,\Psi\colon\Lambda\to\Y\) be two given maps
and \(\lambda\in\R\). Suppose that \(\omega\text{-}\lim_i\Phi(i)\)
and \(\omega\text{-}\lim_i\Psi(i)\) exist. Then it holds that
\begin{subequations}\begin{align}
\label{eq:lin_ultralim_1}
\exists\,\omega\text{-}\lim_i\big(\Phi(i)+\Psi(i)\big)&=
\Big(\omega\text{-}\lim_i\Phi(i)\Big)+\Big(\omega\text{-}\lim_i\Psi(i)\Big),\\
\label{eq:lin_ultralim_2}
\exists\,\omega\text{-}\lim_i\big(\lambda\,\Phi(i)\big)&=
\lambda\Big(\omega\text{-}\lim_i\Phi(i)\Big).
\end{align}\end{subequations}
\end{itemize}
The Ultrafilter Lemma -- which we will enunciate in the next result -- is particularly
relevant for us. In ZF, the Ultrafilter Lemma is strictly weaker than the Axiom of Choice,
but strictly stronger than the Hahn--Banach Extension Theorem.
\begin{theorem}[Ultrafilter Lemma]\label{thm:ultrafilter_lemma}
Let \(\Lambda\neq\varnothing\) be a set and \(\mathcal F\) a filter
on \(\Lambda\). Then there exists an ultrafilter \(\omega\) on \(\Lambda\)
containing \(\mathcal F\).
\end{theorem}
\section{Metric-valued Lebesgue differentiation theorem in measure spaces}\label{s:LDT}
In this section, we prove a version of Lebesgue Differentiation Theorem for metric-valued
maps defined on a measure space (Theorem \ref{thm:Lebesgue_diff}), as well as some
of its direct consequences.
\medskip

We first introduce some relevant terminology.
Let \((\X,\Sigma,\mu)\) be a \(\sigma\)-finite measure space and \((\Y,\tau)\) a Hausdorff topological
space. Then we say that a measurable map \(\varphi\colon\X\to\Y\) is
\textbf{essentially separably valued} if there exists \(N\in\Sigma\) such that
\(\mu(N)=0\) and \(\varphi(\X\setminus N)\) is a separable subset of \(\Y\).
\begin{remark}{\rm
Given any Banach space \(\B\), it is well-known that a
\(\mu\)-measurable map \(v\colon\X\to\B\) is essentially separably valued
if and only if it is strongly measurable.
\fr}\end{remark}
Theorem \ref{thm:Lebesgue_diff} will easily follow from the ensuing key result,
which states that every essentially separably valued, measurable map from a complete
measure space to a metric space is `approximately continuous' at almost every point
with respect to the differentiation basis induced by a lifting. In fact, this notion
of approximate continuity can be expressed in terms of a suitable topology, called the
\emph{density topology}. Since this characterisation is interesting but not
strictly needed for our purposes, we postpone the relative discussion to Appendix
\ref{app:approx_cont}.
\begin{lemma}\label{lem:pre_Leb}
Let \((\X,\Sigma,\mu)\) be a complete, \(\sigma\)-finite measure space
and let \((\Y,\sfd)\) be a metric space. Let \(\ell\) be a von Neumann
lifting of \(\mu\). Let \(\varphi\colon\X\to\Y\) be an essentially
separably valued, measurable map. Then for \(\mu\)-a.e.\ point
\(x\in\X\) the following property holds:
\begin{equation}\label{eq:pre_Leb_claim}
\forall\varepsilon>0,\quad\exists\,I\in\mathcal I^\ell_x:\quad\sfd\big(\varphi(z),\varphi(x)\big)<\varepsilon,\quad\text{for }\mu\text{-a.e.\ }z\in I.
\end{equation}
\end{lemma}
\begin{proof}
Pick a set \(N'\in\Sigma\) with \(\mu(N')=0\) and a separable set
\(\Y'\subset\Y\) such that \(\varphi(\X\setminus N')\subset\Y'\).
Fix a dense sequence \((y_n)_{n\in\N}\) in \(\Y'\). Let us define
\[
A^k_n\coloneqq\big\{x\in\X\setminus N'\;\big|\;\sfd\big(\varphi(x),y_n\big)
<1/k\big\},\quad B^k_n\coloneqq\ell(A^k_n),\quad\text{ for every }n,k\in\N.
\]
Given that \(\mu(A^k_n\Delta B^k_n)=0\) for every \(n,k\in\N\) and
\(\bigcup_{n\in\N}A^k_n=\X\setminus N'\) for every \(k\in\N\), we have that
\(N\coloneqq\X\setminus\bigcap_{k\in\N}\bigcup_{n\in\N}A^k_n\cap B^k_n
\in\Sigma\) is \(\mu\)-negligible. We claim that each point \(x\) belonging
to the full \(\mu\)-measure set \(\hat\X\setminus N\) satisfies
\eqref{eq:pre_Leb_claim}. In order to prove it, fix \(x\in\hat\X\setminus N\)
and \(\varepsilon>0\). Choose any \(k\in\N\) for which \(1/k\leq\varepsilon/2\).
Hence, we can find \(n\in\N\) such that \(x\in A^k_n\cap B^k_n\). Being
\(x\in\hat\X\), there exists some \(I'\in\mathcal I^\ell_x\). Let us now define
\(I\coloneqq I'\cap B^k_n\in\mathcal I^\ell_x\). Observe that
\[
\sfd\big(\varphi(z),\varphi(x)\big)\leq\sfd\big(\varphi(z),y_n\big)+
\sfd\big(y_n,\varphi(x)\big)<\frac{1}{k}+\frac{1}{k}\leq\varepsilon,
\]
for all \(z\in A^k_n\) and thus for \(\mu\)-a.e.\ \(z\in I\).
This proves \eqref{eq:pre_Leb_claim} and accordingly the statement.
\end{proof}
\begin{remark}{\rm
Albeit not needed for our purposes, we point out that Lemma \ref{lem:pre_Leb} holds in greater generality.
More precisely, the target can be chosen to be a \emph{uniformisable}, Hausdorff topological space \((\Y,\tau)\).
In this case, the statement reads as: for \(\mu\)-a.e.\ \(x\in\X\), it holds that
\[
\forall\,\text{entourage }U\text{ of }\Y,\quad\exists\,I\in\mathcal I^\ell_x:\quad\big(\varphi(z),\varphi(x)\big)\in U,
\quad\text{for }\mu\text{-a.e.\ }z\in I,
\]
as one can check by suitably adapting the proof of Lemma \ref{lem:pre_Leb}.
\fr}\end{remark}
Having Lemma \ref{lem:pre_Leb} at disposal, it is almost immediate to obtain the following result:
\begin{theorem}[Metric-valued Lebesgue differentiation theorem]\label{thm:Lebesgue_diff}
Let \((\X,\Sigma,\mu)\) be a complete, \(\sigma\)-finite measure space
and let \((\Y,\sfd)\) be a metric space. Let \(\ell\) be a von Neumann
lifting of \(\mu\). Let \(\varphi\colon\X\to\Y\) be an essentially
separably valued, measurable map. Then it holds that
\begin{equation}\label{eq:Leb_claim1}
\lim_{I\Rightarrow x}\fint_I\sfd\big(\varphi(\cdot),\varphi(x)\big)\,\d\mu=0,
\quad\text{ for }\mu\text{-a.e.\ }x\in\X.
\end{equation}
In particular, if \(\B\) is a Banach space and \(v\colon\X\to\B\)
is a strongly measurable map, then
\begin{equation}\label{eq:Leb_claim2}
v(x)=\lim_{I\Rightarrow x}\fint_I v\,\d\mu,
\quad\text{ for }\mu\text{-a.e.\ }x\in\X.
\end{equation}
The Bochner integral in the right-hand side of
\eqref{eq:Leb_claim2} is well-defined thanks to
Remark \ref{rmk:local_Boch}.
\end{theorem}
\begin{proof}
To prove \eqref{eq:Leb_claim1} amounts to showing that for
\(\mu\)-a.e.\ \(x\in\X\) the following property holds:
\begin{equation}\label{eq:Leb_aux1}
\forall\varepsilon>0\quad\exists\,I\in\mathcal I^\ell_x:\quad
\fint_J\sfd\big(\varphi(\cdot),\varphi(x)\big)\,\d\mu\leq\varepsilon,
\quad\text{for every }J\in\mathcal U^\ell_x(I).
\end{equation}
This readily stems from Lemma \ref{lem:pre_Leb}: given a point
\(x\in\X\) satisfying \eqref{eq:pre_Leb_claim} and any \(\varepsilon>0\),
there exists \(I\in\mathcal I^\ell_x\) such that
\(\sfd\big(\varphi(z),\varphi(x)\big)<\varepsilon\) for
\(\mu\)-a.e.\ \(z\in I\). In particular, for any
\(J\in\mathcal U^\ell_x(I)\) we may estimate
\(\fint_J\sfd\big(\varphi(\cdot),\varphi(x)\big)\,\d\mu\leq\varepsilon\),
whence the claim \eqref{eq:Leb_aux1} follows.

We pass to the verification of \eqref{eq:Leb_claim2}. In view of
Lemma \ref{lem:pre_Leb} and \eqref{eq:Leb_claim1}, we know that for
\(\mu\)-a.e.\ \(x\in\X\) there exists \(I'\in\mathcal I^\ell_x\) such
that \(v\) is bounded (thus, Bochner integrable) on \(I'\) and
\begin{equation}\label{eq:Leb_aux2}
\lim_{I\Rightarrow x}\fint_I\|v(\cdot)-v(x)\|_\B\,\d\mu=0.
\end{equation}
Observe also that for every \(I\in\mathcal U^\ell_x(I')\) it holds that
\(\big\|\fint_I v\,\d\mu-v(x)\big\|_\B\leq\fint_I\|v(\cdot)-v(x)\|_\B\,\d\mu\).
By letting \(I\Rightarrow x\) and using \eqref{eq:Leb_aux2} in the previous
estimate, we obtain \eqref{eq:Leb_claim2}, as required.
\end{proof}
Theorem \ref{thm:Lebesgue_diff} ensures that, once a von Neumann lifting
of the reference measure is fixed, the \(L^p\)-sections of a given Banach
bundle admit a distinguished representative:
\begin{theorem}[Lebesgue points and precise representatives]\label{thm:precise_repr}
Let \((\X,\Sigma,\mu)\) be a complete, \(\sigma\)-finite measure
space and \(\ell\) a von Neumann lifting of \(\mu\).
Let \(\B\) be a Banach space and \(\mathbf E\) a Banach
\(\B\)-bundle over \((\X,\Sigma)\). Let \(p\in[1,\infty]\) and
\(v\in\Gamma_p(\mu;{\mathbf E})\) be given. Then \(\mu\)-a.e.\ \(x\in\hat\X\) is a
\textbf{Lebesgue point} of \(v\), meaning that the following \(\mathcal I^\ell\)-limit
exists and belongs to \({\mathbf E}(x)\):
\begin{equation}\label{eq:def_precise_repr}
\hat v(x)\coloneqq\lim_{I\Rightarrow x}\fint_I v\,\d\mu\in{\mathbf E}(x).
\end{equation}
Moreover, setting \(\hat v(x)\coloneqq 0_{{\mathbf E}(x)}\) for every non-Lebesgue
point \(x\in\X\) of \(v\), it holds that the resulting map \(\hat v\colon\X\to\B\)
belongs to \(\bar\Gamma_p(\mu;{\mathbf E})\) and satisfies \(\pi_\mu(\hat v)=v\).
We will say that the operator \(\hat\#\colon\Gamma_p(\mu;{\mathbf E})\to
\bar\Gamma_p(\mu;{\mathbf E})\) is the \textbf{precise representative map}
associated with \(\ell\).
\end{theorem}
\begin{proof}
The claim is a direct consequence of Theorem \ref{thm:Lebesgue_diff}. Note that
the precise representative map \(\hat\#\colon\Gamma_p(\mu;{\mathbf E})\to
\bar\Gamma_p(\mu;{\mathbf E})\) is well-posed, since two \(\mu\)-a.e.\ equivalent
elements of \(\bar\Gamma_p(\mu;{\mathbf E})\) must share the same precise representative,
as one can immediately deduce from \eqref{eq:def_precise_repr}.
\end{proof}
Another immediate consequence of Theorem \ref{thm:Lebesgue_diff} is the following:
on Banach spaces having the RNP, the so-called Radon--Nikod\'{y}m derivative truly behaves
as a derivative with respect to the differentiation basis induced by some von Neumann
lifting. More precisely:
\begin{corollary}[Differentiation of \(\B\)-valued measures]
\label{cor:diff_measures_RNP}
Let \((\X,\Sigma,\mu)\) be a complete, \(\sigma\)-finite measure
space. Let \(\B\) be a Banach space having the Radon--Nikod\'{y}m
property. Let \(\Omega\) be any \(\B\)-valued measure on
\((\X,\Sigma)\) having bounded variation and satisfying
\(\Omega\ll\mu\). Then it holds
\begin{equation}\label{eq:explicit_RN}
\frac{\d\Omega}{\d\mu}(x)=\lim_{I\Rightarrow x}
\frac{\Omega(I)}{\mu(I)}\in\B,\quad\text{ for }\mu\text{-a.e.\ }x\in\X.
\end{equation}
\end{corollary}
\begin{proof}
Fix a representative \(v\in\mathcal L^1(\mu;\B)\) of
\(\frac{\d\Omega}{\d\mu}\in L^1(\mu;\B)\). Then Theorem \ref{thm:Lebesgue_diff} yields
\[
v(x)=\lim_{I\Rightarrow x}\fint_I v\,\d\mu=\lim_{I\Rightarrow x}\fint_I\frac{\d\Omega}{\d\mu}\,\d\mu
=\lim_{I\Rightarrow x}\frac{\Omega(I)}{\mu(I)},\quad\text{ for }\mu\text{-a.e.\ }x\in\X,
\]
whence the claimed identity \eqref{eq:explicit_RN} follows.
\end{proof}

Corollary \ref{cor:diff_measures_RNP} can be used to provide an immediate proof of the following well-known fact,
which will be used in Section \ref{ss:disint}. Under the assumptions of Corollary \ref{cor:diff_measures_RNP}, we claim that
\begin{equation}\label{eq:key_ineq_density}
\bigg\|\frac{\d\Omega}{\d\mu}(\cdot)\bigg\|_\B\leq\frac{\d\|\Omega\|_\B}{\d\mu},\quad\text{ holds }\mu\text{-a.e.\ on }\X.
\end{equation}
Indeed, for \(\mu\)-a.e.\ \(x\in\X\) we may estimate
\[
\bigg\|\frac{\d\Omega}{\d\mu}(x)\bigg\|_\B\overset{\eqref{eq:explicit_RN}}=
\bigg\|\lim_{I\Rightarrow x}\frac{\Omega(I)}{\mu(I)}\bigg\|_\B\overset\star=\lim_{I\Rightarrow x}\bigg\|\frac{\Omega(I)}{\mu(I)}\bigg\|_\B
\leq\lim_{I\Rightarrow x}\frac{\|\Omega\|_\B(I)}{\mu(I)}\overset{\eqref{eq:explicit_RN}}=\frac{\d\|\Omega\|_\B}{\d\mu}(x),
\]
where the starred identity follows from the continuity of \(\|\cdot\|_\B\colon\B\to\R\). This yields \eqref{eq:key_ineq_density}.
\section{Liftings of sections}\label{s:lift_sections}
Aim of this section is to obtain a lifting at the level of sections of
a measurable Banach bundle (see Theorem \ref{thm:lift_sects}) by applying
the Lebesgue Differentiation Theorem \ref{thm:Lebesgue_diff}.
\medskip

Let \((\X,\Sigma,\mu)\) be a complete, finite measure space
and let \(\ell\) be a von Neumann lifting of \(\mu\).
In the technical Lemma \ref{lem:construct_ultrafilter}, we will
consider the set \(\mathscr A(\ell)\subset\X=\hat\X\), which is defined as
\[
\mathscr A(\ell)\coloneqq\Big\{x\in\X\;\Big|\;
(\mathcal I^\ell_x,\subset)\text{ has a least element }I_x\Big\}.
\]
The measurability of the set \(\mathscr A(\ell)\) is a consequence of Lemma
\ref{lem:A_ell_meas} below, which provides an equivalent characterisation
of \(\mathscr A(\ell)\) in terms of the atoms of the measure \(\mu\).
Before passing to the actual statement, let us recall some basic
terminology and properties about atoms.
Given a \(\sigma\)-finite measure space \((\X,\Sigma,\mu)\),
we say that a set \(A\in\Sigma\) with \(\mu(A)>0\) is an
\textbf{atom} of \(\mu\) if
\[
B\in\Sigma,\;B\subset A
\quad\Longrightarrow\quad\text{either }\mu(B)=0
\text{ or }\mu(A\setminus B)=0.
\]
The \(\sigma\)-finiteness assumption on \(\mu\) ensures that any atom of \(\mu\)
has finite \(\mu\)-measure. Observe also that if two given sets \(A,A'\in\Sigma\)
satisfy \(\mu(A\Delta A')=0\), then \(A\) is an atom of \(\mu\) if
and only if \(A'\) is an atom of \(\mu\). Therefore, since the
measure \(\mu\) is \(\sigma\)-finite, there can exist at most countably many
atoms of \(\mu\) which are not mutually \(\mu\)-a.e.\ equivalent.
\begin{lemma}\label{lem:A_ell_meas}
Let \((\X,\Sigma,\mu)\) be a complete, finite measure space.
Let \(\ell\) be a von Neumann lifting of \(\mu\). Then it holds that
\begin{equation}\label{eq:equiv_A_ell}
\mathscr A(\ell)=\bigcup_{A\text{ atom of }\mu}\ell(A).
\end{equation}
In particular, it holds that \(\mathscr A(\ell)\in\Sigma\).
\end{lemma}
\begin{proof}
In order to prove the inclusion \(\subset\) in \eqref{eq:equiv_A_ell},
let \(x\in\mathscr A(\ell)\) be fixed. We aim to show that the least
element \(I_x\) of \((\mathcal I_x^\ell,\subset)\) is an atom of
\(\mu\). We argue by contradiction: suppose there exists \(B\in\Sigma\)
such that \(B\subset I_x\) and \(\mu(B),\mu(I_x\setminus B)>0\).
Since \(I_x=\ell(I_x)\) is the disjoint union of \(\ell(B)\) and
\(\ell(I_x\setminus B)\), we can assume (possibly interchanging
\(B\) and \(I_x\setminus B\)) that \(x\in\ell(B)\), so that
\(\ell(B)\in\mathcal I_x^\ell\). Given that \(\ell(B)\subset I_x\)
and \(\ell(B)\neq I_x\), we contradict the minimality of \(I_x\).
Then \(\ell(I_x)=I_x\) is an atom of \(\mu\), thus its element
\(x\) belongs to the right-hand side of \eqref{eq:equiv_A_ell}.

Conversely, let us prove the inclusion \(\supset\) in
\eqref{eq:equiv_A_ell}. Fix an atom \(A\) of \(\mu\) and
a point \(x\in\ell(A)\). We aim to show that \(\ell(A)\)
coincides with the minimal element \(I_x\) of \((\mathcal I_x^\ell,
\subset)\), whence it would follow that \(x\in\mathscr A(\ell)\).
To prove it, pick any \(I\in\mathcal I_x^\ell\). Since
\(A\setminus I\subset A\) and \(A\) is an atom of \(\mu\), we have
that either \(\mu(A\cap I)=0\) or \(\mu(A\setminus I)=0\).
This implies that either \(\ell(A)\cap I=\ell(A\cap I)=\varnothing\)
or \(\ell(A)\setminus I=\ell(A\setminus I)=\varnothing\). Since
\(x\in\ell(A)\cap I\), it necessarily holds that \(\ell(A)\setminus I
=\varnothing\), which is equivalent to \(\ell(A)\subset I\). Thanks
to the arbitrariness of \(I\in\mathcal I_x^\ell\), we conclude that
\(\ell(A)=I_x\).

Finally, it follows from the previous discussion concerning atoms
that in the family \(\mathcal I^\ell\) there can be at most
countably many atoms, thus the set appearing in the right-hand
side of \eqref{eq:equiv_A_ell} (if non-empty) can be written
as a countable union of atoms and in particular it is measurable.
Therefore, \eqref{eq:equiv_A_ell} ensures that
\(\mathscr A(\ell)\in\Sigma\), thus proving the last part of
the statement.
\end{proof}
Let us now construct a suitable family \(\{\omega_x\}_{x\in\X}\) of ultrafilters induced by a von Neumann lifting \(\ell\).
These ultrafilters \(\omega_x\) will be a key ingredient during the proof of Theorem \ref{thm:lift_sects}.
\begin{lemma}\label{lem:construct_ultrafilter}
Let \((\X,\Sigma,\mu)\) be a complete, finite measure space.
Let \(\ell\) be a von Neumann lifting of \(\mu\).
Let \(x\in\X\) be given. Then there exists an ultrafilter \(\omega_x\)
on the set \(\tilde{\mathcal I}^\ell_x\) such that
\begin{equation}\label{eq:construct_ultrafilter_claim1}
\tilde{\mathcal U}^\ell_x(I)\in\omega_x,
\quad\text{ for every }I\in\mathcal I^\ell_x.
\end{equation}
In particular, if \((\Y,\tau)\) is a Hausdorff topological space and
\(\Phi\colon\mathcal I^\ell_x\to\Y\) is any given map, then
\begin{equation}\label{eq:construct_ultrafilter_claim2}
\lim_{I\Rightarrow x}\Phi(I)=y,\;\;\text{for some }y\in\Y
\quad\Longrightarrow\quad\omega_x\text{-}\lim_{\X\setminus I}\Phi(I)=y.
\end{equation}
Moreover, the following properties are satisfied:
\begin{itemize}
\item[\(\rm i)\)] If \(x\in\mathscr A(\ell)\), then \(\omega_x\)
is the principal ultrafilter having \(\{\X\setminus I_x\}\) as its least element.
\item[\(\rm ii)\)] If \(x\in\X\setminus\mathscr A(\ell)\),
then \(\omega_x\) is a non-principal ultrafilter.
\end{itemize}
\end{lemma}
\begin{proof}
Given any \(x\in\X=\hat\X\), let us consider the filter
\(\mathcal F_x^\ell\coloneqq\mathcal F_x^{\mathcal I^\ell}\)
on \(\tilde{\mathcal I}_x^\ell\) induced by \(\mathcal I^\ell\)
as in Definition \ref{def:filter_ind_by_diff_basis}. By using
the Ultrafilter Lemma (Theorem \ref{thm:ultrafilter_lemma}), we obtain an
ultrafilter \(\omega_x\) on \(\tilde{\mathcal I}^\ell_x\) containing
\(\mathcal F_x^\ell\). Note that
\(\tilde{\mathcal U}^\ell_x(I)\in\mathcal F_x^\ell\subset\omega_x\)
for every \(I\in\mathcal I^\ell_x\) and thus
\eqref{eq:construct_ultrafilter_claim1} is satisfied.

We pass to the verification of \eqref{eq:construct_ultrafilter_claim2}.
Define \(\tilde\Phi(\X\setminus I)\coloneqq\Phi(I)\) for every
\(I\in\mathcal I_x^\ell\). Then the claimed implication
\eqref{eq:construct_ultrafilter_claim2} can be proven by
observing that, given any \(y\in\Y\), it holds that
\[
\lim_{I\Rightarrow x}\Phi(I)=y
\quad\overset{\rm (a)}\Longleftrightarrow\quad
\tilde\Phi_*\mathcal F_x^\ell\to y
\quad\overset{\rm (b)}\Longrightarrow\quad
\tilde\Phi_*\omega_x\to y
\quad\overset{\rm (c)}\Longleftrightarrow\quad
\omega_x\text{-}\lim_{\X\setminus I}\Phi(I)=y,
\]
where (a) follows from Lemma \ref{lem:equiv_I-limit}, (b) from
the inclusion \(\mathcal F_x^\ell\subset\omega_x\), (c) from
Remark \ref{rmk:equiv_ultralim}.

In order to prove i), it is sufficient to notice that
\(\{\X\setminus I_x\}=\tilde{\mathcal U}^\ell_x(I_x)\in
\mathcal F_x^\ell\subset\omega_x\).
In order to prove ii), we show that for any
\(I\in\mathcal I^\ell_x\) there exists \(S\in\omega_x\) such that
\(\X\setminus I\notin S\). To this aim, let \(I\in\mathcal I^\ell_x\) be fixed.
Given that \(x\notin\mathscr A(\ell)\), we can find some element
\(J\in\mathcal I^\ell_x\) that is strictly contained in \(I\), so that
\(\X\setminus I\notin\tilde{\mathcal U}^\ell_x(J)\in\omega_x\), thus
reaching the sought conclusion.
\end{proof}
The fact that the ultrafilters \(\omega_x\) are induced by the differentiation basis
\(\mathcal I^\ell\) entails a strong compatibility with the lifting \(\ell\) itself,
as it is evident from the following result, which will have an essential role in one
step of the proof of Theorem \ref{thm:lift_sects} (namely, in order to achieve
\eqref{eq:lift_sects_claim4}).
\begin{proposition}\label{prop:compat_lift}
Let \((\X,\Sigma,\mu)\) be a complete, finite measure space.
Let \(\ell\) be a von Neumann lifting of \(\mu\). Let us fix any
family \(\{\omega_x\}_{x\in\X}\) of ultrafilters as in Lemma
\ref{lem:construct_ultrafilter}. Then for every
function \(f\in L^\infty(\mu)\) it holds that
\begin{equation}\label{eq:compat_lift}
\ell(f)(x)=\omega_x\text{-}\lim_{\X\setminus I}\fint_I f\,\d\mu,
\quad\text{ for every }x\in\X.
\end{equation}
\end{proposition}
\begin{proof}
First of all, let us prove that the identity \eqref{eq:compat_lift}
holds for every simple function \(f\in S(\mu)\). We can write
\(f=\sum_{i=1}^n\lambda_i\1_{A_i}^\mu\), where \((\lambda_i)_{i=1}^n\subset\R\)
and \((A_i)_{i=1}^n\subset\Sigma\) is a partition of \(\X\) made of positive
\(\mu\)-measure sets. Given any \(i=1,\ldots,n\) and \(x\in\ell(A_i)\), we have
that \(\ell(A_i)\in\mathcal I^\ell_x\) and \(\fint_I f\,\d\mu=\lambda_i\)
for every \(I\in\mathcal U^\ell_x\big(\ell(A_i)\big)\), thus accordingly
\(\omega_x\text{-}\lim_{\X\setminus I}\fint_I f\,\d\mu=\lambda_i=\ell(f)(x)\).
Given that \(\bigcup_{i=1}^n\ell(A_i)=\ell\big(\bigcup_{i=1}^n A_i\big)=\X\),
the identity \eqref{eq:compat_lift} is proven for every \(f\in S(\mu)\).

Let us pass to the verification of the general statement.
Let \(f\in L^\infty(\mu)\) and \(\varepsilon>0\) be fixed.
Then we can find \(g_\varepsilon,h_\varepsilon\in S(\mu)\) such that
\(f-\varepsilon\leq g_\varepsilon\leq f\leq h_\varepsilon\leq f+\varepsilon\)
everywhere on \(\X\). By taking first the average integrals over
\(I\in\mathcal I^\ell_x\) and then the ultralimits, for any \(x\in\X\) we get
\[\begin{split}
\ell(f)(x)-\varepsilon&=\ell(f-\varepsilon)(x)\leq\ell(g_\varepsilon)(x)
=\omega_x\text{-}\lim_{\X\setminus I}\fint_I g_\varepsilon\,\d\mu
\leq\omega_x\text{-}\lim_{\X\setminus I}\fint_I f\,\d\mu
\leq\omega_x\text{-}\lim_{\X\setminus I}\fint_I h_\varepsilon\,\d\mu\\
&=\ell(h_\varepsilon)(x)\leq\ell(f+\varepsilon)(x)=\ell(f)(x)+\varepsilon,
\end{split}\]
where we used \eqref{eq:compat_lift} for simple functions.
By letting \(\varepsilon\searrow 0\), we finally obtain \eqref{eq:compat_lift} for \(f\).
\end{proof}
We are now in a position to construct the lifting of sections of a given Banach \(\B\)-bundle:
\begin{theorem}[Lifting of sections]\label{thm:lift_sects}
Let \((\X,\Sigma,\mu)\) be a complete, \(\sigma\)-finite measure space and
\(\ell\) a von Neumann lifting of \(\mu\). Let \(\B\) be a Banach space and
\(\mathbf E\) a Banach \(\B\)-bundle over \((\X,\Sigma)\).
Then there exists a map \(\ell\colon\Gamma_\infty(\mu;\mathbf E)\to\bar\Gamma_\infty(\mu;\B'')\)
such that the following properties hold:
\begin{subequations}\begin{align}
\label{eq:lift_sects_claim1}
\ell(v)(x)=J_\B\big(\hat v(x)\big),&\quad\text{ for every }
v\in\Gamma_\infty(\mu;\mathbf E)\text{ and }\mu\text{-a.e.\ }x\in\X,\\
\label{eq:lift_sects_claim2}
\ell(v+w)(x)=\ell(v)(x)+\ell(w)(x),&\quad\text{ for every }
v,w\in\Gamma_\infty(\mu;\mathbf E)\text{ and }x\in\X,\\
\label{eq:lift_sects_claim3}
\ell(f\cdot v)(x)=\ell(f)(x)\,\ell(v)(x),&\quad\text{ for every }
v\in\Gamma_\infty(\mu;\mathbf E)\text{, }f\in L^\infty(\mu)\text{, and }x\in\X,\\
\label{eq:lift_sects_claim4}
\big\|\ell(v)(x)\big\|_{\B''}\leq\ell\big(\|v(\cdot)\|
_{{\mathbf E}(\cdot)}\big)(x),
&\quad\text{ for every }v\in\Gamma_\infty(\mu;\mathbf E)\text{ and }x\in\X,
\end{align}\end{subequations}
where \(\hat\#\colon\Gamma_\infty(\mu;\mathbf E)\to\bar\Gamma
_\infty(\mu;\mathbf E)\) is the precise representative map
provided by Theorem \ref{thm:precise_repr}. In particular, the operator
\(\ell\colon\Gamma_\infty(\mu;\mathbf E)\to\bar\Gamma_\infty(\mu;\B'')\)
is linear and continuous. Moreover, in the case where \(\mathbf E\) coincides with the constant
bundle \(\B\), we can additionally require that
\begin{equation}\label{eq:lift_sects_claim_extra}
\ell(\1_\X^\mu\bar v)(x)=J_\B(\bar v),\quad\text{ for every }\bar v\in\B\text{ and }x\in\X.
\end{equation}
\end{theorem}
\begin{proof}
Thanks to Remark \ref{rmk:wlog_mu_finite}, we can assume
without loss of generality that \(\mu\) is finite.
Fix any family \(\{\omega_x\}_{x\in\X}\) of ultrafilters as in Lemma
\ref{lem:construct_ultrafilter}. Now consider an element
\(v\in\Gamma_\infty(\mu;\mathbf E)\). Given any \(x\in\X\), we have that the
family \(\big\{\fint_I J_\B\circ\hat v\,\d\mu\,:\,I\in\mathcal I^\ell_x\big\}\)
is norm-bounded in \(\B''\). In view of the fact that the weak\(^*\) topology
of \(\B''\) is Hausdorff, as well as of the compactness of its restriction to
any closed ball in \(\B''\) (which is guaranteed by Banach--Alaoglu
Theorem), the following definition is well-posed:
\begin{equation}\label{eq:constr_lift_sects}
\ell(v)(x)\coloneqq\omega_x\text{-}\lim_{\X\setminus I}\fint_I J_\B\circ\hat v
\,\d\mu\in\B'',\quad\text{ for every }x\in\X,
\end{equation}
where the ultralimit is taken with respect to the weak\(^*\) topology
of \(\B''\). Let us verify that the resulting map
\(\ell\) takes values into the space \(\bar\Gamma_\infty(\mu;\B'')\) and
fulfils the required properties.\\
{\color{blue}\textsc{Proof of \eqref{eq:lift_sects_claim1}.}}
Fix any \(v\in\Gamma_\infty(\mu;\mathbf E)\).
We know from Theorem \ref{thm:precise_repr} that for \(\mu\)-a.e.\ \(x\in\X\) the strong
\(\mathcal I^\ell\)-limit \(\lim_{I\Rightarrow x}\fint_I\hat v\,\d\mu\in\B\) exists and
coincides with \(\hat v(x)\). In particular, it holds
\[
(J_\B\circ\hat v)(x)=J_\B\bigg(\lim_{I\Rightarrow x}\fint_I\hat v\,\d\mu\bigg)=
\lim_{I\Rightarrow x}\fint_I J_\B\circ\hat v\,\d\mu
\]
with respect to the strong topology of \(\B''\), so also with respect to the weak\(^*\)
topology of \(\B''\). By recalling \eqref{eq:construct_ultrafilter_claim2},
we obtain \eqref{eq:lift_sects_claim1} and thus in particular
\(\ell(v)\in\bar\Gamma_\infty(\mu;\B'')\) for all \(v\in\Gamma_\infty(\mu;{\mathbf E})\).\\
{\color{blue}\textsc{Proof of \eqref{eq:lift_sects_claim2}.}}
Let \(v,w\in\Gamma_\infty(\mu;\mathbf E)\) be fixed. Then \eqref{eq:lin_ultralim_1}
and Theorem \ref{thm:precise_repr} give
\[\begin{split}
\ell(v)(x)+\ell(w)(x)
&=\omega_x\text{-}\lim_{\X\setminus I}\fint_I J_\B\circ\hat v\,\d\mu
+\omega_x\text{-}\lim_{\X\setminus I}\fint_I J_\B\circ\hat w\,\d\mu
=\omega_x\text{-}\lim_{\X\setminus I}\fint_I J_\B\circ(\hat v+\hat w)\,\d\mu\\
&=\ell(v+w)(x),
\end{split}\]
for every \(x\in\X\), whence \eqref{eq:lift_sects_claim2} follows.\\
{\color{blue}\textsc{Proof of \eqref{eq:lift_sects_claim4}.}}
Recall that the bidual norm function \(\B''\ni L\mapsto\|L\|_{\B''}\in\R\)
is weakly\(^*\) lower semicontinuous. This can be easily seen by
noticing that (by definition) the norm \(\|\cdot\|_{\B''}\) can be expressed as
\(\|L\|_{\B''}=\sup\big\{L(\omega)\,:\,\omega\in\B',\,\|\omega\|_{\B'}=1\big\}\)
and that \(\B''\ni L\mapsto L(\omega)\in\R\) is weakly\(^*\) continuous for every
\(\omega\in\B'\). Hence, for every \(v\in\Gamma_\infty(\mu;\mathbf E)\)
and \(x\in\X\) one has
\[\begin{split}
\big\|\ell(v)(x)\big\|_{\B''}&=\bigg\|\,\omega_x\text{-}\lim_{\X\setminus I}
\fint_I J_\B\circ\hat v\,\d\mu\,\bigg\|_{\B''}
\overset{\eqref{eq:ultra_lsc}}\leq\omega_x\text{-}\lim_{\X\setminus I}
\bigg\|\fint_I J_\B\circ\hat v\,\d\mu\,\bigg\|_{\B''}\\
&\leq\omega_x\text{-}\lim_{\X\setminus I}\fint_I
\big\|(J_\B\circ\hat v)(\cdot)\big\|_{\B''}\,\d\mu
=\omega_x\text{-}\lim_{\X\setminus I}\fint_I\|v(\cdot)\|_\B
\,\d\mu
\overset{\eqref{eq:compat_lift}}=
\ell\big(\|v(\cdot)\|_\B\big)(x).
\end{split}\]
{\color{blue}\textsc{Proof of \eqref{eq:lift_sects_claim3}.}}
As an intermediate step, we are first going to prove that:
\begin{subequations}\begin{align}
\label{eq:lift_sects_aux1}
\ell(\lambda v)(x)=\lambda\,\ell(v)(x),&\quad\text{ for every }
v\in\Gamma_\infty(\mu;\mathbf E)\text{, }\lambda\in\R\text{, and }x\in\X,\\
\label{eq:lift_sects_aux2}
\ell(\1_A^\mu\cdot v)(x)=\1_{\ell(A)}(x)\,\ell(v)(x),&\quad\text{ for every }
v\in\Gamma_\infty(\mu;\mathbf E)\text{, }A\in\Sigma\text{, and }x\in\X.
\end{align}\end{subequations}
The identity \eqref{eq:lift_sects_aux1} follows from \eqref{eq:lin_ultralim_2}.
In order to verify \eqref{eq:lift_sects_aux2}, we distinguish two cases. In the case
where \(x\in\ell(A)\), we have that \(\ell(A)\in\mathcal I^\ell_x\) and thus
accordingly it holds that
\[
\ell(\1_A^\mu\cdot v)(x)=\omega_x\text{-}\lim_{\X\setminus I}\fint_I
J_\B\circ(\1_{\ell(A)}\cdot\hat v)\,\d\mu
=\omega_x\text{-}\lim_{\X\setminus I}\fint_{I\cap\ell(A)}J_\B\circ\hat v\,\d\mu
=\ell(v)(x).
\]
Conversely, if \(x\in\X\setminus\ell(A)\), then \(\ell(\X\setminus A)\in
\mathcal I^\ell_x\) and consequently
\[
\ell(\1_A^\mu\cdot v)(x)=\omega_x\text{-}\lim_{\X\setminus I}\fint_I
J_\B\circ(\1_{\ell(A)}\cdot\hat v)\,\d\mu=
\omega_x\text{-}\lim_{\X\setminus I}\fint_{I\cap\ell(\X\setminus A)}
J_\B\circ(\1_{\ell(A)}\cdot\hat v)\,\d\mu=0.
\]
All in all, we have proven that also \eqref{eq:lift_sects_aux2} is verified.
Observe that \eqref{eq:lift_sects_claim2}, \eqref{eq:lift_sects_aux1},
and \eqref{eq:lift_sects_aux2} imply that \eqref{eq:lift_sects_claim3}
holds for every \(f\in S(\mu)\). Now fix any function
\(f\in L^\infty(\mu)\) and pick a sequence \((f_n)_{n\in\N}\subset S(\mu)\)
that converges to \(f\) in \(L^\infty(\mu)\). Therefore, for every \(x\in\X\)
one has
\[\begin{split}
\big\|\ell(f\cdot v)(x)-\ell(f_n\cdot v)(x)\big\|_{\B''}
&=\big\|\ell\big((f-f_n)\cdot v\big)(x)\big\|_{\B''}
\overset{\eqref{eq:lift_sects_claim4}}\leq
\ell\big(\big\|\big((f-f_n)\cdot v\big)(\cdot)\big\|_\B\big)(x)\\
&=\ell(|f-f_n|)(x)\,\ell\big(\|v(\cdot)\|_\B\big)(x)\longrightarrow 0,
\quad\text{ as }n\to\infty,
\end{split}\]
where we used the continuity of
\(\ell\colon L^\infty(\mu)\to\mathcal L^\infty(\Sigma)\).
Since \(\lim_n\ell(f_n)(x)=\ell(f)(x)\) as well, by letting \(n\to\infty\)
in the identity \(\ell(f_n\cdot v)(x)=\ell(f_n)(x)\,\ell(v)(x)\) we finally
obtain \eqref{eq:lift_sects_claim3}.\\
{\color{blue}\textsc{Proof of \eqref{eq:lift_sects_claim_extra}.}} Assume that \(\mathbf E\) is the constant bundle \(\B\).
Then \eqref{eq:constr_lift_sects} immediately implies that \(\ell(\1_\X^\mu\bar v)(x)=J_\B(\bar v)\) holds for every
\(\bar v\in\B\) and \(x\in\X\), thus proving \eqref{eq:lift_sects_claim_extra}.
\end{proof}
\begin{remark}\label{rmk:target_B}{\rm
Note that in the case where the Banach space \(\B\) has a predual,
the conclusion of Theorem \ref{thm:lift_sects} can be improved to:}
There exists a map \(\ell\colon\Gamma_\infty(\mu;\mathbf E)\to\bar\Gamma_\infty(\mu;\B)\)
such that \eqref{eq:lift_sects_claim2}, \eqref{eq:lift_sects_claim3},
\eqref{eq:lift_sects_claim4} hold, and \(\ell(v)(x)=\hat v(x)\) for
all \(v\in\Gamma_\infty(\mu;\mathbf E)\) and \(\mu\)-a.e.\ \(x\in\X\).
\fr\end{remark}

Some comments on Theorem \ref{thm:lift_sects} are in order. As we will see next,
the inequality in \eqref{eq:lift_sects_claim4} cannot, in general, be improved to
an equality. In the last part of this section, we will introduce a subfamily of
\(\Gamma_\infty(\mu;{\mathbf E})\) where one can require the lifting to
be norm-preserving. This has to do with the compactness of the essential range
of those mappings one would like to lift, see Definition \ref{def:ETB} and
Proposition \ref{prop:improved_lift_sects}. This kind of issue was previously
known in the setting of Lebesgue--Bochner spaces, as we will discuss in Appendix
\ref{app:lift_Leb-Boch}, see in particular Theorem \ref{thm:lift_Leb-Bochn}.
As far as we know, the weaker form of \(1\)-Lipschitz lifting we introduced in
Theorem \ref{thm:lift_sects} was not considered earlier. An application of
Theorem \ref{thm:lift_sects} will be given in Section \ref{ss:disint}.
\begin{remark}[Comparison with liftings of normed modules]{\rm
As recalled in Remark \ref{rmk:sections_NMod}, the space of sections
\(\Gamma_\infty(\mu;{\mathbf E})\) is an \(L^\infty(\mu)\)-normed
\(L^\infty(\mu)\)-module. One of the main achievements of \cite{DiMarinoLucicPasqualetto21}
is a lifting result for \(L^\infty(\mu)\)-normed \(L^\infty(\mu)\)-modules, see
\cite[Theorem 3.5]{DiMarinoLucicPasqualetto21}. Therefore, we actually have at disposal
two different notions of liftings for \(\mathscr M\coloneqq\Gamma_\infty(\mu;{\mathbf E})\).
Following the approach in \cite{DiMarinoLucicPasqualetto21}, we obtain an
\(\mathcal L^\infty(\Sigma)\)-normed \(\mathcal L^\infty(\Sigma)\)-module
\(\bar{\mathscr M}\) together with a lifting map \(\ell\colon\mathscr M\to\bar{\mathscr M}\)
that preserves the pointwise norm, meaning that \(\big|\ell(v)\big|=\ell(|v|)\) for all
\(v\in\mathscr M\). Keeping in mind that in general the notion of lifting provided by Theorem
\ref{thm:lift_sects} \emph{cannot} preserve the norm (cf.\ with Theorem \ref{thm:lift_Leb-Bochn}
and Remark \ref{rmk:lift_LB_larger}), this is telling that (with the exception
of very special cases) the fibers \(\{\bar{\mathscr M}_x\}_{x\in\X}\) of the lifted module
\(\bar{\mathscr M}\) built in \cite[Section 3.3]{DiMarinoLucicPasqualetto21}
cannot be embedded in a measurable way in any given Banach space. In particular,
this means that the lifting of \(\Gamma_\infty(\mu;{\mathbf E})\) as a normed
module is somehow unrelated to the fibers of \(\mathbf E\).
\fr}\end{remark}
Let \((\X,\Sigma,\mu)\) be a measure space, \((\Y,\sfd)\) a metric
space. Then we say that a \(\mu\)-measurable map \(\varphi\colon\X\to\Y\) is
\textbf{essentially compactly valued} provided there exist
\(N\in\Sigma\) with \(\mu(N)=0\) and a compact set \(K\subset\Y\)
such that \(\varphi(\X\setminus N)\subset K\).
\begin{definition}\label{def:ETB}
Let \((\X,\Sigma,\mu)\) be a finite measure space, \(\B\) a Banach
space. Let \(\mathbf E\) be a Banach \(\B\)-bundle over
\((\X,\Sigma)\). Then we denote by \({\rm ECV}(\mu;\mathbf E)\)
the set of all elements \(v\in\Gamma_\infty(\mu;\mathbf E)\)
such that some (thus, any) representative
\(\bar v\in\bar\Gamma_\infty(\mu;\mathbf E)\)
of \(v\) is essentially compactly valued.
\end{definition}
We point out that \({\rm ECV}(\mu;\mathbf E)\) is a linear
subspace of \(\Gamma_\infty(\mu;\mathbf E)\). More generally,
\({\rm ECV}(\mu;\mathbf E)\) is an \(L^\infty(\mu)\)-submodule
of \(\Gamma_\infty(\mu;\mathbf E)\). This can be either easily checked
from the definition or immediately deduced from the following result:
\begin{lemma}\label{lem:equiv_ETB}
Let \((\X,\Sigma,\mu)\) be a finite measure space and \(\B\) a Banach space.
Let \(\mathbf E\) be a Banach \(\B\)-bundle over \((\X,\Sigma)\). Then it holds that
\begin{equation}\label{eq:equiv_ETB}
{\rm ECV}(\mu;\mathbf E)=\Gamma_\infty(\mu;\mathbf E)\cap
{\rm cl}_{\Gamma_\infty(\mu;\B)}\big(S(\mu;\B)\big).
\end{equation}
\end{lemma}
\begin{proof}
To prove \(\subset\), let \(v\in{\rm ECV}(\mu;\mathbf E)\) be
fixed. Pick a representative \(\bar v\in\bar\Gamma_\infty(\mu;\mathbf E)\) of \(v\).
Fix \(N\in\Sigma\) with \(\mu(N)=0\) such
that the set \(\bar v(\X\setminus N)\) is totally bounded. Given any
\(\varepsilon>0\), we can thus find finitely many vectors
\(v_1,\ldots,v_n\in\B\) such that \(\bar v(\X\setminus N)\subset
\bigcup_{i=1}^n B_\varepsilon(v_i)\). Then we define \(E'_i\coloneqq
\big\{x\in\X\,:\,\|\bar v(x)-v_i\|_\B<\varepsilon\big\}
\in\Sigma\) for every \(i=1,\ldots,n\). We also set
\(E_1\coloneqq E'_1\cup N\) and \(E_i\coloneqq E'_i\setminus\bigcup
_{j<i}E'_j\) for every \(i=2,\ldots,n\). The resulting family
\((E_i)_{i=1}^n\subset\Sigma\) is a partition of \(\X\) and the element
\(w\coloneqq\sum_{i=1}^n\1_{E_i}^\mu\cdot v_i\in S(\mu;\B)\)
satisfies \(\|v-w\|_{\Gamma_\infty(\mu;\B)}\leq\varepsilon\) by
construction. Thanks to the arbitrariness of \(\varepsilon\),
we conclude that \(v\in{\rm cl}_{\Gamma_\infty(\mu;\B)}\big(
S(\mu;\B)\big)\), proving \(\subset\) in \eqref{eq:equiv_ETB}.

Conversely, fix \(v\in\Gamma_\infty(\mu;\mathbf E)\cap{\rm cl}
_{\Gamma_\infty(\mu;\B)}\big(S(\mu;\B)\big)\) and a representative \(\bar v\in\bar\Gamma_\infty(\mu;\mathbf E)\) of \(v\).
Given any \(k\in\N\), we can find \(n_k\in\N\), pairwise disjoint
sets \((E_i^k)_{i=1}^{n_k}\subset\Sigma\), and
\((v_i^k)_{i=1}^{n_k}\subset\B\) such that
\(N_k\coloneqq\X\setminus\bigcup_{i=1}^{n_k}E_i^k\in\Sigma\)
is \(\mu\)-null and \(\|\bar v(x)-v_i^k\|_\B\leq 1/k\)
for every \(i=1,\ldots,n_k\) and \(x\in E_i^k\). We claim that
\(\bar v(\X\setminus N)\) is totally bounded, where \(N\in\Sigma\)
stands for the \(\mu\)-null set \(\bigcup_{k\in\N}N_k\). Indeed,
given any \(\varepsilon>0\) and chosen \(k\in\N\) so that
\(1/k<\varepsilon\), we have that the set \(\hat v(\X\setminus N)\)
is contained in \(\bigcup_{i=1}^{n_k}B_\varepsilon(v_i^k)\),
thus showing that \(\hat v(\X\setminus N)\) is totally bounded.
Therefore, also the inclusion \(\supset\) in \eqref{eq:equiv_ETB}
is proven.
\end{proof}
\begin{proposition}\label{prop:improved_lift_sects}
Let \((\X,\Sigma,\mu)\) be a complete, \(\sigma\)-finite measure space,
\(\ell\) a von Neumann lifting of \(\mu\). Let \(\B\) be a Banach
space, \(\mathbf E\) a Banach \(\B\)-bundle over \((\X,\Sigma)\).
Then there exists a map \(\ell\colon\Gamma_\infty(\mu;\mathbf E)
\to\bar\Gamma_\infty(\mu;\B'')\) as in Theorem \ref{thm:lift_sects}
that for every \(v\in{\rm ECV}(\mu;\mathbf E)\) also satisfies
\begin{equation}\label{eq:lift_sects_extra}
\ell(v)(x)\in J_\B(\B),\qquad\big\|\ell(v)(x)\big\|_{\B''}=\ell\big(\|v(\cdot)\|_\B\big)(x),\quad\text{ for every }x\in\X.
\end{equation}
\end{proposition}
\begin{proof}
Thanks to Remark \ref{rmk:wlog_mu_finite}, we can assume
without loss of generality that the measure \(\mu\) is finite.
Moreover, we can apply Theorem \ref{thm:lift_sects}
to the constant bundle \(\B\), thus obtaining a lifting map
\(\tilde\ell\colon\Gamma_\infty(\mu;\B)\to\bar
\Gamma_\infty(\mu;\B'')\), whose restriction \(\ell\coloneqq
\tilde\ell|_{\Gamma_\infty(\mu;{\mathbf E})}\) clearly verifies
\eqref{eq:lift_sects_claim1}, \eqref{eq:lift_sects_claim2},
\eqref{eq:lift_sects_claim3}, and \eqref{eq:lift_sects_claim4}.
Moreover, thanks to \eqref{eq:lift_sects_claim_extra} we can also require that
\begin{equation}\label{eq:tilde_ell_on_const}
\tilde\ell(\1_\X^\mu\bar v)(x)=J_\B(\bar v),
\quad\text{ for every }\bar v\in\B\text{ and }x\in\X.
\end{equation}
Since \(\Gamma_\infty(\mu;\mathbf E)\ni v\mapsto\big\|
\tilde\ell(v)(\cdot)\big\|_{\B''}\in\mathcal L^\infty(\Sigma)\) and
\(\Gamma_\infty(\mu;\mathbf E)\ni v\mapsto\ell\big(\|v(\cdot)\|_\B\big)
\in\mathcal L^\infty(\Sigma)\) are continuous, we know from Lemma
\ref{lem:equiv_ETB} that in order to prove \eqref{eq:lift_sects_extra}
it suffices to show that
\begin{equation}\label{eq:aux_id_lift_on_simple}
\big\|\tilde\ell(v)(x)\big\|_{\B''}=\ell\big(\|v(\cdot)\|_\B\big)(x),
\quad\text{ for every }v\in S(\mu;\B)\text{ and }x\in\X.
\end{equation}
To this aim, fix any \(v=\sum_{i=1}^n\1_{E_i}^\mu\bar v_i
\in S(\mu;\B)\). Then for every \(x\in\X\) we can compute
\[\begin{split}
\big\|\tilde\ell(v)(x)\big\|_{\B''}
&\overset{\eqref{eq:lift_sects_claim2}}=
\bigg\|\sum_{i=1}^n\tilde\ell(\1_{E_i}^\mu\bar v_i)(x)\bigg\|_{\B''}
\overset{\eqref{eq:lift_sects_claim3}}=\bigg\|\sum_{i=1}^n
\1_{\ell(E_i)}(x)\,\tilde\ell(\1_\X^\mu\bar v_i)(x)\bigg\|_{\B''}\\
&\overset{\phantom{\eqref{eq:lift_sects_claim2}}}=\sum_{i=1}^n
\1_{\ell(E_i)}(x)\big\|\tilde\ell(\1_\X^\mu\bar v_i)(x)\big\|_{\B''}
\overset{\eqref{eq:tilde_ell_on_const}}=
\sum_{i=1}^n\1_{\ell(E_i)}(x)\|\bar v_i\|_\B\\
&\overset{\phantom{\eqref{eq:lift_sects_claim2}}}=
\ell\bigg(\sum_{i=1}^n\|\bar v_i\|_\B\1_{E_i}^\mu\bigg)(x)=\ell\big(\|v(\cdot)\|_\B\big)(x),
\end{split}\]
whence \eqref{eq:aux_id_lift_on_simple} follows. Therefore,
the statement is achieved.
\end{proof}
\section{Disintegration of vector-valued measures}\label{s:disint}
In this section, we apply Theorem \ref{thm:lift_sects} to obtain a disintegration
result (Theorem \ref{thm:disint_vect_meas}) for vector measures defined on a Polish
space, whose target is a Banach space which has the Radon--Nikod\'{y}m property. An important
feature of Polish spaces, which is a key ingredient in the proof of Theorem
\ref{thm:disint_vect_meas}, is the existence of strong liftings (cf.\ with Definition \ref{def:strong_lift}). 
\subsection{Strong liftings}\label{ss:strong_lift}
A topological space \((\X,\tau)\) is said to be a \textbf{Polish space} if there exists
a complete, separable distance \(\sfd\) on \(\X\) which induces the topology \(\tau\).
We will also say that \(\tau\) is a \textbf{Polish topology} on \(\X\). We denote by
\(\mathscr B(\X)\) the Borel \(\sigma\)-algebra of \((\X,\tau)\). Moreover, the
\textbf{support} of a given \(\sigma\)-finite Borel measure \(\mu\) on \(\X\)
is the closed set \({\rm spt}(\mu)\subset\X\) defined as
\[
{\rm spt}(\mu)\coloneqq\X\setminus\bigcup\big\{U\in\tau\;\big|\;\mu(U)=0\big\}.
\]
We denote by \(C(\X)\) the space of all continuous functions \(f\colon\X\to\R\),
while \(C_b(\X)\) stands for the space of those elements of \(C(\X)\) that are bounded.
Observe that \(C_b(\X)\subset\mathcal L^\infty\big(\mathscr B(\X)\big)\).
Moreover, the linear space \(C_b(\X)\) becomes a Banach space if endowed with the supremum norm
\(\|f\|_{C_b(\X)}\coloneqq\sup_\X|f|\). If \(\X\) is compact, then the Banach space \(C(\X)=C_b(\X)\) is separable.
\begin{definition}[Strong lifting]\label{def:strong_lift}
Let \((\X,\Sigma,\mu)\) be a complete, \(\sigma\)-finite measure space. Let \(\tau\)
be a Polish topology on \(\X\) with \(\tau\subset\Sigma\). Then we say that a given
von Neumann lifting \(\ell\) of \(\mu\) is a \textbf{strong lifting} of \(\mu\)
provided it holds that
\begin{equation}\label{eq:def_strong_lift}
U\cap{\rm spt}(\mu)\subset\ell(U),\quad\text{ for every }U\in\tau.
\end{equation}
\end{definition}

One can readily check that the condition \eqref{eq:def_strong_lift} holds if and only if
\(\ell(\pi_\mu(f))=f\) on \({\rm spt}(\mu)\) for every \(f\in C_b(\X)\),
where the map \(\ell\colon L^\infty(\mu)\to\mathcal L^\infty(\Sigma)\) is the one given by Theorem \ref{thm:lift_fcts}.
\medskip

Strong liftings exist on all Polish spaces, as proved for example in \cite{IonescuTulceaIonescuTulcea67}.
We will provide an alternative proof of the existence of strong liftings on Polish spaces in Corollary \ref{cor:strong_lifting_on_mms}.
\begin{theorem}[Existence of strong liftings]\label{thm:exist_strong_lift}
Let \((\X,\Sigma,\mu)\) be a complete, \(\sigma\)-finite measure
space. Let \(\tau\) be a Polish topology on \(\X\) with
\(\tau\subset\Sigma\). Then there exists a strong lifting
\(\ell\) of \(\mu\).
\end{theorem}

The property of being a strong lifting carries over to the lifting of sections, as we are going to discuss in the following result.
By \(C_b(\X;\B)\) we will mean the Banach space of all continuous maps \(v\colon\X\to\B\) satisfying
\(\|v\|_{C_b(\X;\B)}\coloneqq\sup_{x\in\X}\|v(x)\|_\B<+\infty\). Observe that it holds \(C_b(\X;\B)\subset\mathcal L^\infty(\mu;\B)\)
and that \(\|v\|_{\mathcal L^\infty(\mu;\B)}\leq\|v\|_{C_b(\X;\B)}\) for every \(v\in C_b(\X;\B)\).
\begin{proposition}
Let \((\X,\Sigma,\mu)\) be a complete, \(\sigma\)-finite measure space and \(\tau\) a Polish topology on \(\X\) with \(\tau\subset\Sigma\).
Let \(\ell\) be a strong lifting of \(\mu\) and \(\B\) a Banach space. Then any lifting operator \(\ell\colon L^\infty(\mu;\B)\to\mathcal L^\infty(\mu;\B'')\)
as in Theorem \ref{thm:lift_sects} satisfies the following property:
\[
\ell(\pi_\mu(v))(x)=J_\B(v(x)),\quad\text{ for every }v\in C_b(\X;\B)\text{ and }x\in{\rm spt}(\mu).
\]
\end{proposition}
\begin{proof}
Let \(v\in C_b(\X;\B)\) and \(x\in{\rm spt}(\mu)\) be fixed. Set \(\bar v\coloneqq v(x)\) for brevity.
Since \(\|v(\cdot)-\bar v\|_\B\) is a continuous function, one has that
\[\begin{split}
\big\|\ell(\pi_\mu(v))(x)-J_\B(\bar v)\big\|_{\B''}&\overset{\eqref{eq:lift_sects_claim_extra}}=
\big\|\ell(\pi_\mu(v))(x)-\ell(\1_\X^\mu\bar v)(x)\big\|_{\B''}\\
&\overset{\phantom{\eqref{eq:lift_sects_claim_extra}}}=\big\|\ell(\pi_\mu(v)-\1_\X^\mu\bar v)(x)\big\|_{\B''}
\overset{\eqref{eq:lift_sects_claim4}}\leq\ell\big(\|\pi_\mu(v)(\cdot)-\bar v\|_\B\big)(x)\\
&\overset{\phantom{\eqref{eq:lift_sects_claim_extra}}}=\ell\big(\pi_\mu\big(\|v(\cdot)-\bar v\|_\B\big)\big)(x)
=\|v(x)-\bar v\|_\B=0.
\end{split}\]
This shows the validity of the statement.
\end{proof}
\subsection{Disintegration of RNP-valued measures of bounded variation}\label{ss:disint}
We begin with the relevant definition of the disintegration of a vector-valued measure of bounded variation.
\begin{definition}[Disintegration of a vector-valued measure]\label{def:disint}
Let \((\X,\Sigma_\X)\) and \((\Y,\Sigma_\Y)\) be two measurable spaces. Let \(\varphi\colon\X\to\Y\) be a measurable map. Let \(\B\) be
a Banach space and \(\Omega\) a \(\B\)-valued measure on \((\X,\Sigma_\X)\) having bounded variation. Then we say that a family
\(\{\Omega_y\}_{y\in\Y}\) is a \textbf{disintegration of \(\Omega\) along \(\varphi\)} provided the following conditions are verified:
\begin{itemize}
\item[\(\rm i)\)] Each \(\Omega_y\) is a \(\B\)-valued measure on \((\X,\Sigma_\X)\) with bounded variation.
Moreover, it holds that \(\|\Omega_y\|_\B(\X)=1\), \(\{y\}\in\Sigma_\Y\), and \(\|\Omega_y\|_\B\big(\X\setminus\varphi^{-1}(\{y\})\big)=0\)
for \(\varphi_\#\|\Omega\|_\B\)-a.e.\ \(y\in\Y\).
\item[\(\rm ii)\)] Given any function \(f\in\mathcal L^\infty(\Sigma_\X)\), the mapping \(\Y\ni y\mapsto\int f\,\d\Omega_y\in\B\)
is strongly measurable, where the domain \((\Y,\Sigma_\Y)\) is endowed with the pushforward measure \(\varphi_\#\|\Omega\|_\B\).
\item[\(\rm iii)\)] It holds that
\begin{equation}\label{eq:def_disint}
\int f\,\d\Omega=\int\!\!\!\int f\,\d\Omega_y\,\d(\varphi_\#\|\Omega\|_\B)(y),\quad\text{ for every }f\in\mathcal L^\infty(\Sigma_\X).
\end{equation}
\end{itemize}
\end{definition}

Observe that \(\big\|\int f\,\d\Omega_y\big\|_\B\leq\|f\|_{\mathcal L^\infty(\Sigma_\X)}\) for \(\varphi_\#\|\Omega\|_\B\)-a.e.\ \(y\in\Y\),
thus \(\Y\ni y\mapsto\big\|\int f\,\d\Omega_y\big\|_\B\) belongs to \(\mathcal L^1(\varphi_\#\|\Omega\|_\B)\)
and accordingly the right-hand side of \eqref{eq:def_disint} is well-posed.
\begin{remark}{\rm
By virtue of the Monotone Class Theorem, the requirements ii) and iii) in Definition \ref{def:disint} can be replaced by the following two conditions:
\begin{itemize}
\item[\(\rm ii')\)] Given any set \(E\in\Sigma_\X\), the mapping \(\Y\ni y\mapsto\Omega_y(E)\in\B\) is strongly measurable,
where the domain \((\Y,\Sigma_\Y)\) is endowed with the pushforward measure \(\varphi_\#\|\Omega\|_\B\).
\item[\(\rm iii')\)] It holds that
\begin{equation}\label{eq:def_disint_sets}
\Omega(E)=\int\Omega_y(E)\,\d(\varphi_\#\|\Omega\|_\B)(y),\quad\text{ for every }E\in\Sigma_\X.
\end{equation}
\end{itemize}
Moreover, in the case where \(\X\) is a Polish space and \(\Sigma_\X=\mathscr B(\X)\), one can equivalently require that the mapping
\(\Y\ni y\mapsto\int f\,\d\Omega_y\in\B\) is strongly measurable for every \(f\in C_b(\X)\) and that the defining identity
\eqref{eq:def_disint} is satisfied for every \(f\in C_b(\X)\).
\fr}\end{remark}

Whenever they exist, disintegrations are (essentially) unique; see Lemma \ref{lem:uniq_disint} below.
Before passing to the verification of this statement, let us recall a useful technical fact:
\begin{remark}\label{rmk:good_family_approx}{\rm
Let \((\X,\tau)\) be a Polish space. Let \(\B\) be a Banach space and \(\Omega\) a
\(\B\)-valued measure on \(\big(\X,\mathscr B(\X)\big)\) having bounded variation. Fix
any countable, dense subset \(D\) of \(\X\). Define \(\mathcal U\) as the countable family
of all finite unions of open balls in \(\X\) having centre in \(D\) and rational radius.
Then we claim that for any \(E\in\mathscr B(\X)\) there exists
\((U_n)_{n\in\N}\subset\mathcal U\) such that
\[
\big\|\Omega(U_n)-\Omega(E)\big\|_\B\to 0,\quad\text{ as }n\to\infty.
\]
Indeed, for any Borel set \(E\subset\X\) we may estimate
\[\begin{split}
\inf_{U\in\mathcal U}\big\|\Omega(U)-\Omega(E)\big\|_\B
&=\inf_{U\in\mathcal U}\big\|\Omega(U\setminus E)-\Omega(E\setminus U)\big\|_\B\leq
\inf_{U\in\mathcal U}\big(\|\Omega(U\setminus E)\|_\B+\|\Omega(E\setminus U)\|_\B\big)\\
&\leq\inf_{U\in\mathcal U}\big(\|\Omega\|_\B(U\setminus E)+\|\Omega\|_\B(E\setminus U)\big)
=\inf_{U\in\mathcal U}\|\Omega\|_\B(U\Delta E)\overset{\star}=0,
\end{split}\]
where the starred identity follows from the regularity of the Borel measure \(\|\Omega\|_\B\).
\fr}\end{remark}
\begin{lemma}[Uniqueness of the disintegration]\label{lem:uniq_disint}
Let \((\X,\tau_\X)\), \((\Y,\tau_\Y)\) be Polish spaces and let \(\varphi\colon\X\to\Y\)
be a Borel measurable map. Let \(\B\) be a Banach space and let \(\Omega\) be a
\(\B\)-valued measure on \(\big(\X,\mathscr B(\X)\big)\) having bounded variation.
Suppose that \(\{\Omega_y\}_{y\in\Y}\) and \(\{\Omega'_y\}_{y\in\Y}\) are disintegrations
of \(\Omega\) along \(\varphi\). Then it holds that \(\Omega_y=\Omega'_y\) for
\(\varphi_\#\|\Omega\|_\B\)-a.e.\ \(y\in\Y\).
\end{lemma}
\begin{proof}
We argue by contradiction: suppose there exists a Borel set \(P_1\subset\Y\) such that
\(\Omega_y\neq\Omega'_y\) for every \(y\in P_1\) and \(\varphi_\#\|\Omega\|_\B(P_1)>0\).
Letting \(\mathcal U\) be as in Remark \ref{rmk:good_family_approx}, we can find
\(U\in\mathcal U\) and \(P_2\subset P_1\) Borel such that \(\Omega_y(U)\neq\Omega'_y(U)\)
for every \(y\in P_2\) and \(\varphi_\#\|\Omega\|_\B(P_2)>0\). Hence, we can find
\(\lambda>0\) and a Borel set \(P_3\subset P_2\) such that
\(\big\|\Omega_y(U)-\Omega'_y(U)\big\|_\B\geq\lambda\) for every \(y\in P_3\) and
\(\varphi_\#\|\Omega\|_\B(P_3)>0\). For brevity, we set \(v(y)\coloneqq\Omega_y(U)\)
and \(v'(y)\coloneqq\Omega'_y(U)\) for every \(y\in P_3\). Since \(v\) and \(v'\) are
strongly measurable, we can find a closed, separable subset \({\rm Z}\) of \(\B\) and
a Borel set \(N\subset P_3\) such that \(\varphi_\#\|\Omega\|_\B(N)=0\) and
\(v(y),v'(y)\in{\rm Z}\) for every \(y\in P_3\setminus N\). Then consider the finite Borel
measure \(\mu\coloneqq(v,v')_\#\big({(\varphi_\#\|\Omega\|_\B)|}_{P_3\setminus N}\big)\) on
\({\rm Z}\times{\rm Z}\). Given that \({\rm spt}(\mu)\) is not contained in the diagonal
\(\{(z,z)\,:\,z\in{\rm Z}\}\), we can find distinct points \(z,z'\in{\rm Z}\) such that
\((z,z')\in{\rm spt}(\mu)\). Choose \(r>0\) so that \(\|z-z'\|_\B>2r\) and \(r<\lambda/2\).
Call \(B\) (resp.\ \(B'\)) the closed ball in \(\rm Z\) of centre \(z\) (resp.\ \(z'\))
and radius \(r\). Then \(B\cap B'=\varnothing\) and \(\mu(B\times B')>0\), thus
\[
P\coloneqq\big\{y\in P_3\setminus N\;\big|\;v(y)\in B,\,v'(y)\in B'\big\}
\]
is a Borel set having positive \(\varphi_\#\|\Omega\|_\B\)-measure.
Define \(E\coloneqq U\cap\varphi^{-1}(P)\). Therefore, one has
\[\begin{split}
\bigg\|\frac{\Omega(E)}{\varphi_\#\|\Omega\|_\B(P)}-z\bigg\|_\B
&=\bigg\|\frac{1}{\varphi_\#\|\Omega\|_\B(P)}\int\Omega_y\big(U\cap\varphi^{-1}(P)\big)
\,\d(\varphi_\#\|\Omega\|_\B)(y)-z\bigg\|_\B\\
&=\bigg\|\frac{1}{\varphi_\#\|\Omega\|_\B(P)}\int\Omega_y\big(U\cap\varphi^{-1}(\{y\}\cap P)\big)
\,\d(\varphi_\#\|\Omega\|_\B)(y)-z\bigg\|_\B\\
&=\bigg\|\fint_P\Omega_y(U)\,\d(\varphi_\#\|\Omega\|_\B)(y)-z\bigg\|_\B
=\bigg\|\fint_P\big(v(y)-z\big)\,\d(\varphi_\#\|\Omega\|_\B)(y)\bigg\|_\B\\
&\leq\fint_P\|v(y)-z\|_\B\,\d(\varphi_\#\|\Omega\|_\B)(y)\leq r,
\end{split}\]
which implies that \(\Omega(E)/(\varphi_\#\|\Omega\|_\B)(P)\in B\). The same computations,
with \((z',\Omega'_y)\) in place of \((z,\Omega_y)\), show that also
\(\Omega(E)/(\varphi_\#\|\Omega\|_\B)(P)\in B'\), thus leading to a contradiction.
\end{proof}

This section ends with an existence result, see Theorem \ref{thm:disint_vect_meas}. First of all, an auxiliary fact:
\begin{remark}\label{rmk:local_constr_disint}{\rm
We point out that disintegrations can be obtained by `patching together' local disintegrations, as we are going
to describe. Let \((\X_n)_{n\in\N}\subset\Sigma_\X\) be a family of pairwise disjoint subsets of \(\X\) with
\(\|\Omega\|_\B\big(\X\setminus\bigcup_n\X_n\big)=0\). Suppose to have a disintegration \(\{\Omega^n_y\}_{y\in\Y}\)
of \(\Omega|_{\X_n}\) along \(\varphi|_{\X_n}\) for every \(n\in\N\). Fix a measurable version \(\alpha_n\colon\Y\to[0,+\infty)\)
of \(\frac{\d\varphi_\#{\|{\Omega|}_{\X_n}\|_\B}}{\d\varphi_\#\|\Omega\|_\B}\). Set
\[
\Omega_y(E)\coloneqq\sum_{n\in\N}\alpha_n(y)\Omega^n_y(E)\in\B,\quad\text{ for every }y\in\Y\text{ and }E\in\Sigma_\X.
\]
It is straightforward to check that \(\{\Omega_y\}_{y\in\Y}\)
is a disintegration of \(\Omega\) along \(\varphi\). For instance,
\[\begin{split}
\Omega(E)&=\sum_{n\in\N}{\Omega|}_{\X_n}(E)=\sum_{n\in\N}\int
\Omega^n_y(E)\,\d(\varphi_\#\|{\Omega|}_{\X_n}\|_\B)(y)=\sum_{n\in\N}
\int\alpha_n(y)\Omega^n_y(E)\,\d(\varphi_\#\|\Omega\|_\B)(y)\\
&=\int\Omega_y(E)\,\d(\varphi_\#\|\Omega\|_\B)(y),\quad
\text{ for every }E\in\Sigma_\X,
\end{split}\]
where the last identity can be justified by using the Dominated
Convergence Theorem.
\fr}\end{remark}

In the case of non-negative measures, a proof (via strong liftings) of the Disintegration Theorem can be found in
\cite{IonescuTulcea65,Hoffmann-Jorgensen71}. In the next result, we adapt their strategy (using Theorem \ref{thm:lift_sects})
to prove existence of the disintegration for RNP-valued measures of bounded variation.
\begin{theorem}[Existence of disintegration for vector-valued measures]\label{thm:disint_vect_meas}
Let \((\X,\tau_\X)\), \((\Y,\tau_\Y)\) be Polish spaces
and \(\varphi\colon\X\to\Y\) a Borel measurable map. Let \(\B\)
be a Banach space having the Radon--Nikod\'{y}m property and
\(\Omega\) a \(\B\)-valued measure on \(\big(\X,\mathscr B(\X)\big)\)
having bounded variation. Then there exists a disintegration
\(\{\Omega_y\}_{y\in\Y}\) of \(\Omega\) along \(\varphi\).
\end{theorem}
\begin{proof}
Thanks to the inner regularity of \(\|\Omega\|_\B\) and to Lusin's Theorem, we can find a countable family
\((K_n)_{n\in\N}\) of pairwise disjoint compact sets in \(\X\) such that \(\varphi|_{K_n}\colon K_n\to\Y\)
is continuous for every \(n\in\N\) and \(\|\Omega\|_\B\big(\X\setminus\bigcup_n K_n\big)=0\). In view of Remark
\ref{rmk:local_constr_disint}, in order to achieve the statement it is sufficient to disintegrate each \(\Omega|_{K_n}\) along
\(\varphi|_{K_n}\). Therefore, we may assume without loss of generality that \(\X\) is compact and \(\varphi\) is continuous.\\
{\color{blue}\textsc{Step 1: Construction of \(\Omega_y\).}}
Let \(\Sigma\) be the completion of \(\mathscr B(\Y)\) under \(\varphi_\#\|\Omega\|_\B\). Fix a strong lifting \(\ell\)
of \(\varphi_\#\|\Omega\|_\B\), which exists by Theorem \ref{thm:exist_strong_lift}. For any \(f\in C(\X)\), it holds that
\begin{equation}\label{eq:disint_vv_aux1}
\|\varphi_\#(f\Omega)\|_\B\overset{\eqref{eq:tv_pushfrwd}}\leq\varphi_\#\|f\Omega\|_\B\overset{\eqref{eq:tv_restr}}=
\varphi_\#(|f|\|\Omega\|_\B)\leq\|f\|_{C(\X)}\varphi_\#\|\Omega\|_\B.
\end{equation}
Hence, we can consider the Radon--Nikod\'{y}m derivative
\(\frac{\d\varphi_\#(f\Omega)}{\d\varphi_\#\|\Omega\|_\B}\in L^\infty(\varphi_\#\|\Omega\|_\B;\B)\).
Consider the lifting \(\ell\colon L^\infty(\varphi_\#\|\Omega\|_\B;\B)\to\mathcal L^\infty(\varphi_\#\|\Omega\|_\B;\B'')\) given by
Theorem \ref{thm:lift_sects}. For \(y\in\Y\), we set
\[
T_y(f)\coloneqq\ell\bigg(\frac{\d\varphi_\#(f\Omega)}{\d\varphi_\#\|\Omega\|_\B}\bigg)(y)\in\B'',\quad\text{ for every }f\in C(\X).
\]
The resulting map \(T_y\colon C(\X)\to\B''\) is linear by construction. Also, for any \(f\in C(\X)\) one has
\begin{equation}\label{eq:disint_vv_aux1bis}
\|T_y(f)\|_{\B''}\overset{\eqref{eq:lift_sects_claim4}}\leq
\ell\bigg(\bigg\|\frac{\d\varphi_\#(f\Omega)}{\d\varphi_\#\|\Omega\|_\B}(\cdot)\bigg\|_\B\bigg)(y)
\overset{\eqref{eq:key_ineq_density}}\leq\ell\bigg(\frac{\d\|\varphi_\#(f\Omega)\|_\B}{\d\varphi_\#\|\Omega\|_\B}\bigg)(y)
\overset{\eqref{eq:disint_vv_aux1}}\leq\|f\|_{C(\X)}.
\end{equation}
Being \(\big(C(\X),\|\cdot\|_{C(\X)}\big)\) separable, we can pick a countable dense subset \(\mathscr C\) of \(C(\X)\).
Thanks to \eqref{eq:lift_sects_claim1}, we can find \(N\in\Sigma\) with \(\varphi_\#\|\Omega\|_\B(N)=0\) such that
\(T_y(f)\in\B\) for every \(y\in\Y\setminus N\) and \(f\in\mathscr C\), where we identify \(\B\) with its image \(J_\B(\B)\subset\B''\).
Since each \(T_y\) is linear and continuous, we deduce that \(T_y(f)\in\B\) for every \(y\in\Y\setminus N\) and \(f\in C(\X)\).
Now fix \(y\in\Y\setminus N\). Given that the linear continuous operator \(T_y\colon C(\X)\to\B\) is weakly compact by Remark
\ref{rmk:RNP_no_copy_c_0}, we know from \cite[Theorems VI.2.1 and VI.2.5]{DiestelUhl77} (see also \cite[Theorem VI.7.3]{DunfordSchwartz58})
that there exists a bounded \(\B\)-valued measure \(\Omega_y\) on \(\big(\X,\mathscr B(\X)\big)\) such that \(\int f\,\d\Omega_y=T_y(f)\)
for all \(f\in C(\X)\).\\
{\color{blue}\textsc{Step 2: \(\Omega_y\) has bounded variation.}} Given that
\(C(\X)\ni f\mapsto\ell\big(\frac{\d\varphi_\#(f\|\Omega\|_\B)}{\d\varphi_\#\|\Omega\|_\B}\big)(y)\in\R\)
is a linear \(1\)-Lipschitz operator (\emph{i.e.}, it is an element of \(C(\X)'\) with norm at most \(1\)),
by using the Riesz--Markov--Kakutani Representation Theorem we obtain a Borel measure \(\mu_y\geq 0\) on \(\X\)
such that \(\mu_y(\X)\leq 1\) and \(\int f\,\d\mu_y=\ell\big(\frac{\d\varphi_\#(f\|\Omega\|_\B)}{\d\varphi_\#\|\Omega\|_\B}\big)(y)\)
for every \(f\in C(\X)\). The operator
\[
S_y\colon\pi_{\mu_y}(C(\X))\to\B,\qquad S_y(\pi_{\mu_y}(f))\coloneqq T_y(f),\quad\text{ for every }f\in C(\X),
\]
is well-defined, is linear, and satisfies \(\big\|S_y(\pi_{\mu_y}(f))\big\|_\B\leq\int|f|\,\d\mu_y\) by \eqref{eq:disint_vv_aux1} and
\eqref{eq:disint_vv_aux1bis}. Hence, since \(\pi_{\mu_y}(C(\X))\) is dense in \(L^1(\mu_y)\), we can uniquely extend \(S_y\) to a
linear and continuous operator \(S_y\colon L^1(\mu_y)\to\B\). By using \cite[Theorem VI.3.3 and Corollary VI.3.7]{DiestelUhl77},
we deduce that the \(\B\)-valued measure \(\Omega_y\) has bounded variation. We also claim that
\begin{equation}\label{eq:disint_vv_aux2}
\|\Omega_y\|_\B(E)\leq\mu_y(E),\quad\text{ for every }E\in\mathscr B(\X).
\end{equation}
In view of Remark \ref{rmk:char_variation}, to prove \eqref{eq:disint_vv_aux2} amounts to checking that \(\|\Omega_y(E)\|_\B\leq\mu_y(E)\)
holds for all \(E\in\mathscr B(\X)\). To this aim, pick a sequence \((f_n)_{n\in\N}\subset C(\X)\) such that \(f_n\geq 0\) for all
\(n\in\N\) and \(\lim_n\int|f_n-\1_E|\,\d(\|\Omega_y\|_\B+\mu_y)=0\). Since \(\big\|\int f_n\,\d\Omega_y\big\|_\B\leq\int f_n\,\d\mu_y\)
for every \(n\in\N\) and
\[
\bigg\|\int f_n\,\d\Omega_y-\Omega_y(E)\bigg\|_\B\leq\int|f_n-\1_E|\,\d\|\Omega_y\|_\B\to 0,\quad\text{ as }n\to\infty,
\]
we conclude that \(\|\Omega_y(E)\|_\B=\lim_{n\to\infty}\big\|\int f_n\,\d\Omega_y\big\|_\B\leq\lim_{n\to\infty}\int f_n\,\d\mu_y=\mu_y(E)\), as desired.\\
{\color{blue}\textsc{Step 3: \(\Omega_y\) is concentrated on \(\varphi^{-1}(\{y\})\).}}
Next we show that \(\|\Omega_y\|_\B\big(\X\setminus\varphi^{-1}(\{y\})\big)=0\) for \(\varphi_\#\|\Omega\|_\B\)-a.e.\ \(y\in\Y\);
the fact that the von Neumann lifting \(\ell\) is strong comes into play only now. Fix any point \(y\in{\rm spt}(\varphi_\#\|\Omega\|_\B)\setminus N\)
and a complete, separable distance \(\sfd_\Y\) on \(\Y\) that induces \(\tau_\Y\).
Define \(\bar g_y\coloneqq\max\{1-\sfd_\Y(y,\cdot),0\}\in C_b(\Y)\) and
\(g_y\coloneqq\pi_{\varphi_\#\|\Omega\|_\B}(\bar g_y)\in L^\infty(\varphi_\#\|\Omega\|_\B)\). Notice that \(\bar g_y<1\) on \(\Y\setminus\{y\}\)
and \(\ell(g_y)(y)=1\). Then, since \(\bar g_y\circ\varphi\in C(\X)\), we can compute
\[\begin{split}
\int(\bar g_y\circ\varphi)f\,\d\mu_y&=\ell\bigg(\frac{\d\varphi_\#\big((\bar g_y\circ\varphi)f\|\Omega\|_\B\big)}{\d\varphi_\#\|\Omega\|_\B}\bigg)(y)
=\ell\bigg(\frac{\d\varphi_\#(f\|\Omega\|_\B)}{\d\varphi_\#\|\Omega\|_\B}\,g_y\bigg)(y)\\
&=\ell(g_y)(y)\int f\,\d\mu_y=\int f\,\d\mu_y,\quad\text{ for every }f\in C(\X).
\end{split}\]
Hence, \(\bar g_y\circ\varphi=1\) holds \(\mu_y\)-a.e., which forces
\(\|\Omega_y\|_\B\big(\X\setminus\varphi^{-1}(\{y\})\big)\leq\mu_y\big(\X\setminus\varphi^{-1}(\{y\})\big)=0\).
\\
{\color{blue}\textsc{Step 4: \(\{\Omega_y\}_{y\in\Y}\) is the disintegration of \(\Omega\).}}
In order to conclude, it remains to prove that \(\{\Omega_y\}_{y\in\Y}\) is the disintegration of \(\Omega\) along \(\varphi\).
The map \(\Y\ni y\mapsto\int f\,\d\Omega_y\in\B\) is strongly measurable for every \(f\in C(\X)\) by construction.
Moreover, for any \(f\in C(\X)\) we have that
\begin{equation}\label{eq:disint_vv_aux3}\begin{split}
\int f\,\d\Omega&=(f\Omega)(\X)=\varphi_\#(f\Omega)(\Y)=\int\frac{\d\varphi_\#(f\Omega)}{\d\varphi_\#\|\Omega\|_\B}\,\d\varphi_\#\|\Omega\|_\B
=\int T_y(f)\,\d\varphi_\#\|\Omega\|_\B\\
&=\int\!\!\!\int f\,\d\Omega_y\,\d(\varphi_\#\|\Omega\|_\B)(y),
\end{split}\end{equation}
whence \eqref{eq:def_disint} follows.
Finally, we claim that \(\|\Omega_y\|_\B(\X)=1\) for \(\varphi_\#\|\Omega\|_\B\)-a.e.\ \(y\in\Y\). Indeed,
\[\begin{split}
\sum_{i=1}^n\|\Omega(E_i)\|_\B
&=\sum_{i=1}^n
\bigg\|\int\Omega_y(E_i)\,\d(\varphi_\#\|\Omega\|_\B)(y)\bigg\|_\B\leq\sum_{i=1}^n\int\|\Omega_y\|_\B(E_i)\,\d(\varphi_\#\|\Omega\|_\B)(y)\\
&=\int\|\Omega_y\|_\B(\X)\,\d(\varphi_\#\|\Omega\|_\B)(y)\overset{\star}\leq\varphi_\#\|\Omega\|_\B(\Y)=\|\Omega\|_\B(\X),
\end{split}\]
holds whenever \((E_i)_{i=1}^n\subset\mathscr B(\X)\) is a partition of \(\X\); the starred inequality follows from the fact that
\(\|\Omega_y\|_\B(\X)\leq\mu_y(\X)\leq 1\) holds for \(\varphi_\#\|\Omega\|_\B\)-a.e.\ \(y\in\Y\). By passing to the supremum over
\((E_i)_{i=1}^n\), we thus obtain that \(\|\Omega\|_\B(\X)=\int\|\Omega_y\|_\B(\X)\,\d(\varphi_\#\|\Omega\|_\B)(y)\), which in turn
forces the identity \(\|\Omega_y\|_\B(\X)=1\) for \(\varphi_\#\|\Omega\|_\B\)-a.e.\ \(y\in\Y\), as claimed. The proof is complete.
\end{proof}
\appendix
\section{From Lebesgue Differentiation Theorem to liftings}\label{app:Leb_to_lift}
In this appendix, we show how to get the existence of von Neumann liftings in a quite vast class of measure spaces, containing
for instance all Polish spaces. First, we recall the concept of a lower density (Definition \ref{def:lower_density}) and how it
induces a lifting (Lemma \ref{lem:from_lower_dens_to_lift}). Then we show that each differentiation basis verifying the Lebesgue
Differentiation Theorem induces a lower density (Theorem \ref{thm:from_Leb_to_lift}). Finally, we provide (under rather mild
assumptions) a quick way to obtain a lower density verifying the Lebesgue Differentiation Theorem and thus to prove the existence
of von Neumann liftings (Theorem \ref{thm:exist_lift_separable}), which become automatically strong
liftings in the case of measures over a Polish space (Corollary \ref{cor:strong_lifting_on_mms}).
\medskip

We begin with the notion of a lower density, whose definition is taken from \cite[341C]{Fremlin3}:
\begin{definition}[Lower density]\label{def:lower_density}
Let \((\X,\Sigma,\mu)\) be a measure space. Then by a
\textbf{lower density} of \(\mu\) we mean a map
\(\phi\colon\Sigma\to\Sigma\) satisfying the following properties:
\begin{itemize}
\item[\(\rm i)\)] \(\phi(\varnothing)=\varnothing\), \(\phi(\X)=\X\), and
\(\phi(E\cap F)=\phi(E)\cap\phi(F)\) for every \(E,F\in\Sigma\).
\item[\(\rm ii)\)] \(\phi(E)=\phi(F)\) for every \(E,F\in\Sigma\)
such that \(\mu(E\Delta F)=0\).
\item[\(\rm iii)\)] \(\mu\big(E\Delta\phi(E)\big)=0\) for every
\(E\in\Sigma\).
\end{itemize}
\end{definition}
Observe that each von Neumann lifting is a lower density.
Moreover, a lower density of \(\mu\) is a von Neumann lifting of \(\mu\) if and only if
\(\phi(E\cup F)=\phi(E)\cup\phi(F)\) for every \(E,F\in\Sigma\).
\medskip

The following key result states that by suitably `enlarging' the sets
appearing in the range of a lower density of a complete measure
\(\mu\), one can obtain a von Neumann lifting of \(\mu\).
We sketch the proof of this fact, expanding the argument in \cite[page 1136]{StraussMacherasMusial02}.
\begin{lemma}[From lower density to lifting]\label{lem:from_lower_dens_to_lift}
Let \((\X,\Sigma,\mu)\) be a complete measure space and let \(\phi\) be a lower density of \(\mu\).
Then there exists a von Neumann lifting \(\ell\) of \(\mu\) such that
\begin{equation}\label{eq:from_ld_to_lift}
\phi(E)\subset\ell(E)\subset\X\setminus\phi(\X\setminus E),\quad\text{ for every }E\in\Sigma.
\end{equation}
\end{lemma}
\begin{proof}[Sketch of proof]
Given any \(x\in\X\), it is easy to check that \(\mathcal F_x\) is a filter on \(\X\), where we set
\[
\mathcal F_x\coloneqq\big\{F\in 2^\X\;\big|\;\exists E\in\Sigma:\;x\in\phi(E),\,E\subset F\big\}.
\]
By applying the Ultrafilter Lemma, we obtain an ultrafilter \(\omega_x\) on \(\X\) containing \(\mathcal F_x\).
Then we define the map \(\ell\colon\Sigma\to\Sigma\) as \(\ell(E)\coloneqq\{x\in\X\,:\,E\in\omega_x\}\) for every
\(E\in\Sigma\). For any \(E\in\Sigma\), we have that \(\phi(E)\in\mathcal F_x\subset\omega_x\) holds for every \(x\in\phi(E)\).
Moreover, if \(x\in\ell(E)\), then \(E\in\omega_x\), so that \(\X\setminus E\notin\omega_x\) and accordingly
\(\X\setminus E\notin\mathcal F_x\), which implies that \(x\notin\phi(\X\setminus E)\). This shows that
\(\phi(E)\subset\ell(E)\subset\X\setminus\phi(\X\setminus E)\), whence \eqref{eq:from_ld_to_lift} follows.
Let us check that \(\ell\) is a von Neumann lifting of \(\mu\). If \(N\in\Sigma\) satisfies \(\mu(N)=0\), then
\(\mu\big((\X\setminus N)\Delta\X\big)=0\) and thus \(\phi(\X\setminus N)=\phi(\X)=\X\), so that the second inclusion
in \eqref{eq:from_ld_to_lift} implies that \(\ell(N)=\varnothing\). Moreover, given any \(E\in\Sigma\), we have
that \(\mu\big(E\Delta\phi(E)\big)=0=\mu\big(E\Delta(\X\setminus\phi(\X\setminus E))\big)\), thus \eqref{eq:from_ld_to_lift}
and the completeness of \(\mu\) ensure that \(\ell(E)\in\Sigma\) and that \(\mu\big(E\Delta\ell(E)\big)=0\). Finally, it is
easy to show that \(\ell\) is a Boolean homomorphism as a consequence of the fact that \(\omega_x\) is an ultrafilter.
\end{proof}
The following well-known result states that every differentiation basis for which the Lebesgue Differentiation Theorem holds
induces a lower density, just taking the density points.
\begin{theorem}[From Lebesgue Differentiation Theorem to lifting]
\label{thm:from_Leb_to_lift}
Let \((\X,\Sigma,\mu)\) be a complete, finite measure space, and \(\mathcal I\) a differentiation basis on \((\X,\Sigma,\mu)\) such that
\begin{equation}\label{eq:from_Leb_to_lift1}
D_{\mathcal I}(E)\in\Sigma,\qquad\mu\big(E\Delta D_{\mathcal I}(E)\big)=0,\quad\text{ for every }E\in\Sigma,
\end{equation}
where the set \(D_{\mathcal I}(E)\subset\X\) of \textbf{density points} of \(E\) with respect to \(\mathcal I\) is defined as
\begin{equation}\label{eq:from_Leb_to_lift2}
D_{\mathcal I}(E)\coloneqq\bigg\{x\in\hat\X\;\bigg|\;\lim_{I\Rightarrow x}\frac{\mu(E\cap I)}{\mu(I)}=1\bigg\}.
\end{equation}
Then there exists a von Neumann lifting \(\ell\) of \(\mu\) such that
\[
D_{\mathcal I}(E)\subset\ell(E)\subset\X\setminus D_{\mathcal I}(\X\setminus E),\quad\text{ for every }E\in\Sigma.
\]
\end{theorem}
\begin{proof}
By virtue of Lemma \ref{lem:from_lower_dens_to_lift}, to prove the statement
it suffices to show that \(D_{\mathcal I}\colon\Sigma\to\Sigma\) is a lower density of
\(\mu\). Item iii) of Definition \ref{def:lower_density} is exactly
\eqref{eq:from_Leb_to_lift1}. Moreover, we immediately deduce from the definition
\eqref{eq:from_Leb_to_lift2} that \(D_{\mathcal I}(\varnothing)=\varnothing\), that
\(D_{\mathcal I}(\X)=\X\), and that item ii) of Definition \ref{def:lower_density} holds.
Finally, in order to prove item i) of Definition \ref{def:lower_density}, it is only left
to check \(D_{\mathcal I}(E\cap F)=D_{\mathcal I}(E)\cap D_{\mathcal I}(F)\) for all
\(E,F\in\Sigma\). On the one hand, using that \(E\cap F\subset E\) it is immediate to see
that \(D_{\mathcal I}(E\cap F)\subset D_{\mathcal I}(E)\), and similarly
\(D_{\mathcal I}(E\cap F)\subset D_{\mathcal I}(F)\), thus the inclusion
\(D_{\mathcal I}(E\cap F)\subset D_{\mathcal I}(E)\cap D_{\mathcal I}(F)\) is proven.
On the other hand, for any \(I\in\mathcal I\) we have
\[
\frac{\mu(E\cap F\cap I)}{\mu(I)}=\frac{\mu(E\cap I)}{\mu(I)}+\frac{\mu(F\cap I)}{\mu(I)}
-\frac{\mu\big((E\cup F)\cap I\big)}{\mu(I)}\geq\frac{\mu(E\cap I)}{\mu(I)}+
\frac{\mu(F\cap I)}{\mu(I)}-1,
\]
so that by taking \(x\in D_{\mathcal I}(E)\cap D_{\mathcal I}(F)\) and
letting \(I\Rightarrow x\) we conclude that \(x\in D_{\mathcal I}(E\cap F)\).
\end{proof}
We point out that, in fact, the assumption in Theorem \ref{thm:from_Leb_to_lift} is a priori weaker
than the statement of the Lebesgue Differentiation Theorem. Namely, we are only requiring that the
differentiation basis \(\mathcal I\) verifies the statement of the Lebesgue Density Theorem.
\begin{remark}{\rm
In general, the map \(D_{\mathcal I}\colon\Sigma\to\Sigma\) defined in \eqref{eq:from_Leb_to_lift2} is not a von Neumann lifting.
For example, consider the real line \(\R\) together with the Lebesgue measure as a measure space, and the family
\(\mathcal I\) of bounded non-trivial intervals as a differentiation basis. Then \(D_{\mathcal I}\) does not
preserve unions: \(0\) is in \(D_{\mathcal I}([-1,1])\), but neither in \(D_{\mathcal I}([-1,0])\) nor in \(D_{\mathcal I}([0,1])\).
\fr}\end{remark}
Let us now focus on those measure spaces \((\X,\Sigma,\mu)\) that are \textbf{separable}, which means
that for some (thus, any) exponent \(p\in[1,\infty)\) the \(p\)-Lebesgue space \(L^p(\mu)\) is separable.
It follows, \emph{e.g.}, from Remark \ref{rmk:good_family_approx} that \(\big(\X,\mathscr B(\X),\mu\big)\)
is separable whenever \((\X,\tau)\) is a Polish space and \(\mu\) is a \(\sigma\)-finite Borel measure on \(\X\).
Next we state a corollary of a result concerning separable measure spaces, recently obtained by Brena
and Gigli in \cite[Theorem A]{BG22}. Its proof is based on Doob's Martingale Convergence Theorem and relies only
on the Axiom of Countable Choice.
\begin{theorem}\label{thm:Doob}
Let \((\X,\Sigma,\mu)\) be a complete, separable, finite measure space. Then there exists a sequence
\(\mathcal P=(\mathcal P_k)_{k\in\N}\) of partitions \(\mathcal P_k=(P_{k,n})_{n\in I_k}\subset\Sigma\) of \(\X\), with \(I_k\subset\N\), such that:
\begin{itemize}
\item[\(\rm i)\)] \(\mu(P_{k,n})>0\) holds for every \(k\in\N\) and \(n\in I_k\).
\item[\(\rm ii)\)] Each family \(\mathcal P_{k+1}\) is a refinement of \(\mathcal P_k\), meaning that for any \(k\in\N\) and \(n\in I_{k+1}\)
there exists (a unique) index \(m\in I_k\) such that \(P_{k+1,n}\subset P_{k,m}\).
\item[\(\rm iii)\)] For any \(x\in\X\), let \(\big(n_k(x)\big)_{k\in\N}\subset\N\) be the unique sequence such that \(x\in P_k^x\coloneqq P_{k,n_k(x)}\)
for every \(k\in\N\). To any \(E\in\Sigma\) we associate the set \(\mathcal D_{\mathcal P}(E)\in\Sigma\), which is given by
\[
\mathcal D_{\mathcal P}(E)\coloneqq\bigg\{x\in\X\;\bigg|\;\lim_{k\to\infty}\frac{\mu(E\cap P_k^x)}{\mu(P_k^x)}=1\bigg\}.
\]
Then it holds that \(\mu\big(E\Delta\mathcal D_{\mathcal P}(E)\big)=0\) for every \(E\in\Sigma\).
\end{itemize}
Moreover, if \((\X,\tau)\) is a Polish space and \(\Sigma\) is the completion of \(\mathscr B(\X)\) under \(\mu\),
then the above statement holds for any sequence \(\mathcal P=(\mathcal P_k)_{k\in\N}\) of partitions
\(\mathcal P_k=(P_{k,n})_{n\in I_k}\subset\mathscr B(\X)\) of \(\X\) verifying the following conditions: each family
\(\mathcal P_{k+1}\) is a refinement of \(\mathcal P_k\) and it holds that
\begin{equation}\label{eq:hp_part_Polish}\begin{split}
&\mu(P_{k,n})>0,\quad\text{ for every }k\in\N\text{ and }n\in I_k,\\
&\lim_{k\to\infty}\sup_{n\in I_k}{\rm diam}\big(P_{k,n}\cap{\rm spt}(\mu)\big)=0,
\end{split}\end{equation}
where diameters are with respect to some complete, separable distance \(\sfd\) on \(\X\) that induces \(\tau\).
\end{theorem}

Observe that, given any \((\X,\Sigma,\mu)\) and \(\mathcal P\) as in Theorem \ref{thm:Doob}, the family of sets
\[
\mathcal I(\mathcal P)\coloneqq\big\{P_{k,n}\;\big|\;k\in\N,\,n\in I_k\big\}\subset\Sigma
\]
is a differentiation basis on \((\X,\Sigma,\mu)\) and it holds that \(D_{\mathcal I(\mathcal P)}(E)=\mathcal D_{\mathcal P}(E)\)
for every \(E\in\Sigma\). Given that \((\X,\Sigma,\mu)\) and \(\mathcal I(\mathcal P)\) verify the hypotheses of Theorem
\ref{thm:from_Leb_to_lift}, it thus follows that:
\begin{theorem}[von Neumann liftings on separable measure spaces]\label{thm:exist_lift_separable}
Let \((\X,\Sigma,\mu)\) be a complete, separable, finite measure space. Let \(\mathcal P\) be any sequence of partitions of the
space \(\X\) as in Theorem \ref{thm:Doob}. Then there exists a von Neumann lifting \(\ell_{\mathcal P}\) of \(\mu\) such that
\begin{equation}\label{eq:cond_lift_induced}
\mathcal D_{\mathcal P}(E)\subset\ell_{\mathcal P}(E)\subset\X\setminus\mathcal D_{\mathcal P}(\X\setminus E),\quad\text{ for every }E\in\Sigma.
\end{equation}
\end{theorem}

Finally, the same argument yields existence of strong liftings on an arbitrary Polish space:
\begin{corollary}[Existence of strong liftings]\label{cor:strong_lifting_on_mms}
Let \((\X,\Sigma,\mu)\) be a complete, finite measure space. Assume that \((\X,\tau)\) is a Polish space and that \(\Sigma\)
is the completion of \(\mathscr B(\X)\) under \(\mu\). Let \(\mathcal P=(\mathcal P_k)_{k\in\N}\) be any sequence of partitions
\(\mathcal P_k=(P_{k,n})_{n\in I_k}\subset\mathscr B(\X)\) of \(\X\) as in Theorem \ref{thm:Doob} and satisfying \eqref{eq:hp_part_Polish}.
Then there exists a strong lifting \(\ell_{\mathcal P}\) of \(\mu\) such that \eqref{eq:cond_lift_induced} is verified.
\end{corollary}
\begin{proof}
Theorem \ref{thm:exist_lift_separable} yields a von Neumann lifting \(\ell_{\mathcal P}\) of \(\mu\) satisfying \eqref{eq:cond_lift_induced}.
Choose a complete, separable distance \(\sfd\) on \(\X\) that induces the topology \(\tau\). Let \(U\in\tau\setminus\{\varnothing\}\) and
\(x\in U\) be fixed. Then there exists \(r>0\) such that the open \(\sfd\)-ball of centre \(x\) and radius \(r\) is contained in \(U\).
Thanks to \eqref{eq:hp_part_Polish}, we can also choose \(\bar k\in\N\) so that \({\rm diam}\big(P_k^x\cap{\rm spt}(\mu)\big)<r/2\)
for every \(k\geq\bar k\). Hence, we have that \(P_k^x\cap{\rm spt}(\mu)\subset U\) for every \(k\geq\bar k\), thus accordingly
\[
\frac{\mu(U\cap P_k^x)}{\mu(P_k^x)}=\frac{\mu(U\cap{\rm spt}(\mu)\cap P_k^x)}{\mu(P_k^x)}=\frac{\mu({\rm spt}(\mu)\cap P_k^x)}{\mu(P_k^x)}=1,
\quad\text{ for every }k\geq\bar k.
\]
This implies that \(x\in\mathcal D_{\mathcal P}(U)\subset\ell_{\mathcal P}(U)\). Thanks to the arbitrariness of \(x\in U\),
we have shown that \(U\subset\ell_{\mathcal P}(U)\). Therefore, \(\ell_{\mathcal P}\) is a strong lifting of \(\mu\), as desired.
\end{proof}
\section{Approximate continuity with respect to a differentiation basis}
\label{app:approx_cont}
In Section \ref{s:LDT}, the main Theorem \ref{thm:Lebesgue_diff} was obtained as a
corollary of Lemma \ref{lem:pre_Leb}, where we proved that `measurable maps are almost
everywhere approximately continuous with respect to a differentiation basis', in a sense
which we did not make precise. In this appendix, we expand the discussion concerning
this point. In Definition \ref{def:density_top} we recall the notion of density
topology, which is a topology induced by a von Neumann lifting. In Proposition
\ref{prop:approx_cont} we provide different characterisations of the notion of
approximate continuity mentioned above. We learnt of this kind of approach
from \cite{BliedtnerLoeb00}.
\medskip

Let us first recall the concept of density topology associated with a lower density (see \cite{LukesMalyZajicek86}):
\begin{definition}[Density topology]\label{def:density_top}
Let \((\X,\Sigma,\mu)\) be a complete, \(\sigma\)-finite measure space
and let \(\phi\) be a lower density of \(\mu\). Then we define the family
of sets \(\mathcal T_\phi\subset\Sigma\) as
\[
{\mathcal T}_\phi\coloneqq\big\{E\in\Sigma\;\big|\;E\subset\ell(E)\big\}.
\]
We say that \(\mathcal T_\phi\) is the \textbf{density topology} on \(\X\)
associated with \(\phi\).
\end{definition}

Let us verify that \(\mathcal T_\phi\) is actually a topology on \(\X\).
Trivially, one has that \(\varnothing,\X\in\mathcal T_\phi\).
Given any \(E,F\in\mathcal T_\phi\), it holds \(E\cap F\subset\phi(E)\cap\phi(F)
=\phi(E\cap F)\) and thus \(E\cap F\in\mathcal T_\phi\). In order to show that
\(\mathcal T_\phi\) is closed under arbitrary unions, fix a (possibly uncountable) family
\(\{E_i\}_{i\in I}\subset\mathcal T_\phi\). Call \(E\coloneqq\bigcup_{i\in I}E_i\) for
brevity. Note that, a priori, we do not know whether \(E\in\Sigma\). However, we can find
a countable family \(C\subset I\) such that \(\mu\big(E\Delta\bigcup_{i\in C}E_i\big)=0\),
thus accordingly the completeness of \(\mu\) and the measurability of \(\bigcup_{i\in C}E_i\)
ensure that \(E\in\Sigma\). Finally, for any \(i\in I\) we have \(E_i\subset\phi(E_i)\subset
\phi(E)\), whence it follows that \(E\subset\phi(E)\). This shows
\(\bigcup_{i\in I}E_i\in\mathcal T_\phi\).
\begin{remark}[Strong liftings and density topology]{\rm
Let \((\X,\Sigma,\mu)\) be a complete, \(\sigma\)-finite measure space
such that \({\rm spt}(\mu)=\X\). Let \(\tau\) be a Polish topology
on \(\X\) with \(\tau\subset\Sigma\). Let \(\ell\) be any strong lifting
of \(\mu\). Then it holds that \(\tau\) is contained in the density topology \(\mathcal T_\ell\).
\fr}\end{remark}
\begin{proposition}\label{prop:approx_cont}
Let \((\X,\Sigma,\mu)\) be a complete, \(\sigma\)-finite measure space and let
\((\Y,\tau)\) be a topological space. Let \(\ell\) be a von Neumann lifting of \(\mu\).
Let \(\varphi\colon\X\to\Y\) be a measurable map. Let \(x\in\hat\X\) be a given point.
Then the following conditions are equivalent:
\begin{itemize}
\item[\(\rm i)\)] For every \(U\in\mathcal N_\Y\big(\varphi(x)\big)\) there exists
\(I\in\mathcal I^\ell_x\) such that \(\varphi(z)\in U\) for \(\mu\)-a.e.\ \(z\in I\).
\item[\(\rm ii)\)] It holds that \(x\in\ell\big(\varphi^{-1}(U)\big)\) for every
\(U\in\mathcal N_\Y\big(\varphi(x)\big)\).
\item[\(\rm iii)\)] The mapping \(\varphi\colon(\X,{\mathcal T}_\ell)\to(\Y,\tau)\)
is continuous at \(x\).
\end{itemize}
When the above hold, we say that \(\varphi\) is \textbf{approximately continuous} at \(x\)
with respect to \(\mathcal I^\ell\).
\end{proposition}
\begin{proof}
\ \\
{\color{blue}\({\rm i)}\Rightarrow{\rm ii)}\).} Fix \(U\in\mathcal N_\Y\big(\varphi(x)\big)\)
and pick any \(I\in\mathcal I^\ell_x\) such that \(\varphi(z)\in U\) for
\(\mu\)-a.e.\ \(z\in I\). Then the set \(E\coloneqq I\cap\varphi^{-1}(U)\in\Sigma\)
satisfies \(\mu(E\Delta I)=0\) and thus \(x\in I=\ell(I)=\ell(E)\subset\ell\big(\varphi^{-1}(U)\big)\).\\
{\color{blue}\({\rm ii)}\Rightarrow{\rm iii)}\).} Given any neighbourhood
\(U\in\mathcal N_\Y\big(\varphi(x)\big)\), we define
\(E\coloneqq\varphi^{-1}(U)\cap\ell\big(\varphi^{-1}(U)\big)\in\Sigma\).
Observe that \(\ell(E)=\ell\big(\varphi^{-1}(U)\big)\supset E\), which shows that
\(E\in{\mathcal T}_\ell\). Notice that, by definition, we have that \(\varphi(z)\in U\)
for every \(z\in E\). Given that \(x\in\varphi^{-1}(U)\subset E\) (which says that \(E\)
is a neighbourhood of \(x\) in the topology \({\mathcal T}_\ell\)) and thanks to the
arbitrariness of \(U\), we can conclude that the mapping \(\varphi\) is
\({\mathcal T}_\ell\)-continuous at \(x\).\\
{\color{blue}\({\rm iii)}\Rightarrow{\rm i)}\).} Let
\(U\in\mathcal N_\Y\big(\varphi(x)\big)\) be given. Then the \({\mathcal T}_\ell\)-continuity
assumption on \(\varphi\) at \(x\) yields the existence of a set
\(E\in{\mathcal T}_\ell\) such that \(x\in E\) and \(\varphi(E)\subset U\). Now fix
an arbitrary set \(I'\in\mathcal I^\ell_x\) and define \(I\coloneqq I'\cap\ell(E)\).
Given that \(x\in E\subset\ell(E)\), we know that \(I\in\mathcal I^\ell_x\). Moreover, the
fact that \(\varphi(z)\in U\) for every \(z\in E\) and \(\mu\big(E\Delta\ell(E)\big)=0\)
imply that \(\varphi(z)\in U\) for \(\mu\)-a.e.\ \(z\in I\).
\end{proof}
Observe that Lemma \ref{lem:pre_Leb} can be rephrased as follows.
Every essentially separably valued, measurable map \(\varphi\colon\X\to\Y\), from
a complete, \(\sigma\)-finite measure space \((\X,\Sigma,\mu)\) with lifting \(\ell\)
to a metric space \((\Y,\sfd)\), is approximately continuous at \(\mu\)-a.e.\ point
of \(\X\) with respect to \(\mathcal I^\ell\).
\section{More on liftings of Lebesgue--Bochner spaces}\label{app:lift_Leb-Boch}
As we already pointed out, the theory of measurable Banach bundles covers the case
of Lebesgue--Bochner spaces, which have been thoroughly studied in the literature,
also from the point of view of lifting theory. We will report here some results
in this direction, taken from the paper \cite{Gutman93}, which show that the lifting
of sections we obtained in Theorem \ref{thm:lift_sects} is the best one can hope for.
More precisely, it will be clear from Theorem \ref{thm:lift_Leb-Bochn} that the inequality
in \eqref{eq:lift_sects_claim4} cannot, in general, be improved to an equality.
\medskip

Let \((\X,\Sigma,\mu)\) be a \(\sigma\)-finite measure space.
Then \(\mu\) is said to be \textbf{atomic} provided it has
at least one atom, while it is said to be \textbf{non-atomic}
when no atom of \(\mu\) exists. Moreover, we say that \(\mu\)
is \textbf{purely atomic} provided \(\mu|_E\) is atomic for
every \(E\in\Sigma\) such that \(\mu(E)>0\).
\begin{lemma}\label{lem:partition_non_pur_atom}
Let \((\X,\Sigma,\mu)\) be a complete, \(\sigma\)-finite measure space.
Suppose that \(\mu\) is not purely atomic. Let \(\ell\) be a von Neumann
lifting of \(\mu\). Then there exist a point \(\bar x\in\X\)
and a family \(\{E_n\}_{n\in\N}\subset\Sigma\) of pairwise disjoint
sets such that \(0<\mu(E_n)<\infty\) for every \(n\in\N\) and
\begin{equation}\label{eq:partition_non_pur_atom}
\mu\bigg(\X\setminus\bigcup_{n\in\N}E_n\bigg)=0,\qquad
\bar x\in\X\setminus\bigcup_{n\in\N}\ell(E_n).
\end{equation}
\end{lemma}
\begin{proof}
Fix a set \(F\in\Sigma\) with \(0<\mu(F)<\infty\) such that \(\mu|_F\) is non-atomic,
whose existence stems from the assumption that \(\mu\) is not purely atomic.
Being \(\mu\) a \(\sigma\)-finite measure, we can find an at most
countable partition \((G_i)_{i\in I}\subset\Sigma\) of
\(\X\setminus F\) into sets having finite, positive \(\mu\)-measure.
Now fix any point \(\bar x\in\ell(F)\) and notice that \(\bar x\notin
\bigcup_{i\in I}\ell(G_i)\). Since \(\mu|_F\) is non-atomic,
we know that it has full range (it follows, \emph{e.g.}, from Lyapunov
Convexity Theorem, see \cite{DiestelUhl77}), meaning that for any
\(\lambda\in\big(0,\mu(F)\big)\) there exists \(F'\in\Sigma\)
such that \(F'\subset F\) and \(\mu(F')=\lambda\). We can thus
proceed with the construction via the following recursive argument.
First, take a measurable set \(F_1\subset F\) such that
\(\mu(F_1)=\mu(F)/2\). Given that
\(\bar x\in\ell(F)=\ell(F_1)\cup\ell(F\setminus F_1)\),
we may assume (up to replacing \(F_1\) by \(F\setminus F_1\))
that \(\bar x\notin\ell(F_1)\). Next, we can argue in the very
same way starting from \(F\setminus F_1\), thus getting a set
\(F_2\subset F\setminus F_1\) such that \(\mu(F_2)=\mu(F)/4\)
and \(\bar x\notin\ell(F_2)\), and so on. All in all, with
this procedure we obtain a sequence \((F_j)_{j\in\N}\) of pairwise disjoint measurable
subsets of \(F\) such that \(\mu(F_j)=\mu(F)/2^j\) holds for every \(j\in\N\)
and \(\bar x\notin\bigcup_{j\in\N}\ell(F_j)\). Hence, the family
\(\{E_n\}_{n\in\N}\coloneqq\{G_i\}_{i\in I}\cup\{F_j\}_{j\in\N}\)
satisfies \eqref{eq:partition_non_pur_atom}.
\end{proof}
The previous lemma is useful in order to achieve the next result,
taken from \cite[Section 4.1]{Gutman93}; we report its proof for the reader's usefulness.
It is convenient, for the moment, to consider only measures that are not purely atomic.
The case of purely atomic measures is easier to handle and will be discussed in
the last part of this appendix (see Remark \ref{rmk:lift_purely_atom}).
\begin{theorem}[Lifting of Lebesgue--Bochner spaces]
\label{thm:lift_Leb-Bochn}
Let \((\X,\Sigma,\mu)\) be a complete, \(\sigma\)-finite measure space.
Suppose that \(\mu\) is not purely atomic. Let \(\B\)
be a Banach space and \(\ell\) a von Neumann lifting of \(\mu\).
Then the following two conditions are equivalent:
\begin{itemize}
\item[\(\rm i)\)] There exists an operator
\(\ell\colon L^\infty(\mu;\B)\to\mathcal L^\infty(\mu;\B)\)
satisfying
\[\begin{split}
\ell(\1_\X^\mu\bar v)(x)=\bar v,&\quad\text{ for every }
\bar v\in\B\text{ and }x\in\X,\\
\ell(v)(x)=\hat v(x),&\quad\text{ for every }v\in L^\infty(\mu;\B)
\text{ and }\mu\text{-a.e.\ }x\in\X,\\
\ell(v+w)(x)=\ell(v)(x)+\ell(w)(x),&\quad\text{ for every }
v,w\in L^\infty(\mu;\B)\text{ and }x\in\X,\\
\ell(f\cdot v)(x)=\ell(f)(x)\,\ell(v)(x),&\quad\text{ for every }
v\in L^\infty(\mu;\B)\text{, }f\in L^\infty(\mu)\text{, and }x\in\X,\\
\big\|\ell(v)(x)\big\|_\B=\ell\big(\|v(\cdot)\|_\B\big)(x),&
\quad\text{ for every }v\in L^\infty(\mu;\B)\text{ and }x\in\X,
\end{split}\]
where \(\hat\#\colon L^\infty(\mu;\B)\to\mathcal L^\infty(\mu;\B)\)
is the precise representative given by Theorem \ref{thm:precise_repr}.
\item[\(\rm ii)\)] The Banach space \(\B\) is finite-dimensional.
\end{itemize}
\end{theorem}
\begin{proof}
The implication \({\rm ii)}\Rightarrow{\rm i)}\) stems from
Proposition \ref{prop:improved_lift_sects} and the observation
that \(L^\infty(\mu;\B)\) and \({\rm ECV}(\mu;\B)\) coincide
when \(\B\) is finite-dimensional. Conversely, in order to prove the
validity of the converse implication \({\rm i)}\Rightarrow{\rm ii)}\),
we argue by contradiction: suppose that \(\rm i)\) holds and
\(\B\) is infinite-dimensional. Then we can find a sequence
\((\bar v_n)_{n\in\N}\subset\B\) such that \(\|\bar v_n\|_\B=1\)
for every \(n\in\N\) and \(\|\bar v_n-\bar v_m\|_\B\geq 1/2\) for
every \(n,m\in\N\) with \(n\neq m\). Recalling Lemma
\ref{lem:partition_non_pur_atom}, we can find \(\bar x\in\X\)
and a sequence of pairwise disjoint sets
\((E_n)_{n\in\N}\subset\Sigma\) with \(\bar x\notin\bigcup_n E_n\)
such that \(\ell(E_n)=E_n\) for all \(n\in\N\) and
\(\mu\big(\X\setminus\bigcup_n E_n\big)=0\). Define
\(v\coloneqq\sum_{n\in\N}\1_{E_n}^\mu\bar v_n\in L^\infty(\mu;\B)\).
Calling \(\bar v\coloneqq\ell(v)(\bar x)\in\B\), we claim that
\(\bar v=\bar v_{\bar n}\) for some \(\bar n\in\N\). If this were
not the case, then there would exist \(\varepsilon>0\) such that
\(\|\bar v-\bar v_n\|_\B\geq\varepsilon\) for every \(n\in\N\);
here, we are using the fact that \(\{\bar v_n\,:\,n\in\N\}\) is
a discrete set. This is not possible, since it would imply that
\[\begin{split}
0&=\big\|\ell(v)(\bar x)-\bar v\big\|_\B=
\big\|\ell(v)(\bar x)-\ell(\1_\X^\mu\bar v)(\bar x)\big\|_\B=
\big\|\ell(v-\1_\X^\mu\bar v)(\bar x)\big\|_\B=
\ell\big(\|v(\cdot)-\bar v\|_\B\big)(\bar x)\\
&=\ell\bigg(\sum_{n\in\N}\|\bar v_n-\bar v\|_\B\1_{E_n}^\mu\bigg)(\bar x)
\geq\ell\big(\varepsilon\1_{\bigcup_n E_n}^\mu\big)(\bar x)
=\varepsilon\,\ell(\1_\X^\mu)(\bar x)=\varepsilon.
\end{split}\]
Let \(\bar n\in\N\) be such that \(\bar v=\bar v_{\bar n}\).
Observe that \(\bar x\in\X\setminus E_{\bar n}=\X\setminus
\ell(E_{\bar n})=\ell(\X\setminus E_{\bar n})\) and thus
\[\begin{split}
0&=\big\|\ell(v)(\bar x)-\bar v\big\|_\B=
\big\|\1_{\ell(\X\setminus E_{\bar n})}(\bar x)\,\ell(v)(\bar x)-
\bar v_{\bar n}\big\|_\B=\big\|\ell\big(\1_{\X\setminus E_{\bar n}}^\mu
\cdot v-\1_{\X\setminus E_{\bar n}}^\mu\bar v_{\bar n}\big)
(\bar x)\big\|_\B\\
&=\ell\big(\1_{\X\setminus E_{\bar n}}^\mu\|v(\cdot)-
\bar v_{\bar n}\|_\B\big)(\bar x)=\ell\bigg(\sum_{n\neq\bar n}
\|\bar v_n-\bar v_{\bar n}\|_\B\1_{E_n}^\mu\bigg)(\bar x)\geq
\ell\bigg(\frac{1}{2}\1_{\X\setminus E_{\bar n}}^\mu\bigg)(\bar x)
=\frac{1}{2},
\end{split}\]
which leads to a contradiction. Therefore,
we have proven that \(\B\) is finite-dimensional.
\end{proof}
\begin{remark}\label{rmk:lift_LB_larger}{\rm
By inspecting the proof of Theorem \ref{thm:lift_Leb-Bochn}, it is evident that a stronger
claim holds: if \(\B\) is infinite-dimensional, then it is not possible to find a
lifting \(L^\infty(\mu;\B)\to\mathcal L^\infty(\mu;\mathbb V)\)
for any Banach space \(\mathbb V\) where \(\B\) can be embedded linearly and isometrically.
\fr}\end{remark}
We point out that all measurable maps are \emph{almost everywhere constant on atoms}: given a metric space \((\Y,\sfd)\),
a strongly measurable map \(\varphi\colon\X\to\Y\), and an atom \(A\) of \(\mu\), it holds that
\begin{equation}\label{eq:const_on_atoms}
\mu\big(\big\{x\in A\;\big|\;\varphi(x)\neq y\big\}\big)=0,\quad\text{ for a (unique) element }y\in\Y.
\end{equation}
To prove it, fix \(N\in\Sigma\) and a separable Borel set \(\tilde\Y\subset\Y\) such that \(\mu(N)=0\) and \(\varphi(\X\setminus N)\subset\tilde\Y\).
Thanks to Lindel\"{o}f Lemma, for any \(n\in\N\) we can find a Borel partition \((E_n^k)_{k\in\N}\) of \(\tilde\Y\) such
that the diameter of each set \(E_n^k\) does not exceed \(1/n\). Without loss of generality, we can also assume that
\((E_{n+1}^k)_k\) is a refinement of \((E_n^k)_k\), \emph{i.e.}, for any \(k\in\N\) there exists \(\ell\in\N\) for
which \(E_{n+1}^k\subset E_n^\ell\). Since \(A\) is an atom of \(\mu\), we can find \((k_n)_{n\in\N}\subset\N\) such
that \(\mu\big(A\setminus\varphi^{-1}(E_n)\big)=0\), where \(E_n\coloneqq E_n^{k_n}\). Notice that \(E\coloneqq\bigcap_{n\in\N}E_n\)
has null diameter, thus in particular it is either empty or a singleton. Given that \(\mu\big(A\setminus\varphi^{-1}(E)\big)=0\),
we deduce that \(E\neq\varnothing\), so that \(E=\{y\}\) for some \(y\in\Y\) and thus \(\varphi(x)=y\) for \(\mu\)-a.e.\ \(x\in A\).
Consequently, the claim \eqref{eq:const_on_atoms} is proven.
\begin{remark}[Lifting of purely atomic measures]\label{rmk:lift_purely_atom}{\rm
Let \((\X,\Sigma,\mu)\) be a \(\sigma\)-finite measure space, with
\(\mu\) purely atomic. As proven, \emph{e.g.}, in
\cite[Theorem 2.2]{Johnson70}, one can find an at most countable
partition \((A_i)_{i\in I}\) of \(\X\) made of atoms of \(\mu\).
Observe that for any \(E\in\Sigma\) there exists a unique index
set \(I_E\subset I\) such that
\(\mu\big(E\Delta\bigcup_{i\in I_E}A_i\big)=0\).
Then we define the map \(\ell\colon\Sigma\to\Sigma\) as
\[
\ell(E)\coloneqq\bigcup_{i\in I_E}A_i,\quad\text{ for every }E\in\Sigma.
\]
It is straightforward to verify that \(\ell\) is a von Neumann lifting
of the measure \(\mu\).

Now fix an arbitrary Banach space \(\B\). Recalling
\eqref{eq:const_on_atoms}, we see that for any
\(v\in L^\infty(\mu;\B)\) there is a unique
\((\bar v_i)_{i\in I}\subset\B\) such that
\(\sup_{i\in I}\|\bar v_i\|_\B<+\infty\) and
\(v=\sum_{i\in I}\1_{A_i}^\mu\bar v_i\), thus we set
\[
\ell(v)\coloneqq\sum_{i\in I}\1_{A_i}\bar v_i\in\mathcal L^\infty(\mu;\B).
\]
It is then easy to check that the resulting operator
\(\ell\colon L^\infty(\mu;\B)\to\mathcal L^\infty(\mu;\B)\)
verifies all the properties listed in item i) of Theorem \ref{thm:lift_Leb-Bochn}.
\fr}\end{remark}
\sectionlinetwo{Black}{88} 
\medskip

\noindent\textbf{Acknowledgements.}
The authors thank Camillo Brena for the useful discussions on the topics of this paper.
The first named author acknowledges the support by the project 2017TEXA3H ``Gradient flows, Optimal Transport and Metric Measure Structures'',
funded by the Italian Ministry of Research and University. The second named author acknowledges the support by the Balzan project led by
Luigi Ambrosio. Both authors acknowledge the support by the Academy of Finland, grant no.\ 314789.
\def\cprime{$'$} \def\cprime{$'$}

\end{document}